\documentclass[12pt]{amsart}

\usepackage{a4wide,amssymb,amsthm,amscd,amsmath,mathtools,stmaryrd,url,fullpage,enumerate,color,mathrsfs}
\usepackage{array}
\usepackage{ulem}
\usepackage{float}
\usepackage{soul}
\usepackage{etoolbox,expl3,pgfkeys,xparse,xstring,tikz}
\usepackage{dynkin-diagrams}
\usepackage{comment}

\usepackage{dsfont} 

\usepackage{listings}
\usepackage{xcolor} 

\newtheorem{theorem}{Theorem}[section]
\newtheorem{lemma}[theorem]{Lemma}
\newtheorem{prop}[theorem]{Proposition}
\newtheorem{proposition}[theorem]{Proposition}

\newtheorem{cor}[theorem]{Corollary}
\newtheorem{corollary}[theorem]{Corollary}
\newtheorem{definition}[theorem]{Definition}

\theoremstyle{remark}
\newtheorem{remark}[theorem]{Remark}

 \usepackage{xcolor}
 \usepackage[pdftex]{hyperref}
 \definecolor{couleur_cite}{rgb}{0.05,.4,0.05}
 \definecolor{couleur_link}{rgb}{0.05,0.05,0.4}
 \hypersetup{
 bookmarksopen,bookmarksnumbered,
 colorlinks,
 linkcolor=couleur_link,citecolor=couleur_cite,
 }

\usepackage[
backend=bibtex,
style=alphabetic,
maxbibnames=99,
doi=false,
isbn=false,
url=false,
minalphanames=3,
maxalphanames=3
]{biblatex}

\newcommand{\R}{\mathbb R}
\newcommand{\C}{\mathbb C}

\newcommand{\Z}{\mathbb Z}
\newcommand{\Q}{\mathbb Q}

\newcommand{\p}{\mathfrak p}
\newcommand{\g}{\mathfrak g}

\newcommand{\ga}{\mathfrak a}

\newcommand{\SL}{\text{SL} }

\newcommand{\Hom}{\text{Hom}}

\newcommand{\be}{\begin{equation}}
\newcommand{\ee}{\end{equation}}
\newcommand{\bes}{\begin{equation*}}
\newcommand{\ees}{\end{equation*}}

\newcommand{\ba}{\begin{eqnarray}}
\newcommand{\ea}{\end{eqnarray}}
\newcommand{\bas}{\begin{eqnarray*}}
\newcommand{\eas}{\end{eqnarray*}}

\def\XXint#1#2#3{{\setbox0=\hbox{$#1{#2#3}{\int}$}
     \vcenter{\hbox{$#2#3$}}\kern-.5\wd0}}

  \title{Quantum ergodicity in the Benjamini--Schramm limit for locally symmetric spaces}
\date{\today}
  
\author{Farrell Brumley}
\address{Sorbonne Universit\'e}
\email{brumley@imj-prg.fr}

\author{Simon Marshall}
\address{University of Melbourne}
\email{simon.marshall@unimelb.edu.au}

\author{Jasmin Matz}
\address{University of Copenhagen}
\email{matz@math.ku.dk}

\author{Carsten Peterson}
\address{Sorbonne Universit\'e}
\email{peterson@imj-prg.fr}

\thanks{F.B. is supported by the Institut Universitaire de France and ANR-FNS Grant  ANR-24-CE93-0016. S.M. was supported by National Science Foundation Grant DMS-1902173. J.M. is supported by the Carlsberg Foundation grant no. CF21-0374. C.P. was supported by the Deutsche Forschungsgemeinschaft (DFG, German Research Foundation) Grant SFB-TRR 358/1 2023 - 491392403 and received funding from the European Union’s Horizon 2020 research and innovation programme under the Marie Sk\l{}odowska-Curie grant agreement No 101034255 and from the National Science Foundation Grant DMS-2503324.}

\begin{document}
\begin{abstract}
We prove that for almost all symmetric spaces $X$ and for any sequence of compact locally symmetric spaces $Y_n$ which is uniformly discrete, has a uniform spectral gap, and converges in the sense of Benjamini--Schramm to $X$, the joint eigenfunctions of all invariant differential operators on $Y_n$ delocalize on average when their spectral parameters are taken to lie in a fixed spectral window.
\end{abstract}

\maketitle

\setcounter{tocdepth}{1}
\tableofcontents

\section{Introduction}

The celebrated Quantum Ergodicity theorem of Snirelman \cite{Snirelman}, Zelditch \cite{Zelditch_87}, and Colin de Verdi\`ere \cite{Colin_de_Verdiere} states that, on a closed Riemannian manifold $Y$ whose geodesic flow is ergodic, the $L^2$-mass of almost every Laplacian eigenfunction equidistributes in the limit of high frequency. This is an early example of the transference principle in semiclassical analysis, whereby high-energy Laplacian eigenfunctions inherit the dynamical properties of the underlying Hamiltonian system, in this case given by the geodesic flow.

More precisely, let $\{\psi_j\}$ be an orthonormal basis of $L^2(Y)$ consisting of eigenfunctions of the Laplacian. We write $\Delta \psi_j = \mu_j^2 \psi_j$, where $\mu_j \ge 0$ is the frequency, and $N(M) = \#\{j : \mu_j \leq M\}$. Quantum Ergodicity affirms that, for any $a \in C(Y)$,
\[
\lim_{M\rightarrow\infty}    \frac{1}{N(M)} \sum_{j: \mu_j \leq M} \Big| \int_{Y} a |\psi_j|^2 d \textnormal{vol} - \frac{1}{\textnormal{vol}(Y)} \int_{Y} a d \textnormal{vol} \Big|^2 = 0.
\]
In fact, the aforementioned authors proved a stronger version of the above statement in which the multiplication operator $a$ is replaced with more general pseudodifferential operators.

An analog of Snirelman's theorem can be posed in the setting of locally symmetric spaces of non-compact type, which form a subclass of Riemannian manifolds of central interest in automorphic forms and harmonic analysis. To do so effectively, one must take into account the extra symmetries enjoyed by these spaces. Indeed, Snirelman's theorem does not always directly apply, since the geodesic flow on the cosphere bundle of a locally symmetric space $Y$ is ergodic only in rank 1. Nevertheless, the \textit{Weyl chamber flow} can serve as a suitable substitute: if $r$ is the rank of $Y$, the Weyl chamber flow is an $\R^r$-action on the Weyl chamber bundle, whose orbits descend to immersed maximal flat subspaces of $Y$, the higher rank analog of geodesics. This action is ergodic with respect to the uniform measure, and its quantization yields a rank $r$ algebra of commuting differential operators containing the Laplacian. One can then ask for the quantum ergodic properties of the joint eigenfunctions for this algebra as the multispectra goes to infinity, a point of view first advanced in \cite{Silberman_Venkatesh}.

In this paper we shall investigate quantum ergodicity for locally symmetric spaces $Y$ under a different limiting procedure to that described above. Following the breakthrough results of \cite{Anantharaman_Le_Masson, Le_Masson_Sahlsten_17}, rather than fixing a single space $Y$ and varying the multispectra, we shall prove a version of quantum ergodicity for a fixed spectral window and a sequence of locally symmetric spaces $Y_n=\Gamma_n\backslash X$ which converge, in the sense of Benjamini--Schramm, to their common universal cover $X$. In so doing we will correct a significant error in the earlier work of the first and third authors \cite{Brumley_Matz_23}, in which $X$ was taken to have isometry group $\SL_n(\R)$, and greatly expand the scope of that paper. The precise result is stated in Theorem \ref{main-theorem} below. 

The proof of our main theorem requires the introduction of several new techniques and auxiliary estimates, both in harmonic analysis (bounds on spherical functions) and the geometry of higher rank symmetric spaces (bounds on intersection translates), which we believe to be of general interest. These results, and their role in the proof of Theorem \ref{main-theorem}, are discussed in detail in Section \ref{sec:technical-thms}.

\subsection{Review of literature: the rank one case}
The framework for studying quantum ergodicity in the Benjamini--Schramm limit originates in the work of Anantharaman--Le Masson \cite{Anantharaman_Le_Masson} in the context of regular graphs. They considered sequences of $(q+1)$-regular graphs for which the adjacency operator (the discrete analogue of the Laplacian) has a uniform spectral gap and for which the number of short loops is small compared to the number of vertices; the latter condition is equivalent to the Benjamini--Schramm convergence of the sequence of graphs to the $(q+1)$-regular tree. They proved a form of quantum ergodicity for eigenfunctions of the adjacency operator on such graphs, with eigenvalue lying in some fixed interval $I$. More precisely, $I$ was taken to be a subinterval of $[-2 \sqrt{q}, 2 \sqrt{q}]$, which is the spectrum of the adjacency operator acting on the $(q+1)$-regular tree. Later, Anantharaman--Sabri \cite{Anantharaman_Sabri} extended this result to a wide class of large finite graphs, not necessarily regular, in the presence of a potential.

Analogous results were proven for hyperbolic surfaces by Le Masson--Sahlsten \cite{Le_Masson_Sahlsten_17}, building on work of Brooks--Le Masson--Lindenstrauss \cite{Brooks_Le_Masson_Lindenstrauss}, the latter of which gave a new proof of the result of Anantharaman--Le Masson \cite{Anantharaman_Le_Masson}. Such results were extended to all rank one locally symmetric spaces by Abert--Bergeron--Le Masson \cite{Abert_Bergeron_Le_Masson_22}.

\subsection{Our main result}\label{sec:main-result}
Let us recall some notation related to locally symmetric spaces, which will be necessary to state our main result.

Let $G$ be a connected non-compact semisimple real Lie group with finite center and $K<G$ a maximal compact subgroup. Let $\mathfrak{g}$ and $\mathfrak{k}$ be their respective Lie algebras. Then $K$ induces a Cartan involution with corresponding eigenspace decomposition $\mathfrak{g} = \mathfrak{p} \oplus \mathfrak{k}$. We endow $X=G/K$ with the $G$-invariant metric coming from the Killing form, making $X$ a Riemannian symmetric space. All groups isogenous to $G$ yield the same symmetric space. Thus we may assume without loss of generality that $G$ is a product of non-compact centerless connected simple real Lie groups.

Let $\Gamma < G$ be an irreducible lattice, which we shall assume to be uniform throughout this paper. Then $Y = \Gamma \backslash X$ is a compact locally symmetric space. We endow $Y$ with the measure $dx$ induced from the Riemannian volume form on $X$. Let $D_{G}(X)$ be the ring of $G$-invariant differential operators on $X$. The action of $D_{G}(X)$ descends to functions on $Y$, and the joint eigenfunctions of $D_G(X)$ in $L^2(Y)$ are called \textit{Maass forms}. A Maass form determines an eigencharacter of $D_G(X)$, which can then be identified with a spectral parameter $\lambda\in\ga_\C^*/W$ by the Harish-Chandra isomorphism. Here $\mathfrak{a}$ is a maximal abelian subspace of $\mathfrak{p}$, $\ga_\C$ is its complexification, and $W$ is the Weyl group of $G$. A Maass form is said to be \textit{tempered} when its spectral parameter lies in $\ga^*/W$, the latter being identified with the $D_G(X)$-spectrum on $L^2(X)$. For example, when $X=\mathbb{H}^n$ is hyperbolic $n$-space and $Y=\Gamma\backslash\mathbb{H}^n$ is a compact hyperbolic manifold, a Maass form $\psi$ is just a Laplacian eigenfunction $\Delta\psi=\mu^2\psi$, and if we write the eigenvalue as $\mu^2=\rho^2+\lambda^2$, with $\rho=(n-1)/2$, then the tempered case $\lambda\in \R$ corresponds to $\mu^2\geq \rho^2$.

We shall be interested in the $L^2$-mass distribution of Maass forms, when their spectral parameters are confined to a compact subset of $\ga^*/W$ and the lattice $\Gamma<G$ is allowed to vary. A sequence of locally symmetric spaces $Y_n=\Gamma_n\backslash X$ as above is said to \textit{Benjamini--Schramm} converge to $X$ if asymptotically almost all points in $Y_n$ have arbitrarily large injectivity radius. Moreover, if the $\Gamma_n$ are torsion free, we say that the sequence $\Gamma_n < G$ is \textit{uniformly discrete} if there is a universal lower bound on the global injectivity radii of $Y_n$.

With the above notation, our main theorem is the following.

\begin{theorem}\label{main-theorem}
Let $G$ be a product of non-compact connected centerless simple real Lie groups, $K$ be a maximal compact subgroup, and $X = G/K$ be the associated symmetric space. Let $\Gamma_n<G$ be a sequence of torsion free, cocompact, uniformly discrete, irreducible lattices. Suppose  $Y_n = \Gamma_n \backslash X$ Benjamini--Schramm converges to $X$ as $n \to \infty$. Fix a simple factor $G_1$ of $G$ and suppose that 
\begin{enumerate}
\item\label{thm:root-system} its reduced root subsystem is of type $A_n, B_n, C_n, D_n$ or $E_7$,
\item\label{thm:spec-gap} the action of $G_1$ on $L^2_0(\Gamma_n\backslash G)$, the orthocomplement to the constant functions, has a uniform spectral gap. 
\end{enumerate}

There exists a finite $W$-stable set of hyperplanes $\{ P_i\}$ in $\ga^*$ such that, for any compact $W$-invariant subset $\Omega \subset \mathfrak{a}^*\smallsetminus\cup_i P_i$ with non-empty interior, the following holds. Let $\{\psi_j^{(n)}\}$ be an orthonormal basis of $L^2(Y_n)$ consisting of Maass forms with associated spectral parameters $\lambda_j^{(n)}$, and write $N(\Omega, \Gamma_n) = \#\{j : \lambda_j^{(n)} \in \Omega \}$. Then for any uniformly bounded sequence $a_n\in L^\infty(Y_n)$ we have
    \begin{gather*}
       \lim_{n\rightarrow\infty} \frac{1}{N(\Omega, \Gamma_n)} \sum_{j: \lambda_j^{(n)} \in \Omega} \Big| \int_{Y_n} a_n(x) |\psi_j^{(n)}(x)|^2 dx - \frac{1}{{\rm vol}(Y_n)}\int_{Y_n} a_n(x) dx \Big|^2 = 0.
    \end{gather*}
\end{theorem}

In particular, the theorem holds for a non-compact simple real Lie group $G$ of rank at least 2 satisfying \eqref{thm:root-system} and any uniformly discrete sequence of cocompact torsion free lattices $\Gamma_n<G$  such that ${\rm vol}(Y_n)\rightarrow\infty$. Indeed, in such a setting the uniform spectral gap assumption \eqref{thm:spec-gap} is automatic by Property (T), and the Benjamini--Schramm convergence of $Y_n=\Gamma_n\backslash X$ to $X$ is automatic by \cite[Theorem 1.5]{7_samurai}.

The exceptional hyperplanes $\cup_i P_i$ in the statement of Theorem \ref{main-theorem} can be taken to depend only on the reduced root system of $G_1$. In particular, when the root subsystem of reduced roots for $G_1$ is of type $A_n, B_n$, or $C_n$, Theorem \ref{main-theorem} holds with $\Omega$ a compact subset of the regular parameters $\ga^*_{\rm reg}$. See Section \ref{sec:spectral-estimate}, and, in particular, Propositions \ref{prop:local-lower-bd} and \ref{prop:Weyl-diff}, for a more complete description of the $\cup_i P_i$.

The reader will note that the types $E_6, E_8, F_4$ and $G_2$ are not considered in Theorem \ref{main-theorem}. Indeed, we eventually reduce the proof to a combinatorial property of the root system of the simple factor $G_1$ which fails to hold in these types (see Section \ref{sec:big-subsystems}). There are precisely eight simple groups/irreducible symmetric spaces excluded from Theorem \ref{main-theorem}. In the notation of  \cite[p. 532]{Helgason_01} these are the groups
\[
E \ I,\; E \ II,\; E\, III,\; E \ VI,\; E \ VIII,\; E \ IX,\; F \ I,\textrm{ and } G.
\]
Despite the failure of our techniques to treat such spaces, we believe Theorem \ref{main-theorem} should remain valid for them.

\subsection{Relation with the preceding work of Brumley--Matz}\label{hi-rank-intro}

The first investigation of quantum ergodicity in the Benjamini--Schramm limit in higher rank can be found in the work of the first and third authors \cite{Brumley_Matz_23} who focused on locally symmetric spaces associated to $\textnormal{SL}_n(\mathbb{R})$. At a certain step in the proof, which we refer to as the \textit{geometric bound}, they must bound the volume of a certain set in the symmetric space. A mistake in the geometric bound was found by the fourth author as part of his thesis work \cite{Peterson_thesis}. This mistake resulted in a gap in the proof of the main theorem of \cite{Brumley_Matz_23}. 

By modifying the techniques of \cite{Brumley_Matz_23}, particularly those related to the geometric bound, the fourth author proved quantum ergodicity in the Benjamini--Schramm limit for the group $\textnormal{PGL}_3$ over a non-archimedean local field, in which case the role of the symmetric space is replaced by that of the Bruhat-Tits building, and the invariant differential operators are replaced by the spherical Hecke algbera \cite{Peterson_23}. In Section \ref{sub_sec_int_vol} we shall say more about the mistake in the geometric bound of \cite{Brumley_Matz_23}, the method of ``fixing'' it for non-archimedean $\textnormal{PGL}_3$ in \cite{Peterson_23}, and how these ideas play a role in the present paper.

\subsection{Acknowledgements}

We would like to thank Jean-Philippe Anker, particularly for his contributions to Section \ref{sub_sec_anker} and \ref{subsec_eks}. F. Brumley and J. Matz would like to thank their co-authors for their generosity in sharing their ideas to repair and extend their previous work on this topic. 

\section{Sketch of proof and main technical theorems}\label{sec:technical-thms}

The overall strategy of the proof of Theorem \ref{main-theorem} follows the argument of Le Masson--Sahlsten \cite{Le_Masson_Sahlsten_17}, which itself derives from \cite{Brooks_Le_Masson_Lindenstrauss}. The subsequent works of \cite{Abert_Bergeron_Le_Masson_22} and \cite{Brumley_Matz_23, Peterson_thesis} all followed the same basic structure, with additional difficulties depending on the given context : higher dimensional hyperbolic manifolds, and higher rank real and $p$-adic symmetric spaces, respectively. The main purpose of the present discussion is to recall the various steps of the argument which are common to the above works, point to the specific sections in this paper where they are executed, and describe in detail the geometric and analytic challenges special to our higher rank setting. 

Our two main auxiliary results which surmount these difficulties are stated in Sections \ref{sub_sec_int_vol} and \ref{intro-sph-bds} below. Of these, the primary result is stated in Theorem \ref{intersection-vol}; it provides sharp (up to logarithmic powers) upper bounds on intersection volumes of spherical shells in the globally symmetric space $X$, and addresses the issue raised in Section \ref{hi-rank-intro}. We prove this geometric bound using harmonic analysis. For this, we prove new uniform bounds on spherical functions in  Theorem \ref{sph-fn-bd}, which provides the central ingredient to our analytic approach.

The reader will notice that Theorem \ref{main-theorem} is stated for $G$ semisimple, while Theorem \ref{intersection-vol} takes $G$ to be simple. As we describe below, the deduction of the former from the latter result takes place in Section \ref{sec:geom-reduction}, specifically in Proposition \ref{prop:main-term-bd}. To simplify the exposition, we shall assume throughout Section \ref{sec:technical-thms} that $G$ is simple.

\subsection{Spectral and geometric reduction steps}
The basic strategy to the proof of Theorem \ref{main-theorem} is to introduce an averaging operator over expanding bi-$K$-invariant sets $S_{tH_0}\subset G$ to convert the distributional properties of the $L^2$-mass of Maass forms into the mean ergodic properties of the $S_{tH_0}$ (or their intersection translates).

\subsection*{Reduction to mean-zero observables}
We begin by observing that it is enough to prove Theorem \ref{main-theorem} for the class of observables which are orthogonal to constants. Indeed, if $a_n$ is as in the theorem, then so is $a_n-\frac{1}{{\rm vol}(Y_n)}\int_{Y_n} a_n$, so we may assume, without loss of generality, that $\int_{Y_n}a_n=0$. In this case the proof of Theorem \ref{main-theorem} reduces to showing
\begin{equation}\label{eq:main-estimate}
\frac{1}{N(\Omega, \Gamma_n)}\sum_{\lambda_j^{(n)}\in \Omega}\big|\langle a_n\psi_j^{(n)},\psi_j^{(n)}\rangle\big|^2\longrightarrow 0
\end{equation}
as $n\rightarrow\infty$.

\subsection*{Spectral properties of the propagator (Section \ref{sec:spectral-estimate})}

In Section \ref{sec:spectral-estimate}, we introduce the aforementioned expanding bi-$K$-invariant sets $S_{tH_0}$ in $G$. More precisely, given a non-zero $H_0 \in \overline{\mathfrak{a}}_+$ and real parameters $t,\epsilon_0>0$, define the \textit{spherical shell directed by $H_0$} by
\begin{gather*}
S_{tH_0} := K \textnormal{exp}(B_{\mathfrak{a}}(tH_0,\epsilon_0))K,
\end{gather*}
where $B_{\mathfrak{a}}(H, r)$ denotes the Euclidean ball in $\ga$ of radius $r > 0$ centered at $H \in \mathfrak{a}$. We may associate with $S_{tH_0}$ its averaging operator $U_t$ on $L^2(\Gamma\backslash G)$, and we denote the corresponding self-adjoint time average by $\mathbf{A}(\tau)$. The latter is an integral operator with kernel given by $\sum_{\gamma\in\Gamma} A(\tau)(g,\gamma h)$, where
\begin{equation}\label{intro-kernel}
A(\tau)(g,h)=\frac1{\tau}\int_{\tau}^{2\tau}e^{-2t\rho(H_0)}\int_{g S_{tH_0}\cap h S_{tH_0}}a(x) dx dt,
\end{equation}
where $\rho$ is the half-sum of the positive roots. The primary goal of Section 4 is then to prove that one can essentially replace the matrix coefficient $\langle a\psi_j,\psi_j\rangle$ in the spectral average \eqref{eq:main-estimate} by the matrix coefficient $\langle \mathbf{A}(\tau)\psi_j,\psi_j\rangle$.

This indeed is the content of Theorem \ref{main-spec-thm}, which is valid in a wide degree of generality. Since Theorem \ref{main-spec-thm} is essentially local in nature, no assumptions on $\Gamma$ are necessary. Moreover, the non-zero directing element $H_0$ can be taken to be arbitrary, provided one restricts to spectral parameters avoiding the exceptional hyperplanes $\{P_i\}$ figuring in the statement of Theorem \ref{main-theorem}. To be precise, these exceptional hyperplanes arise in the proof of Proposition \ref{prop:int-W-sum}, and depend on $H_0$ through the set $\mathfrak{a}_{\rm bad}^*$, defined in \eqref{eq:Omega-condition}. (Eventually $H_0$ itself will be chosen to depend on the reduced root system of $G_1$, implying the claimed dependency of the exceptional hyperplanes $\{P_i\}$ after Theorem \ref{main-theorem}.)

\subsection*{Reduction to the intersection volume bound (Section \ref{sec:geom-reduction})}
The outcome of Section \ref{sec:spectral-estimate} is to reduce Theorem \ref{main-theorem} to a similar estimate on the Hilbert--Schmidt norm $\|\mathbf{A}(\tau)\|_{{\rm HS}}$. We study the latter by realizing it as the $L^2$-norm of its kernel. Theorem \ref{main-geom-thm} establishes a bound on $\|\mathbf{A}(\tau)\|_{{\rm HS}}$ in terms of various geometric and spectral quantities controlled by the hypotheses in Theorem \ref{main-theorem}. We then show, using known Limit Multiplicity theorems recalled in Section \ref{sec:limit-mult}, how Theorem \ref{main-geom-thm} suffices to deduce Theorem \ref{main-theorem}.

As is evident from \eqref{intro-kernel}, the Hilbert--Schmidt norm of the operator $\mathbf{A}(\tau)$ encodes the dynamical properties of the intersection of the spherical shells $S_{tH_0}$ with their group translates. The most difficult step in the proof of Theorem \ref{main-geom-thm} is to bound the volume of these intersections, as stated in Theorem \ref{intersection-vol} below. The bulk of Section \ref{sec:geom-reduction} is then dedicated to the reduction of Theorem \ref{main-geom-thm} to Theorem \ref{intersection-vol}.

This reduction proceeds by bounding $\|\mathbf{A}(\tau)\|_{{\rm HS}}$ by means of a thick-thin decomposition of the locally symmetric space $Y=\Gamma\backslash G/K$. This resulting bound produces a main term, denoted as $M(\tau)$ and defined in \eqref{eq:main-term}, and an error term, occuring in Lemma \ref{lemma:general-HS}. The error term involves the volume of the thin part and the global injectivity radius of $Y$, controlled by the hypotheses of Theorem \ref{main-theorem}, as well as the support of the kernel function \eqref{intro-kernel}, determined in Corollary \ref{cor:intersection-range}.

The analysis of the main term $M(\tau)$ is much more delicate. In Proposition \ref{prop:main-term-bd} we use the Minkowski integral inequality and the Nevo ergodic theorem to set up an application of Theorem \ref{intersection-vol}. Once this is inserted, we execute in Section \ref{sec:Brion} one last integral using a degenerate form of Brion's formula to conclude the proof of Theorem \ref{main-geom-thm}.

\subsection{Bounds on intersection volumes (Sections \ref{sec_int_vol}-\ref{sec:sph-fn})} \label{sub_sec_int_vol}

It then remains to bound the volumes of intersection translates of expanding spherical shells. For this, we wish to find an $H_0$ such that for all $H \in \mathfrak{a}$ and $t\gg 1$, the volume of the intersection $e^H S_{tH_0} \cap S_{tH_0}$ is ``small'', in a suitable sense. We are able to do so for all non-compact simple groups satisfying condition \eqref{thm:root-system} of Theorem \ref{main-theorem}, as we now describe.

In Section \ref{sec_int_vol}, we outline an approach to bounding intersection volumes using \textit{analysis}, rather than geometry. Indeed, we pass to the spectral side using Plancherel inversion and present a strategy which reduces the desired bound to uniform bounds on spherical functions, which we then prove. This reduction strategy only succeeds when the directing element $H_0$ is highly symmetric.

Let $M$ be the centralizer of $H_0$ inside $G$. In Section \ref{sec:big-subsystems}, we make this symmetry condition on $H_0$ precise by examining the way in which the reduced root subsystem of $M$ sits inside that of $G$. We first do this abstractly (making no mention of $H_0$), by introducing a combinatorial property of reduced root subsystems which encodes its relative fullness. Namely, if $\Phi$ is a reduced root system, and $\Phi_0\subset\Phi$ is a semistandard root subsystem, we say that $\Phi_0$ is \textit{semi-dense} in $\Phi$ if, roughly speaking, for any semistandard root subsystem $\Psi \subset \Phi$, $\Phi_0$ contains at least half the roots of $\Psi$. The required inequality is stated exactly in Definition \ref{def:big}.

\begin{theorem}\label{thm:semi-dense} An irreducible reduced root system $\Phi$ contains a semi-dense root subsystem if, and only if, $\Phi$ is of type $A_n, B_n, C_n, D_n$ or $E_7$.
\end{theorem}

We now return to the setting of Theorem \ref{main-theorem} and denote by $\Phi$ the set of restricted roots for $A$ in $G$ and by $\Phi_{\rm red}\subset\Phi$ the subsystem of reduced (or indivisible) roots. As before, let $M$ be the centralizer of $H_0$ in $G$; then $M$ is a standard Levi subgroup of $G$. Let $\Phi_M\subset\Phi$ be the roots of $A$ in $M$, and $\Phi_{M, {\rm red}}\subset\Phi_{\textnormal{red}}$ the corresponding subsystem of reduced roots. In the case where $\Phi_{\rm red}$ is of type $A_n, B_n, C_n, D_n$ or $E_7$, we define in Section \ref{sec:str-sing} a class of elements $H_0\in \overline{\ga}_+$, which we call \textit{extremal}. Extremal $H_0$ will turn out to be our primary source of extremal root subsystems, thanks to the following result.

\begin{theorem}\label{thm:semi-dense2}
Let $G$ be such that $\Phi_{\textnormal{red}}$ is of type $A_n, B_n, C_n, D_n$ or $E_7$ and let $H_0$ be extremal. Then $\Phi_{M,\textnormal{red}}$ is a semi-dense root subsystem of $\Phi_{\rm red}$.
\end{theorem}

Theorems \ref{thm:semi-dense} and \ref{thm:semi-dense2} are proved in Section \ref{sec:big-subsystems}, in the form of Propositions \ref{Sibd}-\ref{lemma:exceptional-small}.

We ultimately provide strong bounds on ${\rm vol}(e^H S_{tH_0} \cap S_{tH_0})$ for extremal $H_0$. To better understand their quality, first remark that the volume of $S_{tH_0}$ is asymptotically of size $e^{2t\rho(H_0)}$. Moreover, as we shall see in Section \ref{sec:ker-supp}, the spherical shell $S_{tH_0}$ and its translate $e^HS_{tH_0}$, where $H\in\overline{\ga}_+$, do not intersect each other as soon as $\rho( H) > 2t\rho(H_0)$. The following bound interpolates between these two extremities (full and empty intersection), losing only a logarithmic factor of $t$.

\begin{theorem} \label{intersection-vol}
Let $G$ be a non-compact simple real Lie group. If $H_0 \in\overline{\mathfrak{a}}_+$ is such that $\Phi_{M,\textnormal{red}}$ is semi-dense in $\Phi_{\textnormal{red}}$, then there exists a non-negative integer $k$ such that for all $H \in \overline{\mathfrak{a}}_+$ and $t \gg 1$: 
    \begin{gather} \label{eqn_vol_bound}
        \textnormal{vol}(e^HS_{t H_0} \cap S_{t H_0}) \ll (\log t)^k e^{\rho(2 t H_0 - H)}.
    \end{gather}
\end{theorem}

J.-P. Anker has explained to us a straightforward argument which bounds the intersection volume $\textnormal{vol}(e^HS_{t H_0} \cap S_{t H_0})$, for all choices of $H_0$, by a similar expression to that in \eqref{eqn_vol_bound} but with the $(\log t)^k$ factor replaced by $t^n$ for some positive integer $n\geq 1$. We review this argument in Section \ref{sub_sec_anker}. The fact that Theorem \ref{intersection-vol} yields poly-logarithmic factors rather than positive integral powers in $t$ is critical to the overall argument. Indeed, any positive integral power of $t$ would kill the decay in the main term of Theorem \ref{main-geom-thm}.

In \cite{Brumley_Matz_23}, the first and third authors considered a similar quantity to the left-hand side of \eqref{eqn_vol_bound}, but for a choice of $H_0$ which was not extremal. In \cite[Lemma 5.8]{Brumley_Matz_23}, they claimed a bound like the right-hand side but with the $(\log t)^k$ factor replaced by a constant. In his Ph.D. thesis \cite{Peterson_thesis}, the fourth author found a mistake in their argument and showed that the analogous bound for ${\rm PGL}_3(\Q_p)$ does not hold. Namely, for the choice of $H_0$ from \cite{Brumley_Matz_23} and specific choices of $H$, we have
\[
\textnormal{vol}(p^H S_{t H_0} \cap S_{t H_0}) \gg t p^{\rho (2 t H_0 - H)},
\]
where $S_{t H_0} = \textnormal{PGL}_3(\Z_p) p^{t H_0} \textnormal{PGL}_3(\Z_p)$. On the other hand, if one takes an extremal $H_0$, then $\textnormal{vol}(p^H S_{t H_0} \cap S_{t H_0}) \ll p^{\rho(2 t H_0 - H)}$. This suggests that in the archimedean setting we should not expect \eqref{eqn_vol_bound} to hold for generic choices of $H_0$.

\subsection{Bounds for spherical functions (Section \ref{sec:sph-fn})}\label{intro-sph-bds}

The proof of Theorem \ref{intersection-vol} relies crucially on a new bound for the spherical function.  As this bound holds on a general semisimple group, and may be of independent interest, we shall state and prove it independently of Theorem \ref{main-theorem} and Theorem \ref{intersection-vol}.

Let $G$ be any non-compact semisimple real Lie group with finite center. We let $\varphi_\lambda$ denote the Harish-Chandra spherical function on $G$ with spectral parameter $\lambda \in \ga^*_\C$. See Section \ref{subsec:sph-fn} for definitions. To state our bound, we first define
\begin{equation}\label{defn-falpha}
    f_{\alpha}(H, \lambda) := \min(|\alpha(H)|+1, |\langle \lambda, \alpha \rangle|^{-1} + 1),
\end{equation}
where $\alpha \in \Phi^+_{\textnormal{red}}$, $H \in \mathfrak{a}$, and $\lambda \in \mathfrak{a}^*_{\mathbb{C}}$. We subsequently define
\begin{gather}
\label{Thetadef}
    \Theta(H, \lambda) := \sum_{w \in W} \prod_{\alpha \in \Phi^+_{\textnormal{red}}} f_{\alpha}(H, w\lambda).
\end{gather}
Note that $\Theta$ is $W$-invariant in both $\lambda$ and $H$, and also satisfies $\Theta(H, \lambda) \geq 1$. Let $\mathfrak{a}^*(\kappa)$ denote those elements in $\mathfrak{a}^*_{\mathbb{C}}$ whose imaginary part is bounded in size by $\kappa$.

\begin{theorem}
\label{sph-fn-bd}
There are $a, \kappa > 0$ such that, for all $H \in \overline{\ga}_+$ and all $\lambda \in \ga^*(\kappa)$, we have
\[
\varphi_\lambda(e^H) \ll (1 + \| \lambda \|)^a \Theta(H, \lambda) \underset{w \in W}{\max} \, e^{-(\rho + w \Im \lambda)(H) }.
\]
\end{theorem}

We shall prove Theorem \ref{sph-fn-bd} in Section \ref{sec:sph-fn}.  To aid the reader in understanding the statement of the theorem, in Section \ref{sec:sph-fn-conseq} we shall illustrate it in the case of ${\rm SL}_2(\C)$ and $\textnormal{SL}_2( \mathbb{R})$, as well as deriving some simpler consequences of it in general.  Section \ref{sec:sph-fn-prev} also contains a discussion of the relation between Theorem \ref{sph-fn-bd} and previous bounds for the spherical function.

When $G$ is a complex group, Theorem \ref{sph-fn-bd} has an interesting `self-improving' property which allows us to significantly strengthen it, subject to the condition that $\lambda$ lie in $\ga^*$.  In particular, we expect this bound to be sharp in all aspects subject to this tempered condition on $\lambda$.  Moreover, this bound may be transferred to the Cartan motion group associated with $G$, whose definition we recall below. This correspondence uses the well-known link between the Cartan motion group and the semisimple group in the complex case. For this reason, we shall state the bound for both groups simultaneously.

The Cartan motion group associated to $G$ is the semidirect product $\p \ltimes K$, which acts on the Euclidean symmetric space $\p$.  For $\lambda \in \ga^*_\C$, we have the Euclidean spherical function $\varphi_\lambda^E \in C^\infty(\p)$ (also known as the generalized Bessel function) defined by
\be
\label{euclideandef}
\varphi_\lambda^E(Z) = \int_K e^{ i\lambda( kZ) } dk \quad (Z \in \p),
\ee
where we have extended $\lambda$ to an element of $\p^*$ by orthogonality.  To recall the relationship between $\varphi_\lambda^E$ and $\varphi_\lambda$, we let
\[
Q(H) = \prod_{\alpha \in \Phi^+} \sinh(\alpha(H))/\alpha(H)\quad (H\in\ga).
\]
As $Q$ is Weyl invariant, it extends to a function on $\p$, and is equal to the Jacobian of the exponential map $\exp: \p \to G/K$.  We then have 
\[
\varphi_\lambda(e^Z) = C Q(Z)^{-1/2} \varphi_\lambda^E(Z)
\]
for some constant $C > 0$, see for instance Ch. IV, Theorem 4.7 of \cite{Helgason_00}.  Our bounds for $\varphi_\lambda$ and $\varphi_\lambda^E$ are as follows.

\begin{theorem}
\label{sph-fn-cx}

Suppose that $G$ is a complex group.  Then for $H \in \overline{\ga}_+$ and $\lambda \in \ga^*$, we have
\be
\label{cxeuclideanbd}
\varphi_\lambda^E(H) \ll \sum_{w \in W} \prod_{\alpha \in \Phi^+} (1 + |\alpha(H) \langle w\lambda, \alpha \rangle|)^{-1},
\ee
and
\be
\label{cxbound}
\varphi_\lambda(e^H) \ll e^{-\rho(H)} \prod_{\alpha \in \Phi^+} (|\alpha(H)|+1) \sum_{w \in W} \prod_{\alpha \in \Phi^+} (1 + |\alpha(H) \langle w\lambda, \alpha \rangle|)^{-1}.
\ee

\end{theorem}

We note that
\[
Q(H)^{-1/2} \asymp e^{-\rho(H)} \prod_{\alpha \in \Phi^+} (|\alpha(H)|+1)
\]
for $H \in \overline{\ga}_+$, so that the bounds \eqref{cxeuclideanbd} and \eqref{cxbound} are equivalent.  In Section \ref{sec:sph-fn-conseq} we explain why we expect Theorem \ref{sph-fn-cx} to be sharp based on an analysis of the oscillatory integral \eqref{euclideandef}, and in Section \ref{sec:sph-fn-prev} we discuss the relation between Theorem \ref{sph-fn-cx} and previous bounds for the spherical function on complex groups.  We deduce Theorem \ref{sph-fn-cx} from Theorem \ref{sph-fn-bd} in Section \ref{sec:cx-pf}.

\section{Preliminaries}\label{sec:preliminaries}
In this section we introduce most of the notation that will be in force throughout the rest of the paper, and recall some of the definitions and foundational results that enter into the statement and proof of Theorem \ref{main-theorem}. More exactly, after setting up some standard notational conventions in Sections \ref{sec:basic}-\ref{subsec:sph-fn}, we give standard bounds on the Harish-Chandra $c$-function and recall an asymptotic formula for the spherical function due to Gangolli--Varadararajan in Section \ref{sec:c-bds}, define \textit{maximally singular} and \textit{extremal} elements in Section \ref{sec:str-sing}, and state the pertinent formulation of Benjamini--Schramm convergence in Section \ref{sec:BS}. Finally, in Sections \ref{sec:limit-mult}-\ref{sec:SG} we recall the relevant limit multiplicity theorems   and the notion of uniform spectral gap.

Throughout this section we shall let $G$ denote a connected semisimple real Lie group with finite center and no compact factors. Additional hypotheses on $G$ (such as those appearing in Theorem \ref{main-theorem}) will be assumed later in the paper as needed.

\subsection{Basic notation}\label{sec:basic}
Fix a maximal compact subgroup $K$ of $G$. Denote by $\mathfrak{k}$ and $\g$ their respective Lie algebras. Then $K$ induces a Cartan involution $\Theta$ on $G$. The differential at the identity of $\Theta$ defines an involution on $\g$, whose $-1$ and $+1$ eigenspaces determine the Cartan decomposition $\g=\mathfrak{p}\oplus\mathfrak{k}$.

Let $\mathfrak{a}\subset\mathfrak{p}$ denote a maximal abelian subspace. Let $A = \exp(\mathfrak{a})$ be the corresponding analytic subgroup of $G$. Let $\ga^*$ denote the dual space of $\ga$. Let $\Phi \subset\ga^*$ be the set of restricted roots. Let $\Phi^+\subset\Phi$ be a choice of positive roots. We let $\Delta\subset\Phi^+$ denote the set of simple roots; they form a basis for $\ga^*$. Then the fundamental coweights $\widehat{\Delta}^\vee=\{\varpi_\alpha^\vee : \alpha\in\Delta\}\subset\ga$ form the dual basis to the simple roots and give rise to decompositions
$\lambda=\sum_{\alpha\in\Delta} \lambda(\varpi_\alpha^\vee)\alpha$,  $H=\sum_{\alpha\in\Delta}\alpha(H)\varpi_\alpha^\vee$ of elements $\lambda\in\ga^*, H\in\ga$. Let
\[
\ga_+=\{H\in\ga: \alpha(H)>0\;\forall\alpha\in\Delta\}=\sum_{\alpha\in\Delta} \R_{> 0}\varpi_\alpha^\vee\\
\]
be the fundamental Weyl chamber in $\ga$ and denote by $\overline{\ga}_+=\sum_{\alpha\in\Delta} \R_{\ge 0}\varpi_\alpha^\vee$ its closure.

We have a restricted root space decomposition
\[
\g=\mathfrak{m}\oplus\ga\oplus\sum_{\alpha\in\Phi}\g_\alpha,
\]
where $\mathfrak{m}=Z_{\mathfrak{k}}(\ga)$ is the centralizer of $\ga$ in $\mathfrak{k}$ and $\g_\alpha$ is the root space for $\alpha$. As usual, we let $\rho=\frac{1}{2}\sum_{\alpha\in\Phi^+}m_\alpha \alpha$ be the half-sum of the positive roots, where we have put $m_\alpha=\dim\g_\alpha$. Let $r=\dim\ga$ be the rank of $G$ and $W=N_K(\ga)/Z_K(\ga)$ the Weyl group of $G$. Let $\Phi_\text{red}$ denote the set of reduced roots of $\Phi$, consisting of those $\alpha\in\Phi$ such that $\alpha/2\notin\Phi$. Then $\Phi_\textnormal{red}$ is a reduced root system, and we refer to it as the reduced root system of $G$. Let $\Phi_\text{red}^+ := \Phi_\text{red} \cap \Phi^+$. Note that $\Delta$ is also a base for $\Phi_{\textnormal{red}}$. 

We let $\langle \cdot ,\cdot\rangle$ denote the restriction of the Killing form to $\ga$. We equip $\ga^*$ with the inner product $\langle \alpha ,\beta \rangle =\langle H_\alpha, H_\beta\rangle$, where $H_\alpha\in\ga$ is uniquely determined by $\alpha(Z)=\langle H_\alpha, Z\rangle$. For any non-zero $\alpha\in\ga^*$, we may then define the orthogonal reflection $\lambda\mapsto \lambda-2\frac{\langle\lambda,\alpha\rangle}{\langle\alpha,\alpha\rangle}\alpha$ across the root hyperplane $\alpha^\perp$. Note that $\lambda\mapsto 2\frac{\langle\lambda,\alpha\rangle}{\langle \alpha,\alpha \rangle}$ is an element in $\ga^{**}$, and as such determines an element $\alpha^\vee\in\ga$, the \textit{coroot} associated with $\alpha$, such that $2\frac{\langle \lambda,\alpha \rangle}{\langle \alpha,\alpha \rangle}=\lambda(\alpha^\vee)$.  Then the orthogonal reflection can be written $\lambda\mapsto\lambda-\lambda(\alpha^\vee)\alpha$ and hence $\alpha^\perp=
\{\lambda\in\ga^*: \lambda(\alpha^\vee)=0\}$.

Denote the set of (restricted) coroots and simple coroots by $\Phi^\vee=\{\alpha^\vee: \alpha\in\Phi\}\subset\ga$ and $\Delta^\vee=\{\alpha^\vee: \alpha\in\Delta\}\subset\ga$, respectively. By definition, the fundamental weights $\widehat{\Delta}=\{\varpi_\alpha: \alpha\in \Delta\}\subset\ga^*$ form the dual basis to the simple coroots $\Delta^\vee$. We have corresponding decompositions $\lambda=\sum_{\alpha\in\Delta} \lambda(\alpha^\vee)\varpi_\alpha$, $H=\sum_{\alpha\in\Delta}\varpi_\alpha(H)\alpha^\vee$ of elements $\lambda\in\ga^*, H\in\ga$. Let
\[
\ga_+^*=\{\lambda\in\ga^*: \lambda(\alpha^\vee)>0\;\forall\alpha\in\Delta\}=\sum_{\alpha\in\Delta} \R_{> 0}\varpi_\alpha\\
\]
be the fundamental Weyl chamber in $\ga^*$ and denote by $\overline{\ga}_+^*=\sum_{\alpha\in\Delta} \R_{\ge 0}\varpi_\alpha$ its closure. We write
\begin{equation}\label{eq:reg-sing}
\ga^*_{\rm sing}=\bigcup_{\alpha\in\Phi^+} \{\lambda\in\ga^*: \lambda(\alpha^\vee)=0\}
\end{equation}
for the union of all root hyperplanes and put $\ga^*_{\rm reg}=\ga^*\smallsetminus \ga^*_{\rm sing}$. Then $\ga^*_{\rm reg}=W.\ga_+^*$. Finally, we write $\mathfrak{a}^*_\mathbb{C} := \mathfrak{a}^* \otimes \mathbb{C}$ for the complexification of $\mathfrak{a}^*$.

\subsection{Subsets of simple roots and associated structures}\label{sec:subsets}
Let $\mathfrak{n} = \bigoplus_{\alpha \in \Phi^+} \mathfrak{g}_\alpha$ and write $N = \exp(\mathfrak{n})$ for the corresponding analytic subgroup. Let $P = Z_K(\mathfrak{a}) A N$ be the standard minimal parabolic subgroup. The Bruhat decomposition states that
\begin{gather*}
    G = \bigsqcup_{w \in W} P w P,
\end{gather*}
for any choice $w$ of coset representatives for $W = N_{K}(\mathfrak{a})/Z_K(\mathfrak{a})$.

Any subgroup of $G$ containing $P$ is called a standard parabolic subgroup. These can be constructed in the following way. Let $I \subseteq \Delta$, and let $W_I$ be the subgroup of $W$ generated by the reflections associated to $\alpha \in I$; we define $W_{\emptyset} = 1$. Standard parabolic subgroups are in bijection with subsets $I\subseteq \Delta$, through the map sending $I$ to
\begin{gather*}
    Q_I = \bigsqcup_{w \in W_I} P w P;
\end{gather*}
see \cite[\S\S 29.2-3]{Humphreys}. This recovers the Bruhat decomposition for $G$ when $I=\Delta$.

Let $\Phi_I$ be the root subsystem consisting of $\Z$-linear combinations of roots in $I$ lying in $\Phi$.  Let $\Phi_I^+=\Phi_I\cap\Phi^+$. Note that the roots of $Q_I$ with respect to $A$ are $\Phi^+\cup(-\Phi_I^+) = (\Phi^+ \smallsetminus \Phi_I^+) \cup \Phi_I$ \cite[\S 30.1, Theorem (b)]{Humphreys}.  The following proposition relates $\Phi_I^+$ with the action of $W_I$ on $\Phi^+$.

\begin{proposition} \label{lemma:WL-stable}
Let $I\subseteq \Delta$. Then
\begin{equation}\label{eq:WI}
W_I=\{w\in W\mid w \Phi^+\subseteq (\Phi^+ \smallsetminus \Phi_I^+) \cup \Phi_I \}.
\end{equation}
In particular, $w(\Phi^+ \smallsetminus \Phi_I^+) = \Phi^+ \smallsetminus \Phi_I^+$ if and only if $w \in W_I$. Furthermore,
\begin{equation}
\begin{aligned}\label{eq:Phi-WI}
\Phi_I^+&=\{\alpha\in\Phi^+ \mid \exists w\in W_I:\, w\alpha<0\},\\
\Phi^+\smallsetminus\Phi_I^+ &=\{\alpha\in\Phi^+ \mid w\alpha>0 \, \forall w\in W_I\}.
\end{aligned} 
\end{equation}
\end{proposition}

\begin{proof}
The direct inclusion in \eqref{eq:WI} holds, since $W_I$ preserves the roots of $Q_I$. On the other hand, if $w\in W$ is such that $w \Phi^+\subseteq \Phi^+\cup (-\Phi_I^+)$, then  $w$ maps $P$ into $Q_I$ and therefore belongs to $Q_I$ by \cite[\S 29.3, Lemma D]{Humphreys} so that $w\in Q_I\cap W=W_I$.

To see that $W_I$ consists precisely of those $w\in W$ for which $w(\Phi^+ \smallsetminus \Phi_I^+) = \Phi^+ \smallsetminus \Phi_I^+$  we first observe that since $W_I$ preserves $\Phi_I$ and $(\Phi^+ \smallsetminus \Phi_I^+) \cup \Phi_I$, it must also preserve $\Phi^+ \smallsetminus \Phi_I^+$. Next, taking complements in \eqref{eq:WI} gives
\[
W_I  = \{w\in W \mid -( \Phi^+ \smallsetminus \Phi_I^+) \subset w (- \Phi^+) \}  = \{w\in W \mid w^{-1}( \Phi^+ \smallsetminus \Phi_I^+) \subset \Phi^+ \},
\]
so that $w(\Phi^+ \smallsetminus \Phi_I^+) = \Phi^+ \smallsetminus \Phi_I^+$ implies $w \in W_I$ as required.

The two statements in \eqref{eq:Phi-WI} are clearly equivalent so that it suffices to prove the second.  The direct inclusion follows from the fact that $W_I$ preserves $\Phi^+ \smallsetminus \Phi_I^+$. On the other hand, suppose $\alpha \in \Phi^+$ is such that $w \alpha > 0$ for all $w \in W_I$. It suffices to show that $W_I \alpha \subset \Phi^+ \smallsetminus \Phi_I^+$, since this implies, by the preceding claim, that $\alpha \in \Phi^+ \smallsetminus \Phi_I^+$. Now \eqref{eq:WI} already yields the inclusion $W_I \alpha\subset(\Phi^+ \smallsetminus \Phi_I^+) \cup \Phi_I$. But $W_I\alpha$ has no intersection with $\Phi_I$. Indeed, every element in $W_I\alpha$ is positive while the $W_I$-orbit of every element in $\Phi_I$ meets $-\Phi^+$, since $\Phi_I$ is a root system with Weyl group $W_I$.
\end{proof}

Let $Q=Q_I$ be a standard parabolic subgroup. Its unipotent radical $U$ has Lie algebra $\bigoplus_{\alpha \in \Phi^+\smallsetminus \Phi_I^+} \mathfrak{g}_{\alpha}$. For example, when $I=\emptyset$, we have $Q=P$ and $U=N$. Then $Q$ admits a Levi decomposition $Q = L \ltimes U$, where $L=L_I=Q \cap \Theta(Q)$. The reductive subgroup $L$ is called the Levi component of $Q$; its Lie algebra decomposes as $\mathfrak{a} \oplus\mathfrak{m}\oplus \bigoplus_{\alpha \in \Phi_I} \mathfrak{g}_{\alpha}$. By \textit{standard Levi subgroup} we shall mean any Levi component of a standard parabolic subgroup. The Weyl group of $L$, relative to the maximal compact subgroup $K_L = K \cap L$, and maximal split torus $A\subset L$, is equal to $W_I$ \cite[\S 27.1, Theorem]{Humphreys}.

The Levi component $L_I$ itself determines the standard parabolic subgroup $Q_I$. Moreover, we may identify $L_I$ as the centralizer of any element in the open cone $C_I=\sum_{\alpha \in \Delta \smallsetminus I} \mathbb{R}_{> 0} \varpi^\vee_\alpha$, and all standard Levi subgroups are obtained in this way \cite[\S 30.2]{Humphreys}. The nodes of the Dynkin diagram of $\Phi$ are also naturally labeled by elements in $\Delta$. We thus have one-to-one correspondences between subsets $I\subseteq\Delta$, standard parabolics $Q_I$, standard Levi subgroups $L_I$, faces $C_I$ of $\overline{\mathfrak{a}}_+$, and subdiagrams of the Dynkin diagram for $\Phi$ obtained by deleting the nodes corresponding to $\Delta \setminus I$. In this last correspondence, the type of the resulting Dynkin diagram determines the type of the root system $\Phi_I$.

\subsection{Measures and Jacobian factors}\label{sec:Jacobian}
Let $X=G/K$ be the Riemannian globally symmetric space associated to $G$. Let $B$ denote the Killing form on $\g$. Then $B$ is positive definite on $\mathfrak{p}$ and hence defines an inner product on $\mathfrak{p}$. Identifying $\mathfrak{p}$ with the tangent space of $X$ at the origin, we then transport this inner product to a $G$-invariant Riemannian metric on $X$, of non-positive curvature. The restriction of $B$ to $\mathfrak{a}$ defines a $W$-invariant inner product. Let $dH$ denote the corresponding Lebesgue measure on $\mathfrak{a}$. Let $dk$ denote the probability Haar measure on $K$. We normalize Haar measure $dg$ on $G$ in such a way that its quotient by $dk$ agrees with the measure $dx$ induced by the Riemannian volume form on $X$. If $E$ is a subset of $G$ or $X=G/K$, we write ${\rm vol}(E)$ for $\int_E dg$ or $\int_E dx$, respectively.

The Cartan decomposition states that $G = K \exp(\overline{\mathfrak{a}}_+)K$. Correspondingly, there is a constant $b_G>0$ such that, for every $f\in C_c(G)$, we have 
\begin{equation}\label{eq:Cartan-int}
\int_G f(g)dg = b_G\int_K\int_{\mathfrak{a}_+}\int_K f(k_1 \exp(H) k_2) J(H) dk_1dH dk_2,
\end{equation}
where, for $H \in \mathfrak{a}_+$, the radial volume factor $J \in C^\infty(\ga_+)$ is given by
\begin{equation}\label{eq:def:J}
 J(H) =\prod_{\alpha\in \Phi^+}\big(\sinh \alpha(H) \big)^{m_\alpha}.
\end{equation}
See \cite[\S 2.4, Prop. 2.4.11]{Gangolli_Varadarajan_88}. By extending to $\ga_{\rm reg}$ by $W$-invariance, and then by continuity from $\ga_{\rm reg}$ to $\ga$, we may also view $J(H)$ as a continuous $W$-invariant function on $\ga$.

More generally, we let $L=L_I$ be a standard Levi subgroup, where $I$ is the corresponding subset of $\Delta$, as in the bijective correspondence of Section \ref{sec:subsets}. Denote by $W_L = W_I$ its Weyl group. Let $\Phi_L = \Phi_I$, and $\Phi_L^+ = \Phi_L \cap \Phi^+$. We let
\be
\label{aLdef}
\ga_{L,+} = \{ H \in \ga : \alpha(H) > 0 \; \forall \alpha \in I \}
\ee
be the fundamental Weyl chamber for $L$. We again give $K_L$ the probability Haar measure $dk$, and we denote by $dl$ the Haar measure on $L$, normalized in the same way as for $G$. The Cartan decomposition for $L$ states that $L=K_L\exp(\overline{\ga}_{L,+})K_L$. Once again, there is a constant $b_L>0$, depending only on $L$, such that 
\begin{equation}\label{L-int-formula}
\int_L f(l)dl = b_L \int_{K_L}\int_{\ga_{L,+}}\int_{K_L} f(k_1\exp(H)k_2)J_L(H) dk_1dH dk_2,
\end{equation}
where the radial volume factor with respect to $L$ is
\begin{equation}\label{eq:def:JM}
 J_L(H)=\prod_{\alpha\in \Phi_L^+} \big(\sinh\alpha(H) \big)^{m_\alpha}.
\end{equation}
We define $J^L = J / J_L$.

\subsection{Spherical functions, spherical transform, spherical inversion}\label{subsec:sph-fn}

Let $D_{G}(X)$ denote the algebra of $G$-invariant differential operators on $X$. A spherical function on $X$ is a $K$-invariant joint eigenfunction of $D_{G}(X)$, normalized to take the value $1$ at the identity. Harish-Chandra gave an integral representation for spherical functions, which we now review.

The Iwasawa decomposition for $G$ states that the multiplication map $N \times A \times K\rightarrow G$ is a diffeomorphism. Let $\mathcal{H}: G \to \ga$ be the Iwasawa projection, given by $g = n\exp(\mathcal{H}(g))k$. For $\lambda \in \ga^*_{\mathbb{C}}$ the map $g\mapsto e^{i\lambda(\mathcal{H}(g))}$ descends to a plane wave on $X$. The average over left-$K$-orbits
\begin{gather}\label{eqn_spherical_fn}
    \varphi_{\lambda}(g) := \int_K e^{(i\lambda + \rho) (\mathcal{H}(k g))} dk
\end{gather}
defines a bi-$K$-invariant function on $G$ --- equivalently, a $K$-invariant function on $X$ --- such that
\begin{gather*}
D \varphi_{\lambda} = \gamma_{\rm HC}(D)(\lambda) \varphi_{\lambda} \qquad (\forall\; D\in D_G(X)),
\end{gather*}
and $\varphi_{\lambda}(e) =1$. Here, $\gamma_{\rm HC}$ denotes the Harish-Chandra isomorphism
\begin{gather*}
    \gamma_{\rm HC}: D_{G}(X) \xrightarrow{\;\sim\;} \mathrm{Sym}(\ga_{\mathbb{C}})^W
\end{gather*}
onto the $W$-invariant elements of the symmetric algebra of $\ga_{\mathbb{C}}$, viewed as $W$-invariant polynomials on $\ga^*_{\mathbb{C}}$. We call $\varphi_{\lambda}$ the \textit{spherical function with spectral parameter $\lambda$.}

We have $\varphi_{\lambda} = \varphi_{\lambda'}$ if and only if $\lambda$ and $\lambda'$ are in the same $W$-orbit. For this reason, we call $\ga^*_{\mathbb{C}}/W$ the space of \textit{spectral parameters.} The spectral parameters lying in $\ga^*/W$ are called \textit{tempered spectral parameters}. The tempered subspace $\ga^*/W$ arises naturally in the decomposition of $L^2(X)$, as will be seen in the following subsection.

The group $L$ has its own Iwasawa decomposition $L = N_L A K_L  $ where $N_L$ is the analytic subgroup corresponding to $\bigoplus_{\alpha \in \Phi_L^+} \mathfrak{g}_\alpha$. We can thus in turn define the spherical function $\varphi_\lambda^L$ on $L$ by the formula \eqref{eqn_spherical_fn}, replacing $\rho$ with $\rho_L = \frac{1}{2} \sum_{\alpha \in \Phi_L^+} m_\alpha \alpha$. Then $\varphi_{\lambda}^L = \varphi_{\lambda'}^L$ if and only if $\lambda$ and $\lambda'$ are in the same $W_L$-orbit.

Let $C_c^\infty(G/\!\!/K)$ denote the space of compactly supported smooth bi-$K$-invariant functions. Given $k \in C_c^\infty(G /\!\!/ K)$ let
\begin{equation}\label{eq:HC-transform}
\widehat{k}(\lambda) = \int_G k(g) \varphi_{-\lambda}(g) dg= b_G \int_{\ga^+} k(e^H) \varphi_{-\lambda}(e^H) J(H)dH
\end{equation}
be the Harish-Chandra spherical transform. This transform can be inverted using the Harish-Chandra $c$-function $c(\lambda)$, which is a meromorphic function on $\lambda\in\ga^*_\C$, given by the product
\begin{equation}\label{eq:explicit-c}
c(\lambda)=C\prod_{\alpha\in\Phi_{\textnormal{red}}^+}c_\alpha\Big( \frac{i \langle \lambda, \alpha \rangle}{\langle \alpha, \alpha \rangle}\Big),\qquad c_\alpha(s) = \frac{2^{-s} \Gamma (s)}{\Gamma(\frac{1}{2}(\frac{1}{2}m_\alpha + 1 + s)) \Gamma(\frac{1}{2}(\frac{1}{2}m_\alpha + m_{2 \alpha} + s))},
\end{equation}
for a non-zero constant $C$. Then we have the following inversion formula:
\begin{gather*}
    k(g) = \int_{\ga_+^*} \widehat{k}(\lambda) \varphi_{\lambda}(g) |c(\lambda)|^{-2} d \lambda.
\end{gather*}
More generally, for a standard Levi subgroup $L$, we let $c_L$ be its $c$-function, given by the product \eqref{eq:explicit-c} but with $\Phi_{\textnormal{red}}^+$ replaced by $\Phi^+_{\textnormal{red},L}$.

\subsection{Bounds on the $c$-function and the spherical function}\label{sec:c-bds}
In this section we recall standard bounds on the $c$-function that will be used throughout the later sections and state the generalized Harish-Chandra asymptotic of the spherical function, a result which undergirds the proof of Theorem \ref{sph-fn-bd}.

From its definition in \eqref{eq:explicit-c} it is clear that each function $c_\alpha(s)$ has a simple pole at $s = 0$. Therefore, for $\lambda$ close to zero, we have
\begin{gather} \label{eqn_c_lambda_small}
    |c(\lambda)|^{-2} \ll \prod_{\alpha \in \Phi_{\textnormal{red}}^+} |\langle \lambda, \alpha \rangle|^2.
\end{gather}
On the other hand, if $s$ is large with real part bounded by $\kappa$, we have that \cite{Duistermaat_Kolk_Varadarajan_79}
\begin{gather} \label{eqn_c_lambda_big}
    |c_{\alpha}(s)|^{-2} \ll_\kappa |s|^{m_\alpha + m_{2 \alpha}}. 
\end{gather}

Let $L$ be a standard Levi subgroup of $G$. From \eqref{eq:reg-sing} and the expression \eqref{eq:explicit-c} it follows that $c(\lambda)$ is holomorphic and non-vanishing on the regular parameters $\ga^*_{\rm reg}$. More generally, the quotient
\begin{equation}\label{def:cM}
c^L(\lambda)=\frac{c(\lambda)}{c_L(\lambda)}=C^L \prod_{\alpha\in\Phi_{\textnormal{red}}^+\smallsetminus\Phi^+_{\textnormal{red},L}}c_\alpha\Big(\frac{i \langle \lambda, \alpha \rangle}{\langle \alpha, \alpha \rangle}\Big)\qquad (C^L=C/C_L),
\end{equation}
is holomorphic and non-vanishing on
\begin{equation}\label{eq:cM-domain}
\ga^*\smallsetminus \bigcup_{\alpha\in\Phi^+\smallsetminus\Phi^+_L}\{\lambda\in\ga^*: \lambda(\alpha^\vee)=0\},
\end{equation}
an open set in $\ga^*$ containing $\ga^*_{\rm reg}$. The main term in the generalized Harish-Chandra asymptotic, stated below, will be seen to be
\begin{equation}\label{defn:thetaL}
    \theta_L(H, \lambda) := \sum_{w \in W_L \backslash W} c^L(w \lambda) \varphi^L_{w\lambda}(e^H),
\end{equation}
and the error term is governed by the function $\beta_L: \ga \to \mathbb{R}$ given by
\begin{equation}\label{defn:beta}
    \beta_L(H) := \min_{\alpha \in \Phi^+ \smallsetminus \Phi_L^+} |\alpha(H)|.
\end{equation}
Note that $\theta_L(H, \lambda)$ is defined on $\ga^*_{\rm reg}$.

\begin{proposition}[Theorem 5.9.4 of \cite{Gangolli_Varadarajan_88}] \label{prop_hc_expansion}
Let $\varepsilon > 0$. There exist constants $C > 0$ and $s \geq 0$ such that for all $\lambda \in \ga^*_{\textnormal{reg}}$ and for all $H \in \overline{\ga}_+$ such that $\beta_L(H) \geq \varepsilon$, we have
\begin{gather*}
    \Big| e^{\rho(H)} \varphi_\lambda(e^H) - e^{\rho_L(H)} \theta_L(H, \lambda) \Big| \leq C (1 + \|\lambda\|)^s (1 + \|H\|)^s e^{-2 \beta_L(H)}.
\end{gather*}
\end{proposition}
We shall discuss an extension of this result to a small tube about the tempered subspace $\ga^*$ in Proposition \ref{lem:complexGV}. In the case where $L = Z_{K}(\ga) A$, so that $L$ is the unique minimal standard Levi, we have that $\rho_L = 0$, and $\varphi_{\lambda}^L(e^H) = e^{i\lambda(H)}$. It follows that $e^{-\rho(H)} \theta_L(H, \lambda)$ is equal to the leading term in the classical Harish-Chandra expansion.

\subsection{Maximally singular and extremal elements}\label{sec:str-sing}

An element $H_0 \in \overline{\ga}_+$ is called \textit{singular} if $H_0\in \overline{\ga}_+\smallsetminus\ga_+$.
Non-zero elements lying on the one-dimensional faces of the polyhedral cone $\overline{\ga}_+$ are the most singular among all non-zero elements and shall be called \textit{maximally singular}. As the fundamental coweights form a conical basis for $\overline{\ga}_+$, the maximally singular elements are nothing other than the positive multiples of fundamental coweights. Note that the fundamental coweights correspond bijectively with the nodes of the Dynkin diagram. 

Among the maximally singular elements in $\overline{\ga}_+$ we would now like to isolate a subclass which enjoys certain extremal properties. We will only be able to find such a nice class in types $A_n$, $B_n$, $C_n$, $D_n$, and $E_7$, so for the rest of this subsection we shall assume that $\Phi_{\text{red}}$ is of one of these types. We say that $H_0\in\overline{\ga}_+$ is \textit{extremal} if $H_0$ is equal to a positive multiple of an \textit{extremal} fundamental coweight. The latter are defined via the following table, which identifies the node (or nodes) of the Dynkin diagram of $\Phi_{\text{red}}$ which corresponds to the extremal fundamental coweights:

\begin{table}[H]
 \begin{tabular}{|c|c|c|}
 \hline
 type of $\Phi_{\text{red}}$ & extremal nodes $\bullet$ & type of $\Phi_{M, \text{red}}$ \\ \hline
  $A_n \, (n\ge 1)$ & $\begin{dynkinDiagram}[mark=o]A{} 
\dynkinRootMark{*}1 \dynkinRootMark{*}4 \end{dynkinDiagram}$ &$A_{n-1}$\\ 
  $B_n \, (n\ge 2)$ & $\begin{dynkinDiagram}[mark=o]B{} 
\dynkinRootMark{*}1 \end{dynkinDiagram}$ & $B_{n-1}$\\
  $C_n \, (n\ge 3)$ & $\begin{dynkinDiagram}[mark=o]C{} 
\dynkinRootMark{*}1 \end{dynkinDiagram}$ & $C_{n-1}$\\
  $D_n \, (n\ge 4)$ & $\begin{dynkinDiagram}[mark=o]D{} 
\dynkinRootMark{*}1 \end{dynkinDiagram}$ & $D_{n-1}$\\
$E_7$ & $\begin{dynkinDiagram}[mark=o]E{7} 
\dynkinRootMark{*}7 \end{dynkinDiagram}$ & $E_6$\\
\hline
 \end{tabular}
 \bigskip
 \caption{\label{table1}Darkened nodes correspond to extremal fundamental coweights}
\end{table}

For the infinite families $A_n, B_n, C_n, D_n$, the extremal nodes are the ones that you can remove from the Dynkin diagram to obtain a Dynkin diagram in the same infinite family but of rank one less (with the conventions $B_1 = A_1, C_2 = B_2, D_3 = A_3$).


\subsection{Benjamini--Schramm convergence}\label{sec:BS}
A lattice $\Gamma < G$ is a discrete subgroup of finite covolume. The quotient $Y=\Gamma \backslash X$ is a Riemannian locally symmetric space of non-positive curvature. A lattice $\Gamma$ is said to be \textit{uniform} if $Y$ is compact and \textit{irreducible} if the projection of $\Gamma$ onto each simple factor is dense. We will generally assume that $\Gamma$ is uniform, irreducible, and torsion-free.

The injectivity radius of a point $y \in Y$, denoted $\textnormal{InjRad}_Y(y)$, is the supremum of all $r$ such that the ball of radius $r$ centered at any lift $\tilde{y} \in X$ of $y$ maps injectively to $Y$ under the canonical projection $X \to Y$. The (global) injectivity radius of $Y$, denoted $\textnormal{InjRad}(Y)$, is the infimum over all $y \in Y$ of $\textnormal{InjRad}_Y(y)$. We define
\begin{gather*}
    Y_{\leq R} := \{y \in Y: \textnormal{InjRad}_{Y}(y) \leq R\}.
\end{gather*}

Suppose $Y_n$ is a sequence of locally symmetric spaces obtained from quotienting $X$ by a sequence of torsion-free lattices. We say that the sequence $Y_n$ \textit{Benjamini--Schramm converges} to $X$ if, for every $R > 0$, we have
\begin{gather*}
    \frac{\textnormal{vol}((Y_n)_{\leq R})}{\textnormal{vol}(Y_n)} \to 0.
\end{gather*}
We say that the sequence $Y_n$ is \textit{uniformly discrete} if there is a universal non-zero lower bound on $\textnormal{InjRad}(Y_n)$.

\subsection{Plancherel convergence}\label{sec:limit-mult}
Let $\widehat{G}$ denote the unitary dual of $G$, consisting of isomorphism classes of irreducible unitary representations of $G$, and endowed with the Fell topology. Suppose $\Gamma$ is a cocompact lattice in $G$. Then
\begin{equation}\label{eq:spec-decomp}
    L^2(\Gamma \backslash G) = \bigoplus_{\pi \in \widehat{G}} m(\pi,\Gamma)\pi,
\end{equation}
where the multiplicity $m(\pi,\Gamma)=\dim\Hom_G(\pi,L^2(\Gamma\backslash G))$ is finite, and $m(\pi,\Gamma)=0$ for all but countably many $\pi$. We define the spectral measure on $\widehat{G}$ relative to $\Gamma$ as
\begin{gather*}
    \mu_\Gamma := \frac{1}{\textnormal{vol}(\Gamma \backslash G)} \sum_{\pi \in \widehat{G}} m(\pi,\Gamma) \delta_{\pi}.
\end{gather*}
An important component of the proof of Theorem \ref{main-theorem} is the limiting behavior of the spectral measure $\mu_\Gamma$ when $Y=\Gamma\backslash X$ converges Benjamini--Schramm to $X$.

Recall that we have fixed a Haar measure $dg$ on $G$. If $f \in C_c^\infty(G)$ and $\pi$ is an irreducible unitary representation of $G$, we define the trace class operator $\pi(f) := \int_{G} f(g) \pi(g) dg$. The Plancherel measure $d\mu_{\rm Pl}$ is the unique Radon measure on $\widehat{G}$ verifying the inversion formula $f(e)=\int_{\pi\in\widehat{G}}{\rm tr}\, \pi(f)d\mu_{\rm Pl}(\pi)$.

As our main interest is $L^2(\Gamma\backslash X)$ rather than $L^2(\Gamma\backslash G)$, we shall restrict our attention to irreducible unitary \textit{spherical} representations $\pi$ of $G$ occurring in \eqref{eq:spec-decomp}. An irreducible representation is said to be spherical if the space of $K$-invariant vectors is non-zero, in which case it is one-dimensional. Any such representation can be realized as the unique spherical subquotient $\pi_\lambda$ of the (unitarily normalized) principal series representation ${\rm Ind}_P^G\chi_\lambda$, where $\lambda\in\ga_\C^*$ and $\chi_\lambda$ is the character $\exp(\lambda(\mathcal{H}(p))$ of the minimal parabolic $P$. Since $\pi_\lambda\simeq \pi_{\lambda'}$ if and only if there is $w\in W$ with $\lambda=w\lambda'$, the map sending $\pi_\lambda$ to $\lambda$ descends to an injective map from the spherical unitary dual  $\widehat{G}^{\rm sph}$ to $\ga_\C^*/W$. Let the image of this injection be denoted $\ga^*_{\text{un}}/W$; it contains all of $\ga^*/W$. By comparison, for any $\lambda\in\ga^*_\C$, the matrix coefficient (relative to the inner product given by integration over $K$, which is unitary only for $\lambda\in\ga^*/W$) of any unit vector in $\pi_\lambda^K$ recovers the Harish-Chandra integral expression for the spherical function in \eqref{eqn_spherical_fn}. Without abuse, we may then speak of $\lambda$ as the \textit{spectral parameter} for $\pi_\lambda$. In this parametrization, the restriction of the Plancherel measure to the unitary spherical representations has density function equal to $|c(\lambda)|^{-2}$ on $\ga^*/W$ and is identically zero outside of this locus. Moreover, if $\lambda\in\ga^*_{\text{un}}/W$ and $f \in C_c^\infty(G /\!\!/ K)$, then $\textnormal{tr}\,\pi_\lambda(f)=\widehat{f}(\lambda)$, the Harish-Chandra spherical transform of $f$ evaluated at $\lambda$, as defined in \eqref{eq:HC-transform}.

We have the following result \cite{Deitmar_18, 7_samurai}.
\begin{theorem} \label{thm_bs_plancherel_conv}
    Suppose $\Gamma_n$ is a sequence of cocompact, uniformly discrete lattices in $G$. Then the following are equivalent:
    \begin{enumerate}
        \item The sequence of locally symmetric spaces $\Gamma_n \backslash X$ Benjamini--Schramm converges to $X$.
        \item For every $f \in C_c^\infty(G)$
        \begin{gather}
            \int_{\widehat{G}} \textnormal{tr}\,\pi(f)\, d \mu_{\Gamma_n}(\pi) \to \int_{\widehat{G}} \textnormal{tr}\,\pi(f)\, d \mu_{\rm Pl}(\pi). \label{eqn_weak_conv}
        \end{gather}
    \end{enumerate}
\end{theorem}

For $\lambda\in\ga^*_{\rm un}$ let $L^2(\Gamma\backslash X)_\lambda$ denote the $K$-invariant subspace of the $\pi_\lambda$-isotypic component of $L^2(\Gamma\backslash G)$. Taking $K$-invariants in \eqref{eq:spec-decomp}, we obtain an orthonormal basis $\mathcal{B}_\Gamma=\{\psi_j\}_{j \geq 0}$ of $L^2(\Gamma\backslash X)$ such that $\psi_j\in L^2(\Gamma\backslash X)_{\lambda_j}$. The functions $\psi_j$ are the Maass forms introduced in Section \ref{sec:main-result} and $\lambda_j$ is their spectral parameter. For a $W$-invariant subset $\Omega\subset\ga^*/W$ recall from the statement of Theorem \ref{main-theorem} that $N(\Omega,\Gamma)=\sum_{\lambda_j\in\Omega}\dim L^2(\Gamma\backslash X)_{\lambda_j}$.

The arguments of \cite[\S 6]{Peterson_23}, with slight modifications, imply the following result. 

\begin{proposition} \label{prop_spectral_conv}
    Suppose $\Gamma_n$ is a sequence of cocompact uniformly discrete torsion-free lattices in $G$ such that the corresponding sequence of locally symmetric spaces Benjamini--Schramm converges to $X$. Let $\Omega \subset \ga^*/W$ be a bounded measurable subset such that $\mu_{\rm Pl}(\partial \Omega) = 0$. Then
    \begin{gather*}
        \Big| \frac{N(\Omega, \Gamma_n)}{\textnormal{vol}(Y_n)} - \mu_{\rm Pl}(\Omega) \Big| \to 0
    \end{gather*}
    as $n \to \infty$.
\end{proposition}

For our purposes, a non-zero lower bound on $\frac{N(\Omega, \Gamma_n)}{\text{vol}(Y_n)}$ ultimately suffices. Thus, as in the statement of Theorem \ref{main-theorem}, it suffices to assume that $\Omega$ is compact with non-empty interior, as in this case it contains an open ball.

In the deduction of Proposition \ref{prop_spectral_conv} from Theorem \ref{thm_bs_plancherel_conv}, the main difficulty is that the indicator function of $\Omega$ is not of the form $\widehat{f}$ for some $f \in C_c^\infty(G /\!\!/ K)$. The latter class of functions may be identified, using the Harish-Chandra Paley--Wiener theorem, with the $W$-invariant functions in the Paley--Wiener space $\mathcal{PW}(\ga^*_{\mathbb{C}})$. All such functions restricted to $\ga^*_{\text{un}}/W$ vanish at infinity and thus we may use the Stone--Weierstrass theorem (together with Urysohn's lemma) to approximate $\mathds{1}_\Omega$ on $\ga^*_{\textnormal{un}}/W$ by elements in $\mathcal{PW}(\ga^*_{\mathbb{C}})^W$. From here the arguments in \cite{Peterson_23}, in particular the proofs of Proposition 6.5 and Theorem 1.9, can be repeated.

\subsection{Spectral gap}\label{sec:SG}
If $(\pi, V_{\pi})$ is a (not necessarily irreducible) separable unitary representation of $G$, then we define its integrability exponent by
\begin{gather*}
    q(\pi) := \inf\{q \ge 2 : \langle \pi(g) v_1, v_2 \rangle \in L^q(G) \textnormal{ for $v_1, v_2$ in a dense subspace of $V_{\pi}$}\}.
\end{gather*}
We say that $(\pi, V_{\pi})$ has a spectral gap if $q(\pi) < \infty$, and that a family of unitary representations $\{(\pi_n, V_{\pi_n})\}$ has a uniform spectral gap if we can find a uniform upper bound on $q(\pi_n)$. We call the representation \textit{tempered} if $q(\pi)=2$. 

\section{Spectral estimate}\label{sec:spectral-estimate}

In this section, we establish an upper bound on the spectral sums appearing in \eqref{eq:main-estimate} by averaging over an expanding spherical shell. As in the classical proof of quantum ergodicity, the insertion of this time dependent averaging operator, or wave propagator, will allow us in later sections to use their ergodic properties in the presence of a spectral gap.

We now drop the dependence on $n$ in the subscripts, writing $\Gamma<G$ for a cocompact lattice in $G$, $Y=\Gamma\backslash X$, and $\psi_j$ for the orthonormal basis of Maass forms on $Y$. 

\subsection{Averaging set}\label{sec:av-set}

We fix a simple factor $G_1$ and let $\mathcal{G}_2=G_2\times\cdots \times G_s$ denote the product of the remaining simple factors (if any). Let $K_1$ (resp. $\mathcal{K}_2$) and $A_1$ (resp. $\mathcal{A}_2$) denote the images of $K$ and $A$ inside $G_1$ (resp. $\mathcal{G}_2$). We decompose $\ga$ as $\ga_1\oplus {\rm Lie}(\mathcal{A}_2)$, where $\ga_1={\rm Lie}(A_1)$. For $r>0$ and $H\in\ga_1$ let $B_{\ga_1}(H,r)$ denote the Euclidean ball in $\ga_1$ of radius $r$ and centered at $H$. Similarly, $\mathcal{B}_2(0,r)$ denotes the Euclidean ball in ${\rm Lie}(\mathcal{A}_2)$ of radius $r$, centered at $0$.

Let $H_0\in\overline{\ga}_{1,+}$ be non-zero and $\epsilon_0,t>0$. In practice, $t$ will be large and tending to $\infty$, $\epsilon_0$ will be sufficiently small but fixed. We shall sometimes refer to $H_0$ as the \textit{directing element}; it will be fixed and all implied constants will depend on it. We let
\begin{equation}\label{eq:def-St}
S_t= K_1\exp(B_{\ga_1}(tH_0,\epsilon_0))K_1\subset G_1 \quad\textrm{and}\quad B=\mathcal{K}_2 \exp( \mathcal{B}_2(0,\epsilon_0)) \mathcal{K}_2\subset \mathcal{G}_2.
\end{equation}
We will propagate eigenfunctions along the subset
\begin{equation}\label{eq:StimesB}
E_t = S_t\times B\subset G
\end{equation}
using the right-regular representation $\varrho_{\Gamma\backslash G}$ of $G$ on $L^2(\Gamma\backslash G)$. Note that, for any function $f\in L^1(G)$, the adjoint of $\varrho_{\Gamma\backslash G}(f)$ is given by $\varrho_{\Gamma\backslash G}(f^\vee)$, where $f^\vee(g) := \overline{f(g^{-1})}$. We may then define an operator $U_t$ and its adjoint $U_t^*$ on $L^2(\Gamma\backslash G)$ by
\[
U_t := \varrho_{\Gamma\backslash G}(e^{-t \rho(H_0)} \mathds{1}_{E_t}) \quad\textrm{and}\quad U_t^* := \varrho_{\Gamma\backslash G}(e^{-t \rho(H_0)}\mathds{1}_{E_t^{-1}}).
\]
Note that, for $t$ sufficiently large,
\[
\textnormal{vol}(E_t)\asymp \int_{B_{\ga_1}(tH_0,\epsilon_0)\cap \ga_{1,+}} J(H)dH\asymp e^{2t \rho(H_0)},
\]
where we have used \eqref{eq:Cartan-int} and \eqref{eq:def:J}, and where $J(H)$ here denotes the radial volume factor for $G_1$. It follows that the normalization factor of $e^{-t \rho(H_0)}$ in the definition of $U_t$ is essentially the square-root of the volume of $E_t$. Observe furthermore that the propagation takes place solely within $G_1$, as the ball $B\subset\mathcal{G}_2$ is independent of the parameter $t$. This will eventually allow us to analyze the ergodic properties of $U_t$ assuming only that $G_1$, and not necessarily the other simple factors, satisfies the root data constraints of Theorem \ref{main-theorem}.

We have elected to suppress the dependency in the notation for the sets $E_t, S_t, B$, as well as for the operator $U_t$, on $H_0$ and $\epsilon_0$. While this lightens the notational load, we remark that in some places, such as in Theorems \ref{main-spec-thm} and \ref{main-geom-thm}, it will be necessary to take $\epsilon_0$ sufficiently small. More importantly, it will be crucial in Section \ref{sec_int_vol}, wherein we prove Theorem \ref{intersection-vol},  that $H_0$ is chosen to be extremal.

Some comments are in order on the relation between the sets $E_t$ and the sets used for similar purposes elsewhere in the literature. The averaging operators used in the work of Le Masson--Sahlsten \cite{Le_Masson_Sahlsten_17} were defined by suitably normalized hyperbolic balls. In the higher rank settings of Brumley--Matz \cite{Brumley_Matz_23} and Peterson \cite{Peterson_23}, ``polytopal ball'' averaging operators were used instead. The operator $U_t$ introduced above dispenses with the polytopes, but preserves their core features, by averaging only in a small neighborhood of the directing element $H_0$.  The operator $U_t$, and the choice of $H_0$ implicit in its definition, plays a similar role for the real Lie group $G$ that the normalized Hecke operator
\[
q^{-\langle \mu, \rho \rangle} \mathds{1}_{\mathbf{G}(\mathcal{O}) \varpi^\mu \mathbf{G}(\mathcal{O})}
\]
associated with a choice of dominant cocharacter $\mu$, plays on an algebraic group $\mathbf{G}$ over a non-archimedean local field (with ring of integers $\mathcal{O}$, residue field  of order $q$, and uniformizer $\varpi$). Note that, similarly to the archimedean setting, $\textnormal{vol}(\mathbf{G}(\mathcal{O}) \varpi^{\mu} \mathbf{G}(\mathcal{O})) \asymp q^{\langle \mu, 2 \rho \rangle}$. 

Ultimately, the fundamental problem, which we have already emphasized in Section \ref{sub_sec_int_vol}, is to choose a directing element $H_0$ for which the corresponding intersection volumes are minimized, as in Section \ref{sec_int_vol}. By contrast, the results we prove in this section are valid more generally, with significantly fewer constraints on the group $G$ and the directing element $H_0$. 

\subsection{Main spectral theorem and reduction to a local integral}\label{sec:reduction2local}
Let $a$ be a bounded measurable function on $Y$. When viewed as a right-$K$-invariant function on $\Gamma \backslash G$, the function $a$ determines a multiplication operator on $L^2(\Gamma \backslash G)$. We consider the time average
\begin{equation}\label{def:propagator}
\mathbf{A}(\tau)=\frac{1}{\tau}\int_{\tau}^{2\tau} U_t a U_t^* dt.
\end{equation}
In line with our convention for $U_t$, we have suppressed the dependency on the parameter $\epsilon_0$, as well as the directing element $H_0$, from the notation $\mathbf{A}(\tau)$.

The aim of this section is to prove the following result. 

\begin{theorem}[Spectral estimate]\label{main-spec-thm} Let $G$ be a product of non-compact simple real Lie groups, and $\Gamma<G$ an irreducible lattice. As in Section \ref{sec:av-set}, we fix a simple factor $G_1$ of $G$ and let $H_0\in \overline{\ga}_{1,+}$ be non-zero. Then there exists a finite $W$-stable set of hyperplanes $\{ P_i\}$ in $\ga^*$, depending on $H_0$, such that the following holds. Let $\Omega \subset \ga^*\smallsetminus\cup_i P_i$ be compact and $W$-invariant. There are constants $c, \tau_0, \epsilon_0 > 0$, depending on $\Omega$, such that for all $\tau \geq \tau_0$ and all $a\in L^\infty(Y)$ we have 
\[
\sum_{j: \lambda_j\in\Omega} |\langle a\psi_j,\psi_j\rangle|^2\le c\sum_{j:\lambda_j\in\Omega}\Big|\langle \mathbf{A}(\tau)\psi_j,\psi_j\rangle \Big|^2.
\]
Here, $\mathbf{A}(\tau)$ is defined relative to the parameters $\epsilon_0>0$ and $H_0$. 
\end{theorem}

We begin by reducing Theorem \ref{main-spec-thm} to a purely local statement (independent of $\Gamma$),  involving only $G_1$. Using the notation introduced in Section \ref{sec:av-set}, we have $\ga^*=\ga_1^*\oplus {\rm Lie}(\mathcal{A}_2)^*$. We may decompose $\lambda\in\ga^*$ as $\lambda=\lambda_1+\lambda_2$ according to this decomposition. Since $U_t^*$ acts on Maass forms $\psi_\lambda$ of spectral parameter $\lambda$ by the scalar
\[
e^{-t \rho(H_0)} \widehat{\mathds{1}_{E_t^{-1}}}(-\lambda) = e^{-t \rho(H_0)}\widehat{\mathds{1}_{E_t}}(\lambda) = e^{-t \rho(H_0)}\widehat{\mathds{1}_{S_t}}(\lambda_1)\widehat{\mathds{1}_B}(\lambda_2),
\]
we find
\begin{align}
\langle \mathbf{A}(\tau)\psi_\lambda,\psi_\lambda\rangle&=\frac{1}{\tau}\int_{\tau}^{2\tau} \langle U_t a U_t^*\psi_\lambda,\psi_\lambda\rangle dt\nonumber\\
&=\frac{1}{\tau}\int_{\tau}^{2\tau} \langle a U_t^*\psi_\lambda, U_t^*\psi_\lambda\rangle dt\nonumber\\
&= |\widehat{\mathds{1}_B}(\lambda_2)|^2 \left(\frac{1}{\tau}\int_{\tau}^{2\tau} e^{-2t \rho(H_0)}|\widehat{\mathds{1}_{S_t}}(\lambda_1)|^2 dt\right)\langle a\psi_\lambda,\psi_\lambda\rangle.\label{local factors}
\end{align}
Here, $\widehat{\mathds{1}_E}$ denotes the Harish-Chandra transform \eqref{eq:HC-transform}, for the group $G_1$ or $\mathcal{G}_2$ as appropriate, of the characteristic function of a subset $E$.

Let $\Omega_2\subset {\rm Lie}(\mathcal{A}_2)^*$ be compact and invariant under the Weyl group for $\mathcal{G}_2$. An elementary argument using compactness, to be given below, shows that $\widehat{\mathds{1}_B}(\lambda_2)$ is bounded away from zero, uniformly for $\lambda_2\in\Omega_2$. 

\begin{lemma}\label{lemma:ball}
Let $G$ be a product of non-compact simple real Lie groups. Let $\Omega\subset\ga^*$ be compact and $W$-invariant. Then there are constants $\epsilon_0,c>0$ such that the characteristic function $\mathds{1}_B$ of $B=K\exp(B_\ga(0,\epsilon_0))K$ satisfies $|\widehat{\mathds{1}_B}(\lambda)|\geq c$ uniformly for $\lambda\in\Omega$. 
\end{lemma}

\begin{proof}
As usual, let $\varphi_\lambda$ denote the spherical function on $G$. For all $\lambda\in\ga^*_{\C}$, we have $\varphi_\lambda(e)=1$; moreover, the map
\[
{\rm Re}\,\varphi: \ga\times \ga^*\rightarrow \R,\qquad (H,\lambda)\mapsto {\rm Re}\,\varphi_{-\lambda}(e^H)
\]
is continuous. Thus, for any $\lambda\in\ga^*$ we may find a neighborhood $U_\lambda\times V_\lambda\subset \ga\times\ga^*$ of $(0,\lambda)$ on which ${\rm Re}\, \varphi>1/2$, say. By the compactness of $\Omega$ there exists a finite subcover $\{V_{\lambda_j}\}_{j=1}^N$ of $\{V_\lambda: \lambda\in\Omega\}$; let $U=\bigcap_{j=1}^NU_{\lambda_j}$. Then ${\rm Re}\,\varphi_{-\lambda}(e^H)> 1/2$ for all $H\in U$ and $\lambda\in\Omega$. We can therefore choose $\epsilon_0>0$ sufficiently small so that
\[
{\rm Re}\left(\widehat{\mathds{1}_B}(\lambda)\right) = b_G \int_{B_{\ga}(0,\epsilon_0) \cap \ga_+}  {\rm Re}\,\varphi_{-\lambda}(e^H) J(H)dH >\frac{ b_G}{2} \int_{B_{\ga}(0,\epsilon_0) \cap \ga_+}  J(H) dH > 0,
\]
for all $\lambda \in \Omega$, as desired.
\end{proof}

We are therefore reduced to proving uniform lower bounds for the integral over $t$ in \eqref{local factors}. We note that Lemma \ref{lemma:ball} imposes no further conditions on $\Omega_2$, beyond compactness and Weyl group invariance. In view of the factorization of the scalar factors in \eqref{local factors} according to the components $\lambda_1$ and $\lambda_2$, this implies that the hyperplanes $P_i\subset\ga_1^*\oplus {\rm Lie}(\mathcal{A}_2)^*$ that $\Omega$ must avoid in Theorem \ref{main-spec-thm} can, and will, be taken to contain ${\rm Lie}(\mathcal{A}_2)^*$. Their exact nature is described in the following result (with $G$ playing the role of the privileged factor $G_1$), which will then be enough to prove (a more precise form of) Theorem \ref{main-spec-thm}.

\begin{prop}\label{prop:local-lower-bd}
Let $G$ be a non-compact simple real Lie group. Let $H_0\in\overline{\ga}_+$ be non-zero, with centralizer $M$. Fix a compact and $W$-invariant subset
\begin{equation}\label{eq:Omega-condition}
\Omega\subset\ga^*\smallsetminus \ga^*_{\rm bad},\qquad \ga^*_{\rm bad}:=\bigcup_{w\notin W_M} W\{\lambda\in\ga^*:  \lambda(H_0- wH_0)=0 \big\}.
\end{equation}
Then there is $\epsilon_0>0$ (implicit in the definition of $S_t$) as well as constants $c,\tau_0>0$ such that 
\[
\frac{1}{\tau} \int_{\tau}^{2\tau} e^{-2t \rho(H_0)} |\widehat{\mathds{1}_{S_t}}(\lambda)|^2 dt > c
\]
for all $\lambda\in \Omega$ and all $\tau \geq \tau_0$. 
\end{prop}

The basic strategy of the proof of Proposition \ref{prop:local-lower-bd} is to use the asymptotic for the spherical function $\varphi_\lambda$ provided by Proposition \ref{prop_hc_expansion} along the direction determined by $H_0$ to reduce the estimate to a weighted sum of complex exponentials. For the off-diagonal terms, one then needs to provide a uniform bound in $\tau$ of the quantity
\[
\max_{\substack{w, w' \in W/W_M\\ w\neq w' \textrm{ in }  W/W_M }}\bigg|\int_{\tau}^{2\tau} e^{-it \lambda (wH_0-w'H_0)}  dt\bigg|.
\]
Such a uniform bound will reflect the oscillation of the integrand, provided the phases $\lambda (wH_0-w'H_0)$ are bounded away from zero. The latter condition is ensured by the hypothesis on $\Omega$.

\subsection{Reduction to complex exponentials}
We now reduce the proof of Proposition \ref{prop:local-lower-bd} to a corresponding lower bound, stated in Proposition \ref{prop:int-W-sum} below, in which the variation in the $t$-parameter is expressed solely through complex exponentials. The key input is the next lemma, which uses the asymptotic formula for $\varphi_\lambda$ provided by Proposition \ref{prop_hc_expansion} as a critical ingredient.

\begin{lemma}\label{lemma:expansion}
Let $G$ be a non-compact simple real Lie group. Let $\Omega\subset\ga_{\rm reg}^*$ be compact and $W$-invariant. Let $H_0\in\overline{\ga}_+$ with centralizer $M$.  Let $X_M = M/K_M$ be the globally symmetric space associated to $M$.  Put $C_M = (b_G/b_M) 2^{-\ell}$, where $\ell = \sum_{\alpha \in \Phi^+ \smallsetminus \Phi_M^+} m_\alpha$ and $b_L$ is defined in \eqref{L-int-formula}. 

For any $\epsilon_0>0$ (implicit in the definition of $S_t$) there is $t_0> 0$ such that for all $\lambda \in \Omega$ and all $t\geq t_0$ we have
    \begin{gather*}
e^{-t \rho(H_0)} \widehat{\mathds{1}_{S_t}}(\lambda) = C_M \sum_{w \in W / W_M }c^M( - w\lambda)\widehat{\mathds{1}_{B^M}}(w\lambda)   e^{-itw\lambda( H_0)} +O( \epsilon_0^{ \dim X_M + 1} ),
    \end{gather*} 
 where $B^M=K_M\exp(B_\ga(0,\epsilon_0))K_M \subset M$.  Here the implied constant depends only on $\Omega$ and $H_0$.
 
\end{lemma}

\begin{proof}
It follows from the definition \eqref{eq:HC-transform} that
    \begin{gather*}
e^{-t \rho(H_0)} \widehat{\mathds{1}_{S_t}}(\lambda)= b_G e^{-t \rho(H_0)} \int_{B_\ga(t H_0,\epsilon_0)_+} \varphi_{-\lambda}(e^H) J(H) dH,
    \end{gather*}
where $B_\ga(t H_0,\epsilon_0)_+=B_\ga(t H_0,\epsilon_0)\cap\ga_+$. Since $\Omega$ avoids the singular locus, we may apply Proposition \ref{prop_hc_expansion} to $\varphi_{-\lambda}$, with $L=M$, to get
\begin{equation}\label{eq:asymp-hatSt}
e^{-t \rho(H_0)} \widehat{\mathds{1}_{S_t}}(\lambda)=\sum_{w \in W / W_M} c^M(-w\lambda) I^M(w\lambda) +O(\mathcal{E}_t),
  \end{equation}
 the implied constant depending only on $\Omega$, where
  \[
  I^M(\lambda)= b_G \int_{B_\ga(t H_0,\epsilon_0)_+} \left(e^{-t \rho(H_0)-\rho^M(H)}J^M(H)\right) \varphi_{-\lambda}^M(e^H) J_M(H)dH,
  \]
with $\rho^M(H)=\rho(H)-\rho_M(H)$, the factor $J_M(H)$ defined as in \eqref{eq:def:JM}, and
\[
\mathcal{E}_t= \int_{B_\ga(t H_0,\epsilon_0)_+} \left(e^{-t \rho(H_0)-\rho(H)}J(H)\right)e^{-2\beta_M(H)}(1+\|H\|)^s  dH,
\]
with $\beta_M(H)$ defined in \eqref{defn:beta} and $s$ as in Proposition \ref{prop_hc_expansion}.

We now let $S_t^M=e^{tH_0}B^M=K_M\exp(B_\ga(tH_0,\epsilon_0))K_M$ be the ball of radius $\epsilon_0$ about $e^{t H_0}$ in $M$. We claim that for any $\epsilon_0>0$ there is $t_0>0$ such that for all $t\geq t_0$ and all $\lambda \in \Omega$, we have
\begin{equation}\label{main-and-error}
  I^M(\lambda) = C_M\widehat{\mathds{1}_{S_t^M}}(\lambda)+O(\epsilon_0^{ \dim X_M + 1}) \qquad\textrm{and}\qquad \mathcal{E}_t=O(\epsilon_0^{ \dim X_M + 1}).
\end{equation}
Inserting this into \eqref{eq:asymp-hatSt}, we recover the statement of the lemma. Indeed, $\max_{\lambda\in\Omega}|c^M( -\lambda)|\ll 1$, since $-\Omega$ is bounded away from the polar hyperplanes of $c^M$, and
\[
        \widehat{\mathds{1}_{S_t^M}}(\lambda) =  \int_M \mathds{1}_{B^M}(e^{-tH_0}m)\varphi_{-\lambda}^M(m) dm=\int_M \mathds{1}_{B^M}(m)\varphi_{-\lambda}^M(e^{tH_0}m) dm= e^{-it\lambda(H_0)} \widehat{\mathds{1}_{B^M}}(\lambda).
  \]
In the last equality, we have used the fact, which follows from \eqref{eqn_spherical_fn}, that $\varphi^M_{\lambda}(e^Z m) = e^{i \lambda(Z)} \varphi^M_{\lambda}(m)$ whenever $Z \in \ga$ centralizes $M$. 

For the main term identity in \eqref{main-and-error}, we begin by observing that the parenthetical expression in the definition of $I^M(\lambda)$ is asymptotically constant on the support of the integral:
\begin{equation}\label{eq:normalization}
e^{-t \rho(H_0)-\rho^M(H)}J^M(H)=2^{-\ell}(1+O(\epsilon_0)),\qquad \forall\, H \in B_\ga(tH_0,\epsilon_0),\;\forall\, t \ge t_0.
\end{equation}
To see this, we use $\alpha(H_0) > 0$ for $\alpha \in \Phi^+ \smallsetminus \Phi_M^+$ to show that, for such $H$:
\[
J^M(H) =2^{-\ell}\prod_{\alpha \in \Phi^+ \smallsetminus \Phi_M^+} \big(e^{\alpha(H)} - e^{-\alpha(H)}\big)^{m_\alpha} = 2^{-\ell}e^{2 \rho^M(H)} +  O(e^{2t \rho^M(H_0) - 2t \beta_M(H_0) }),
\]
where the quantity $\beta_M(H_0) > 0$ is defined in (\ref{defn:beta}).  If we choose $t_0$ such that $e^{ -2t_0 \beta_M(H_0)} < \epsilon_0$, this becomes $J^M(H) = 2^{-\ell} e^{2 \rho^M(H)} (1+O(\epsilon_0))$.  Combining this with
\[
e^{\rho^M(H)}=e^{t\rho^M(H_0)}e^{\rho^M(H-tH_0)}=e^{t\rho(H_0)}(1+O(\epsilon_0)),
\]
we obtain the claim \eqref{eq:normalization}, which may then be inserted into $I^M(\lambda)$ to yield
\[
I^M(\lambda) = C_M b_M \int_{B_\ga(t H_0,\epsilon_0)_+} \varphi_{-w\lambda}^M(e^H) J_M(H)dH + \int_{B_\ga(t H_0,\epsilon_0)_+} O(\epsilon_0) \varphi_{-\lambda}^M(e^H) J_M(H)dH.
\]
To estimate the error term, we may apply the bounds $| \varphi_{-\lambda}^M(e^H) | \le 1$, ${\rm vol}(B_\ga(t H_0,\epsilon_0)_+) \ll \epsilon_0^{\dim \ga}$, and $J_M(H) \ll \epsilon_0^{ | \Phi^+_M |}$ for $H \in B_\ga(t H_0,\epsilon_0)_+$ to see that it is $O( \epsilon_0^{ \dim \ga + | \Phi^+_M | + 1}) = O(\epsilon_0^{ \dim X_M + 1})$, as required.

To express the main term as $C_M\widehat{\mathds{1}_{S_t^M}}(\lambda)$, we fix $t_0$ large enough so that $B_\ga(tH_0,\epsilon_0)_+ = B_\ga(tH_0,\epsilon_0) \cap \ga_{M,+}$ for all $t > t_0$, where $\ga_{M,+}$ is defined in \eqref{aLdef}. In particular, the union $\bigcup_{w\in W_M}wB_\ga(t H_0,\epsilon_0)_+$ is disjoint, with closure $\overline{B}_\ga(t H_0,\epsilon_0)$. Since $J_M(wH)=J_M(H)$, $\varphi_{\lambda}^M(e^{wH}) = \varphi_{\lambda}^M(e^{H})$ for $w \in W_M$, the integration formula \eqref{L-int-formula} implies that
\[
b_M \int_{B_\ga(tH_0,\epsilon_0)_+}\varphi_{-\lambda}^M(e^H) J_M(H)dH = \widehat{\mathds{1}_{S_t^M}}(\lambda)
\]
as required.

To estimate the error term $\mathcal{E}_t$, we first apply \eqref{eq:def:J} to deduce that $e^{-t\rho(H_0)-\rho(H)}J(H)\ll 1$ on $B_\ga(tH_0,\epsilon_0)_+$. Moreover, given $\epsilon_0>0$, there exists $t_0>0$ such that $2\beta_M(H) \geq \tfrac{3}{2} t \beta_M(H_0)$ for all $t\geq t_0$ and $H\in B_\ga(tH_0,\epsilon_0)$. We deduce 
\[
\mathcal{E}_t= O\left(e^{-\tfrac{3}{2} t \beta_M(H_0)}\int_{B_\ga(tH_0,\epsilon_0)_+}(1+\|H\|)^sdH\right)=O(e^{-t \beta_M(H_0)}),
\]
by the polynomial growth in $t$ of the integral.  Choosing $t_0$ such that $e^{ -t_0 \beta_M(H_0)} < \epsilon_0^{\dim X_M + 1}$ completes the proof.\end{proof}

Lemma \ref{lemma:expansion} requires that the compact $\Omega$ avoid the singular locus. Proposition \ref{prop:local-lower-bd}, on the other hand, requires that $\Omega$ avoid $\ga^*_{\rm bad}$, a set of hyperplanes depending on $H_0$ defined in \eqref{eq:Omega-condition}. The following lemma makes use of our assumption that $H_0$ is non-zero to show an inclusion of the former in the latter.

\begin{lemma}\label{lemma:max-Levi}
Let $G$ be a non-compact simple real Lie group. Let $H_0\in\overline{\ga}_+$ be non-zero. Then $\ga^*_{\textnormal{sing}} \subseteq \ga^*_{\textnormal{bad}}$.
\end{lemma}

\begin{proof}
Since $W_M$ consists precisely of those elements of $W$ which fix $H_0$, $\ga^*_{\textnormal{bad}}$ is exactly the set of $\lambda \in \ga^*$ that vanish on a nonzero element of 
\begin{equation}\label{lines:3}
\bigcup_{w\notin W_M} W.\R(H_0- wH_0)=\bigcup_{w\in W} W\R.(H_0 - wH_0).
\end{equation}
Recalling the definition of $\ga^*_{\textnormal{sing}}$ in \eqref{eq:reg-sing}, we must therefore show that
\begin{equation}\label{eq:W-orbit-subset}
\bigcup_{w\in W} W\R.(H_0 - wH_0) \supseteq \bigcup_{\alpha\in\Phi}\R.\alpha^\vee.
\end{equation}
For $\alpha\in \Phi$ we let $s_\alpha\in W$ denote the orthogonal reflection in $\ga$ across $(\alpha^\vee)^\perp$. In particular,
\begin{equation}\label{eq:diff:weights}
H_0- s_\alpha H_0
 = 2\frac{\langle H_0,\alpha^\vee \rangle}{\langle\alpha^\vee,\alpha^\vee\rangle}\alpha^\vee.
\end{equation}
Note that the subset of all long roots in $\Phi$ spans $\ga^*$, as does the subset of all short roots. Since $H_0\neq 0$ there must therefore be a long root $\beta_1$ and a short root $\beta_2$ such that $\langle H_0,\beta_1^\vee \rangle\neq0$ and $\langle H_0,\beta_2^\vee \rangle \neq 0$. Therefore $\beta_1^\vee$ and $\beta_2^\vee$ are  contained in the left-hand side of \eqref{eq:W-orbit-subset}. Since $W$ acts transitively on the subset of all long roots in $\Phi$ and also on the subset of all short roots, all of $\Phi^\vee$ is therefore contained in the left-hand side of \eqref{eq:W-orbit-subset}.
\end{proof}

\subsection{Proof of Proposition \ref{prop:local-lower-bd}}\label{sec:spectral-proof}

Our goal is to prove Proposition \ref{prop:local-lower-bd} which, as we have seen in Section \ref{sec:reduction2local}, implies Theorem \ref{main-spec-thm}. From Lemmas \ref{lemma:expansion} and \ref{lemma:max-Levi}, it suffices to show the following result. We explain this implication in detail after the conclusion of the proof. 
\begin{prop}\label{prop:int-W-sum}
Let $G$ be a non-compact simple real Lie group. Let $H_0\in\overline{\ga}_+$ be non-zero. Let $\Omega\subset \ga^*\smallsetminus  \ga_{\rm bad}^*$ be compact and $W$-invariant.  There are constants $c,\tau_0>0$ such that for any sufficiently small $\epsilon_0>0$ (present in the definition of $B^M$) we have
\[
 \frac{1}{\tau} \int_{\tau}^{2\tau} \Big| C_M  \sum_{w \in W /  W_M} c^M(-w\lambda)\widehat{\mathds{1}_{B^M}}(w\lambda)e^{-it w\lambda(H_0)} \Big|^2 dt > c \epsilon_0^{2 \dim X_M}
\]
for all $\lambda \in \Omega, \tau \geq \tau_0$.
\end{prop}

\begin{proof}
We expand the square and swap the order of summation and integral to get the sum of the diagonal term
\[
D(\lambda):= \sum_{w \in W / W_M}  \Big|c^M(-w\lambda)\widehat{\mathds{1}_{B^M}}(w\lambda) \Big|^2, 
\]
which is independent of $\tau$, plus the off-diagonal contribution
\[
E(\tau,\lambda):=\sum_{\substack{w, w' \in W / W_M\\ w\neq w' \textrm{ in }\, W / W_M}} c^M(-w\lambda) \overline{c^M(-w'\lambda)}\widehat{\mathds{1}_{B^M}}(w\lambda)\overline{\widehat{\mathds{1}_{B^M}}(w'\lambda)}  \frac{1}{\tau} \int_{\tau}^{2\tau} e^{-it(w \lambda-w'\lambda) (H_0)}  dt.
\]
We claim that for any compact $\Omega \subset \ga^* \smallsetminus \ga^*_{\textnormal{bad}}$ we have
\begin{enumerate}
\item\label{point1} $D(\lambda) \gg \epsilon_0^{2 \dim X_M}$ for all $\lambda\in\Omega$, and 
\item\label{point2} $E(\tau, \lambda)\ll \tau^{-1} \epsilon_0^{2 \dim X_M}$ for all $\lambda\in\Omega$,
\end{enumerate}
with the implied constants depending on $\Omega$. Taking $\tau_0$ sufficiently large will guarantee that $E$ is negligible with respect to $D$ for $\tau \geq \tau_0$.

We begin by proving the stated lower bound on $D(\lambda)$. It will be sufficient to show that if $\epsilon_0>0$ is sufficiently small we have $c^M(-\lambda)\widehat{\mathds{1}_{B^M}}(\lambda) \gg \epsilon_0^{\dim X_M}$ for all $\lambda \in \Omega$. Recall that $c^M(-\lambda)$ is holomorphic and non-vanishing on the open subset containing $\ga_{\rm reg}^*$ described in \eqref{eq:cM-domain}. Since $\Omega\subset \ga^*_{\rm reg}$ and $\Omega$ is compact, we deduce that $c^M(-\lambda)$ is uniformly bounded away from zero on $\Omega$. Furthermore, we have
\[
|\widehat{\mathds{1}_{B^M}}(\lambda)|\ge |{\rm Re} \,\widehat{\mathds{1}_{B^M}}(\lambda)|=\bigg|\int_{B^M}   {\rm Re}\,\varphi_{-\lambda}^M(g) dg \bigg|.
\]
As in the proof of Lemma \ref{lemma:ball} there exists a neighborhood $U$ of the identity in $G$ such that  ${\rm Re}\,\varphi_{-\lambda}(g)> 1/2$ for all $g\in U$ and $\lambda\in\Omega$. We can therefore choose $\epsilon_0>0$ small enough in the definition of $B^M=K_M\exp(B_\ga(0,\epsilon_0))K_M$ so that, for all $\lambda \in \Omega$,
\[
\int_{B^M} {\rm Re}\,\varphi_{-\lambda}^M(g) dg >\frac{1}{2} {\rm vol}(B^M) \gg \epsilon_0^{\dim X_M}.
\]
We deduce that $\widehat{\mathds{1}_{B^M}}(\lambda) \gg \epsilon_0^{\dim X_M}$ for all $\lambda\in\Omega$, proving point \eqref{point1}.

We now address the off-diagonal term. Firstly, since $\lambda \in \ga^*$, we have that $| \widehat{\mathds{1}_{B^M}}(w\lambda) | \le {\rm vol}(B^M) \ll \epsilon_0^{\dim X_M}$. On the other hand $c^M(-w \lambda)$ is singular on a subset of $\ga^*_{\textnormal{sing}}$. Since, by hypothesis, $\Omega$ avoids $\ga^*_{\textnormal{bad}}$, it also avoids  $\ga^*_{\textnormal{sing}}$ by Lemma \ref{lemma:max-Levi}. Therefore for $\lambda \in \Omega$ we also have a uniform bound on $c^M(-w \lambda)$. 

We thus have
\[
E(\tau, \lambda)\ll \epsilon_0^{2 \dim X_M} \max_{\substack{w, w' \in W \\ wW_M\neq w'W_M}} \frac{1}{\tau} \bigg|\int_{\tau}^{2\tau} e^{-it \lambda (wH_0-w'H_0)}  dt\bigg|.
\]
Recall from \eqref{lines:3} that elements in $\ga^*_{\textrm{bad}}$ are precisely those that vanish on a nonzero element of
\begin{align*}
\bigcup_{w'\in W} W\R.(H_0 - w'H_0)&=\bigcup_{w,w'\in W} \R. w(H_0 -w^{-1} w'H_0)\\
&=  \bigcup_{w,w'\in W} \R. (wH_0 - w'H_0).
\end{align*}
We deduce that $\ga^*_{\textnormal{bad}}$ is precisely the locus where the exponential phase vanishes. Since $\lambda$ is confined to a compact $\Omega\subset \ga^*\smallsetminus \ga^*_{\textnormal{bad}}$, the phases appearing in the $t$-integral are all uniformly bounded away from zero. Thus, if $X$ is any one of the differences $wH_0-w'H_0$, then 
\[
\bigg|\int_{\tau}^{2\tau} e^{-it \lambda (X)}  dt\bigg|=|\lambda(X)|^{-1}\big|1-e^{-i\tau\lambda(X)}\big|\le 2|\lambda(X)|^{-1}\ll 1,
\]
which establishes point \eqref{point2}, and completes the proof of Proposition \ref{prop:int-W-sum}.
\end{proof}

We now return to the proof of Proposition \ref{prop:local-lower-bd}.  Define
\begin{align*}
    A_1(t, \lambda, \epsilon_0) &:= C_M \sum_{w \in W / W_M }c^M( - w\lambda)\widehat{\mathds{1}_{B^M}}(w\lambda)   e^{-itw\lambda( H_0)}, \\
    A_2(t, \lambda, \epsilon_0) &:= e^{-t \rho(H_0)} \widehat{\mathds{1}_{S_t}}(\lambda) - A_1(t, \lambda, \epsilon_0).
\end{align*}
Since $c^M(-w \lambda)$ is uniformly bounded on $\Omega$, and $|\widehat{\mathds{1}_{B^M}}(w\lambda)| = O(\epsilon_0^{\dim X_M})$ uniformly in $\lambda \in \Omega$, we get that $A_1(t, \lambda, \epsilon_0) = O(\epsilon_0^{\dim X_M})$ uniformly in $\lambda \in \Omega$ and $t$. On the other hand, we know that $A_2(t, \lambda, \epsilon_0)$ is $O(\epsilon_0^{\dim X_M + 1})$ uniformly for $\lambda \in \Omega$ for $t \geq t_0$ as in Lemma \ref{lemma:expansion}.
We have
\begin{align*}
    &\frac{1}{\tau} \int_{\tau}^{2\tau} e^{-2t \rho(H_0)} |\widehat{\mathds{1}_{S_t}}(\lambda)|^2 dt \\
    = &\frac{1}{\tau} \int_{\tau}^{2 \tau} |A_1(t, \lambda, \epsilon_0)|^2 dt + \frac{1}{\tau} \int_{\tau}^{2 \tau} \Big( 2 \textnormal{Re}\big(A_1(t, \lambda, \epsilon_0) \overline{A_2(t, \lambda, \epsilon_0)} \big) + |A_2(t, \lambda, \epsilon_0)|^2 \Big) dt
\end{align*}
The first term is lower bounded by $c \epsilon_0^{2 \dim X_M}$ by Proposition \ref{prop:int-W-sum}, where $c > 0$ is the constant appearing there. We can subsequently make $\epsilon_0$ small enough to make the second term smaller than $c \epsilon_0^{2 \dim X_M}/2$ in absolute value. Taking $t_0'$ to be the maximum of the $\tau_0$ from Proposition \ref{prop:int-W-sum} and the $t_0$ from Lemma \ref{lemma:expansion} for this choice of $\epsilon_0$, we obtain Proposition \ref{prop:local-lower-bd} with the constants $c \epsilon_0^{2 \dim X_M}/2$ and $t_0'$.

\subsection{Identifying $\ga_{\rm bad}^*$ for extremal $H_0$ in classical groups}\label{sec:hyperplanes}

In this section, we assume that $G$ is simple. We wish to describe the nature of $\ga^*_{\textnormal{bad}}$ in case the directing element $H_0 \in \overline{\ga}_+$ is extremal, as defined in Section \ref{sec:str-sing}.

Thus far we have only defined extremal for the classical types, as well as $E_7$, as these are the only root systems admitting semi-dense root subsystems. However, the spectral estimate (Theorem \ref{main-spec-thm}) holds in any type. In order to have a complete analysis for all types in the following proposition, we shall define extremal elements in case the reduced root system is of the remaining types $E_6, E_8, F_4$ or $G_2$ via the table below. In subsequent sections, the term extremal will again be reserved only for those elements specified in Section \ref{sec:str-sing}. 

\begin{table}[H]
 \begin{tabular}{|c|c|c|}
 \hline
 type of $\Phi_{\text{red}}$ & extremal nodes $\bullet$ & type of $\Phi_{M, \text{red}}$ \\ \hline
  $E_6$ & $\begin{dynkinDiagram}[mark=o]E{6} 
\dynkinRootMark{*}1 \dynkinRootMark{*}6 \end{dynkinDiagram}$ &$D_5$\\ 
  $E_8$ & $\begin{dynkinDiagram}[mark=o]E{8} 
\dynkinRootMark{*}8 \end{dynkinDiagram}$ & $E_7$\\
  $F_4$ & $\begin{dynkinDiagram}[mark=o]F{4} 
\dynkinRootMark{*}1 \dynkinRootMark{*}4 \end{dynkinDiagram}$ & $B_3$ or $C_3$\\
  $G_2$ & $\begin{dynkinDiagram}[mark=o]G{2} 
\dynkinRootMark{*}1 \dynkinRootMark{*}2 \end{dynkinDiagram}$ & $A_1$\\
\hline
 \end{tabular}
 \bigskip
 \caption{Extremal nodes for the remaining exceptional types}
\end{table}

We will show that, under the extremal hypothesis on $H_0$, the inclusion \eqref{eq:W-orbit-subset} is nearly an equality --- in fact, an exact equality in most cases. The basic principle is that the Weyl group $W_M$ of an extremal $H_0$ is large, which will tend to minimize the number of distinct lines $\R.(wH_0-w'H_0)\subset\ga$.

\begin{prop}\label{prop:Weyl-diff}
    Assume that $G$ is simple and let $H_0 \in \overline{\ga}_+$ be extremal.
    \begin{enumerate}
        \item If $\Phi_{\textnormal{red}}$ is of type $A_n, B_n$, $C_n$, or $G_2$ then $\ga^*_{\textnormal{bad}} = \ga^*_{\textnormal{sing}}$.
        \item If $\Phi_{\textnormal{red}}$ is of type $D_n$, then 
        \begin{gather*}
        \ga^*_{\textnormal{bad}} = \ga^*_{\textnormal{sing}} \cup W.(\mathbb{R}. \varpi_1^\vee),
        \end{gather*}
        where $\varpi_1^\vee$ is the unique extremal fundamental coweight. 
    \end{enumerate}
\end{prop}

\begin{proof}
We must show that
\begin{enumerate}
\item\label{Weyl-lemma2a} If $G$ is of type $A_n$, $B_n$, $C_n$ or $G_2$ then equality holds in \eqref{eq:W-orbit-subset};
\item\label{Weyl-lemma2b} If $G$ is of type $D_n$, then equality holds in \eqref{eq:W-orbit-subset} with the right-hand side replaced by
\[
\Big(\bigcup_{\alpha\in\Phi}\R\alpha^\vee \big) \cup W.(\R.\varpi^\vee_1)
\]
where $\varpi_1^\vee$ is the extremal fundamental coweight (a multiple of $H_0$).
\end{enumerate}
In all cases $H_0$ is, by assumption, a non-zero multiple of an extremal fundamental coweight, which we denote by $\varpi_0^\vee\in\widehat{\Delta}^\vee$. Without loss of generality we can simply take $H_0 = \varpi_0^\vee$.

We begin by assuming that $\Phi_{\textnormal{red}}$ is of type $A_n, B_n$, $C_n$, or $G_2$. For equality to hold in \eqref{eq:W-orbit-subset}, it is enough to show that for any $w \in W/W_M$, the difference $\varpi_0^\vee - w \varpi_0^\vee$ lies in the line spanned by a coroot. In fact, using the computation \eqref{eq:diff:weights}, we need only to show that every non-trivial class in $W/W_M$ can be represented by a reflection in $W$. First note that the reflections $s_\alpha$  for $\alpha\in \Phi_{\textnormal{red}}^+ \smallsetminus \Phi_{\textnormal{red},M}^+$ are pairwise distinct mod $W_M$. Thus, the image of the set $\{1\}\cup \{s_\alpha \mid  \alpha\in \Phi_{\textnormal{red}}^+ \smallsetminus \Phi_{\textnormal{red},M}^+ \}$ in $W/W_M$ has cardinality $|\Phi_{\textnormal{red}}^+ \smallsetminus \Phi_{\textnormal{red},M}^+| + 1$. For point \eqref{Weyl-lemma2a}, it will therefore suffice to check that
\[
 |W/W_M|
 = |\Phi_{\textnormal{red}}^+ \smallsetminus \Phi_{\textnormal{red},M}^+| + 1
\]
if $\Phi_{\textnormal{red}}$ is of type $A_n$, $B_n$, $C_n$, or $G_2$. For that we compute the following table:

\begin{table}[H]
 \begin{tabular}{|c|c|c|>{\bfseries} c|c|c|>{\bfseries}c|}
\hline  type of $G$ & $|W|$ & $|W_M|$ & $|W/W_M|$& $|\Phi_{\textnormal{red}}|$ & $|\Phi_{\textnormal{red},M}|$  & $|\Phi_{\textnormal{red}}^+ \smallsetminus \Phi_{\textnormal{red},M}^+|$ \\\hline
   $A_n$ & $(n+1)!$ & $n!$ &                            n+1&         $n(n+1)$ &  $(n-1)n$ & n\\
   $B_n$ & $2^n n!$ & $2^{n-1} (n-1)!$ &        2n &        $2n^2$ & $2(n-1)^2$ & 2n-1\\
   $C_n$ &  $2^n n!$ & $2^{n-1} (n-1)!$ &          2n &      $2n^2$ & $2(n-1)^2$ & 2n-1\\
   $D_n$ & $2^{n-1} n!$ & $2^{n-2} (n-1)!$ &        2n & $2n(n-1)$ & $2(n-1)(n-2)$ & 2n-2\\
   $E_6$ & 51840 & 1920                                       &27 & 72&   40 & 16\\
$E_7$ & 2903040 & 51840                                 &56 & 126 & 72 & 27 \\
   $E_8$ & 696729600 & 2903040                         &240& 240 &  126 & 57\\
   $F_4$ & 1152 & 48                                               &24&48 & 18 & 15\\
   $G_2$ & 12 & 2                                                 &6& 12& 2 & 5\\ 
   \hline
 \end{tabular}
 \end{table}
\noindent An examination of the boldfaced columns of the first, second, third, and last rows concludes the proof of point \eqref{Weyl-lemma2a}. 

We now take $G$ to have root system $D_n$. According to Table \ref{table1}, we may take as extremely singular element $H_0$ the highest fundamental coweight $\varpi_1^\vee$. To prove point \eqref{Weyl-lemma2b}, we must therefore show that, for any $w\in W/W_M$, the difference $\varpi_1^\vee-w\varpi_1^\vee$ is either in a line spanned by a coroot or by $\varpi_1^\vee=H_0$ itself. The $2n-1$ elements $w\in \{1\}\cup\{s_\alpha: \alpha \in \Phi_{\textnormal{red}}^+ \smallsetminus \Phi_{\textnormal{red},M}^+ \}$ have distinct images in the size $2n$ quotient $W/W_M$, and by \eqref{eq:diff:weights}, satisfy $\varpi_1^\vee-w\varpi_1^\vee\in \bigcup_{\alpha\in\Phi} \R\alpha^\vee$. The remaining coset in $W/W_M$, in the notation of paragraphs (IX) and (X) of Ch. VI, \S 4.8, of \cite{Bourbaki}, is represented by, say, $s_{12}$, which changes the sign of the first and second basis vectors $\varepsilon_1,\varepsilon_2$. Since $\varepsilon_1=\varpi_1^\vee$, this yields $\varpi_1^\vee-s_{12}(\varpi_1^\vee)=\varpi_1^\vee-(-\varpi_1^\vee)=2\varpi_1^\vee$, as required. This concludes the proof of point \eqref{Weyl-lemma2b}, and hence the proposition. \end{proof}

\section{Geometric estimate: statement and reduction steps}\label{sec:geom-reduction}

We deduce from Theorem \ref{main-spec-thm}, under the stated conditions on $\Omega, \tau$, and $a$, that
\[
\sum_{j: \lambda_j\in\Omega} |\langle a\psi_j,\psi_j\rangle|^2\ll \sum_{j,k}\big |\langle \mathbf{A}(\tau)\psi_j,\psi_k\rangle\big|^2=\|\mathbf{A}(\tau)\|_{\rm HS}^2,
\]
where the spectral sum has been extended by positivity, and the right-hand side is the Hilbert--Schmidt norm of the operator $\mathbf{A}(\tau)$. We have thus reduced our problem to bounding the Hilbert--Schmidt norm of $\mathbf{A}(\tau)$.

The remainder of the paper is organized around the proof of Theorem \ref{main-geom-thm}, stated below. 

\begin{theorem}[Geometric estimate]\label{main-geom-thm}
Let $G$ be a product of non-compact simple real Lie groups. Let $X = G/K$, where $K$ is a maximal compact subgroup, be the associated symmetric space. Let $\Gamma<G$ be a torsion free, cocompact, irreducible lattice and set $Y = \Gamma\backslash X$. Assume that $G$ admits a simple factor $G_1$ satisfying condition \eqref{thm:root-system} of Theorem \ref{main-theorem}. Let $H_0\in\overline{\ga}_{1,+}$ be an extremal coweight for $G_1$, identified with an element in $\overline{\ga}_+$ through $\ga_1\subset\ga$.

There are constants $c_1,c_2, c_3,\epsilon_0>0$, depending only on $G$, an integer $k\in\Z_{\geq 0}$, depending only on the reduced root system of $G_1$, and $\theta>0$, depending on the integrability exponent of $G_1$ acting of $L^2_0(\Gamma\backslash G)$, such that, for all $\tau \gg 1$ and all mean-zero functions $a\in L^\infty(Y)$,
\begin{equation}\label{eq:HS-bound}
\|\mathbf{A}(\tau)\|_{\rm HS}^2\ll \|a\|_2^2\frac{(\log\tau)^k}{\tau\theta^2}+\frac{e^{c_1 \tau}}{{\rm InjRad}(Y)^{\dim Y}} {\rm vol}\big( Y_{\le c_2 \tau+c_3}\big)\|a\|_\infty^2. 
\end{equation}
Here, $\mathbf{A}(\tau)$ is defined relative to the parameters $\epsilon_0>0$ and $H_0$.
\end{theorem}

The purpose of this section is to reduce the proof of Theorem \ref{main-geom-thm} to Theorem \ref{intersection-vol} from the introduction, applied to the intersection of the spherical shell $S_t\subset G_1$ and its translates in the irreducible symmetric space $X_1=G_1/K_1$. We carry out this reduction by a thick-thin decomposition of the kernel of the operator $\mathbf{A}(\tau)$ in Lemma \ref{lemma:general-HS}, a calculation of the support of this kernel in Corollary \ref{cor:intersection-range}, and, most notably, an invocation of the Nevo ergodic theorem (Proposition \ref{prop_nevo} below). The remaining sections of the paper will then be dedicated to the proof of Theorem \ref{intersection-vol}.

Together, Proposition \ref{prop_spectral_conv}, Theorem \ref{main-spec-thm}, and Theorem \ref{main-geom-thm} yield the estimate \eqref{eq:main-estimate}, completing the proof of Theorem \ref{main-theorem}, as we now explain. Let $\mathbf{A}_n(\tau)$ denote the operator defined in \eqref{def:propagator} with respect to the test function $a_n$. Inserting the aforementioned results, along with $\|a_n\|_2^2\ll {\rm vol}\big(\Gamma_n\backslash G)\|a_n\|_\infty^2$, $\|a_n\|_\infty=O(1)$, the uniformly discrete hypothesis on $\Gamma_n$, and the uniform spectral gap for $G_1$, we find, that for all $n$ and $\tau$ sufficiently large,
\begin{align}
\frac{1}{N(\Gamma_n,\Omega)}\sum_{\lambda_j^{(n)}\in\Omega} |\langle a_n\psi_j^{(n)},\psi_j^{(n)}\rangle|^2&\ll \frac{1}{{\rm vol}\big(\Gamma_n\backslash G)}\sum_{\lambda_j^{(n)}\in\Omega} |\langle a_n\psi_j^{(n)},\psi_j^{(n)}\rangle|^2\nonumber\\
&\ll \frac{1}{{\rm vol}\big(\Gamma_n\backslash G)}\|\mathbf{A}_n(\tau)\|_{\rm HS}^2\nonumber\\
&\ll \frac{(\log\tau)^k}{\tau}+e^{c_1\tau}\frac{{\rm vol}\big( (Y_n)_{\le c_2\tau +c_3}\big)}{{\rm vol}\big(Y_n)}. \label{eq:total-bound}
\end{align}
It remains to find a sequence of $\tau_n\to\infty$ such that the second term goes to zero. It follows from the Benjamini--Schramm convergence of $Y_n$ towards $X$ that there exists a sequence of $R_n$ tending toward infinity such that 
\[
\alpha_n:=\frac{{\rm vol}\big( (Y_n)_{\le R_n}\big)}{{\rm vol}(Y_n)}\rightarrow 0.
\]
Let $r_n$ be a sequence tending to infinity slowly enough so that both $r_n+c_3\leq R_n$ and $e^{c_1 r_n/c_2}\alpha_n\rightarrow 0$. Then, setting $\tau_n=r_n/c_2$, the second term in \eqref{eq:total-bound} is
\[
e^{c_1 r_n/c_2}\frac{{\rm vol}\big( (Y_n)_{\le r_n+c_3}\big)}{{\rm vol}(Y_n)}\leq e^{c_1r_n/c_2}\alpha_n\rightarrow 0,
\]
as desired.

\subsection{A general bound on Hilbert--Schmidt norms}

We begin with a very general upper bound on the Hilbert--Schmidt norm of an integral operator on $L^2(\Gamma\backslash G)$.

\begin{lemma}\label{lemma:general-HS}
Let $G$ be a product of non-compact simple real Lie groups. Let $\Gamma < G$ be a uniform lattice. Let $A$ be a measurable function on $X\times X$ which is invariant under the left diagonal $\Gamma$-action. Let $\mathbf{A}$ be the integral operator on $L^2(\Gamma\backslash G)$ with kernel $\sum_{\gamma\in\Gamma}A(g,\gamma h)$. Then there is $c>0$, depending only on $G$, such that
\[
\|\mathbf{A}\|_{\rm HS}^2\le \int_{\Gamma\backslash G}\int_G|A(g,h)|^2dg dh+ \frac{e^{cT}}{{\rm InjRad}(Y)^{\dim Y}} {\rm vol}\big( Y_{\le 2T}\big)\|A\|_\infty^2,
\]
for all $T > 0$ satisfying $\{g^{-1}h: (g,h)\in {\rm supp}(A)\}\subset K\exp(B_\ga (0,T))K$. 
\end{lemma}

\begin{proof}
See Lemma 5.2 of \cite{Brumley_Matz_23}, which is an adaptation of Lemma 5.1 of \cite{Le_Masson_Sahlsten_17} to the general setting.
\end{proof}

To prove Theorem \ref{main-geom-thm}, we shall apply Lemma \ref{lemma:general-HS} with the integral operator $\mathbf{A}(\tau)$, defined in \eqref{def:propagator}. This will require examining its kernel, which is described in the following result.

Throughout this subsection and the next, we may work under more general hypotheses on $G$ and $H_0$ than those imposed by the conditions of Theorem \ref{main-geom-thm}. The condition on the subsystem of reduced roots of $G_1$ and the extremal coweight hypothesis on $H_0$ will only play a role starting in subsection \ref{sec:MTbound}.

\begin{lemma}\label{lemma:kernel}
The integral operator $\mathbf{A}(\tau)$ on $L^2(\Gamma\backslash G)$, as defined in \eqref{def:propagator}, has kernel \; $\sum_{\gamma\in\Gamma} A(\tau)(g,\gamma h)$, where $A(\tau)$ is the  function on $\Gamma\backslash (G\times G)$ given by
\[
A(\tau)(g,h)=\frac{1}{\tau}\int_{\tau}^{2\tau} e^{-2t\rho(H_0)}\int_{gE_t\cap hE_t}a(x)dx dt.
\]
\end{lemma}

\begin{proof}
By definition, for a function $f \in L^2(\Gamma \backslash G)$,
\[
(U_t aU_t^*)f(g)=e^{-2t\rho(H_0)}\int_{E_t}a(gh_1)\int_{E_t^{-1}} f(gh_1h_2)dh_2dh_1.
\]
We change variables in the inner sum by setting $h=gh_1h_2$ and use Fubini's theorem to get
\begin{align*}
(U_t aU_t^*)f(g)&=e^{-2t\rho(H_0)}\int_{E_t}a(gh_1)\int_G f(h){\mathds{1}}_{E_t^{-1}}(h_1^{-1}g^{-1}h)dhdh_1\\
&=e^{-2t\rho(H_0)}\int_G f(h)\int_G a(gh_1){\mathds{1}}_{E_t^{-1}}(h_1^{-1}g^{-1}h){\mathds{1}}_{E_t}(h_1)dh_1 dh.
\end{align*}
We then change variables via $x=gh_1$, and use the left-$\Gamma$-invariance of $f$, to get
\begin{align*}
(U_t aU_t^*)f(g)&=e^{-2t\rho(H_0)}\int_G f(h)\int_G a(x){\mathds{1}}_{E_t^{-1}}(x^{-1}h){\mathds{1}}_{E_t}(g^{-1}x)dx dh\\
&=e^{-2t\rho(H_0)}\int_G f(h)\int_{gE_t\cap hE_t} a(x)dx dh\\
&=e^{-2t\rho(H_0)}\int_{\Gamma\backslash G} f(h)\sum_{\gamma\in\Gamma} \int_{gE_t\cap \gamma hE_t} a(x)dx dh.
\end{align*}
Averaging over $t$ yields the result.
\end{proof}

\subsection{The support of the kernel function}\label{sec:ker-supp}
To obtain the second term in the estimate \eqref{eq:HS-bound} via an application of Lemma \ref{lemma:general-HS}, it will be necessary to calculate the real parameter $T$ which bounds the size of the support of the kernel.

It follows from Lemma \ref{lemma:kernel}, and the bi-$K$-invariance of the sets $E_t$, that 
\begin{equation}\label{eq:support-kernel}
\begin{aligned}
\{g^{-1}h: (g,h)\in {\rm supp}(A(\tau))\} &\subseteq \bigcup_{\tau \le t \le 2\tau} \{g^{-1}h: gE_t\cap hE_t\neq\emptyset\}\\
& = \bigcup_{\tau \le t \le 2\tau} K\{e^H : e^H E_t\cap  E_t\neq\emptyset\}K.
\end{aligned}
\end{equation}
We therefore need to estimate the Euclidean norm of the largest element $H\in\bar{\ga}_+$ for which $e^HE_t$ intersects $E_t$. Now, recall from \eqref{eq:StimesB} that $E_t$ factorizes as $S_t\times B \subset G_1\times\mathcal{G}_2$, so that 
\begin{equation}\label{eq:inter-fact}
e^HE_t\cap E_t=(e^{H_1} S_t\cap S_t)(e^{H_2}B\cap B), \quad\textrm{where }\; H=H_1+H_2\in\ga_1\oplus {\rm Lie}(\mathcal{A}_2).
\end{equation}
We may take $\epsilon_0>0$ small enough in the definition \eqref{eq:def-St} of $B$ so that $e^{H_2}B\cap B=\emptyset$ if $\|H_2\|\geq 1$. It therefore suffices to estimate the Euclidean norm of the largest element $H\in\bar{\ga}_{1,+}$ for which $e^HS_t$ intersects $S_t$. For this reason, for the rest of this subsection, we may and shall assume that $G=G_1$ is a non-compact simple real Lie group and $E_t=S_t$.

We begin by proving the following triangle inequality in the irreducible symmetric space $X=G/K$. Given $H \in \ga$, let $\textnormal{Conv}(W.H)$ denote the convex hull of the $W$-orbit of $H$. We define a partial order on $\bar{\ga}_+$ via $H_1 \preceq H_2$ if and only if $\textnormal{Conv}(W. H_1) \subseteq \textnormal{Conv}(W. H_2)$. We define a ``Weyl chamber-valued metric'' on $G/K$ via $d_{\bar{\ga}_+}(gK, hK) = H \in \bar{\ga}_+$ if $g^{-1} h \in K e^H K$.

\begin{lemma}\label{lemma:triangle}
For all $x, y, z \in G/K$, we have
\[
        d_{\bar{\ga}_+}(x, z) \preceq d_{\bar{\ga}_+}(x, y) + d_{\bar{\ga}_+}(y, z).
  \]
    \end{lemma}

\begin{proof}
We first reduce the proof to the corresponding Lie algebra version, using \cite{Kapovich_Leeb_Millson_08}. Let $\mathfrak{p}$ be the orthogonal complement to $\mathfrak{k}$ with respect to the Cartan involution; thus $\mathfrak{p}$ is invariant under $\textnormal{Ad}(K)$. We have a Cartan decomposition with respect to adjoint $K$-action on $\mathfrak{p}$, namely $\mathfrak{p} = \bigsqcup_{H \in \overline{\ga}_+} K.H$.

Let $\kappa: \mathfrak{p} \to\overline{\ga}_+$ denote the corresponding Cartan projection, and let $d^{\, \inf}_{\bar{\ga}_+}: \mathfrak{p} \times \mathfrak{p} \to \overline{\ga}_+$, $d_{\bar{\ga}_+}^{\, \inf}(X, Y) = \kappa(Y-X)$, denote the associated Weyl chamber-valued metric; this metric is invariant under the action of the Cartan motion group $K \ltimes \mathfrak{p}$. According to Main Theorem 1.2 of \cite{Kapovich_Leeb_Millson_08}, given three points $x, y, z \in G/K$ such that $d_{\bar{\ga}_+}(x, y) = \lambda_1$, $d_{\bar{\ga}_+}(y, z) = \lambda_2$ and $d_{\bar{\ga}_+}(x, z) = \lambda_3$, there exist three points $X, Y, Z \in \mathfrak{p}$ such that $d_{\bar{\ga}_+}^{\inf}(X, Y) = \lambda_1$, $d_{\bar{\ga}_+}^{\inf}(Y, Z) = \lambda_2$ and $d_{\bar{\ga}_+}^{\inf}(X, Z) = \lambda_3$ (and vice versa). Thus, it suffices to show that for any $X, Y, Z \in \mathfrak{p}$ we have that
    \begin{gather*}
        d_{\bar{\ga}_+}^{\, \inf}(X, Z) \preceq d_{\bar{\ga}_+}^{\, \inf}(X, Y) + d_{\bar{\ga}_+}^{\, \inf}(Y, Z).
    \end{gather*}
This is in turn equivalent to showing that 
    \begin{gather}
        \kappa(X + Y) \preceq \kappa(X) + \kappa(Y) \label{eqn_triangle_ineq2}
    \end{gather}
for any $X, Y \in \mathfrak{p}$.

    Recall that the Killing form on $\mathfrak{p}$ is positive definite. Let $\Pi_\ga: \mathfrak{p} \to \ga$ denote the orthogonal projection with respect to the Killing form. By the Kostant convexity theorem \cite[Theorem 8.2]{Kostant_73}, for any $X \in \mathfrak{p}$, we have $\Pi_\ga(\textnormal{Ad}(K).X) = \textnormal{Conv}(W. X)$. Now, given $X, Y \in \mathfrak{p}$, we wish to show \eqref{eqn_triangle_ineq2}. By applying an element of $K$, we may assume that $X + Y \in \bar{\ga}_+$, and thus $\kappa(X + Y) = X + Y$. We also have that $X = k_1.\kappa(X)$ and $Y = k_2.\kappa(Y)$ for some $k_i \in K$. We thus have
    \begin{align*}
        \kappa(X + Y) &= X+Y \\
                    &= \Pi_\ga(X + Y) \\
                    &= \Pi_\ga(X) + \Pi_\ga(Y) \\
                    &= \Pi_\ga(k_1.\kappa(X)) + \Pi_\ga(k_2.\kappa(Y)) \\
                    & \in \textnormal{Conv}(W.\kappa(X)) + \textnormal{Conv}(W.\kappa(Y)) \\
                    & \subseteq \textnormal{Conv}(W.(\kappa(X) + \kappa(Y))),
    \end{align*}
the last equality following from Lemma \ref{lemma:conv-sum} below.
\end{proof}

\begin{lemma}\label{lemma:conv-sum}

If $X, Y \in \bar{\ga}_+$, then $\textnormal{Conv}(W.X) + \textnormal{Conv}(W.Y) \subseteq \textnormal{Conv}(W.(X+Y))$.

\end{lemma}

\begin{proof}

Since both sides of the inclusion are $W$-invariant, it suffices to show
\[
(\textnormal{Conv}(W.X) + \textnormal{Conv}(W.Y)) \cap \bar{\ga}_+\subset \textnormal{Conv}(W.(X+Y)) \cap \bar{\ga}_+.
\]
We apply \cite[Ch. IV, \S 8, Lemma 8.3]{Helgason_00} to express the right-hand side as $(X + Y - {}_+ C) \cap \bar{\ga}_+$, where
\[
{}_+ C = \{ H \in \ga : \langle H, H' \rangle \ge 0 \text{ for all } H' \in \bar{\ga}_+ \}.
\]
We are therefore reduced to showing that
\[
(\textnormal{Conv}(W.X) + \textnormal{Conv}(W.Y)) \cap \bar{\ga}_+\subset \{H\in \bar{\ga}_+: \langle H, H' \rangle \le \langle X+Y, H' \rangle,\; \forall\; H' \in \bar{\ga}_+\}.
\]
Let $H = H_1 + H_2\in\overline{\ga}_+$, where $H_1 \in \textnormal{Conv}(W.X)$ and $H_2 \in \textnormal{Conv}(W.Y)$.  By another application of \cite[Ch. IV, \S 8, Lemma 8.3]{Helgason_00}, we have $\textnormal{Conv}(W.X) \subset X - {}_+ C$, which implies that $\langle H_1, H' \rangle \le \langle X, H' \rangle$ for all $H' \in \bar{\ga}_+$. Likewise $\langle H_2, H' \rangle \le \langle Y, H' \rangle$ for all $H' \in \bar{\ga}_+$.  Adding these gives the desired inequality.
\end{proof}

Let $H_0\in\overline{\ga}_+$ be non-zero. Let 
\begin{equation}\label{defn:P}
P= \textnormal{Conv}(W.(H_0 -w_0H_0)) \cap \bar{\ga}_+,
\end{equation}
where $w_0$ is the longest element of the Weyl group $W$. By definition, $w_0$ is the unique element of maximal length with respect to the simple reflections $\{s_\alpha\mid \alpha\in\Delta\}$. The relevant property of $w_0$ in this context is that it satisfies $-w_0(\overline{\ga}_+)=\overline{\ga}_+$. For $t>0$ write $P_t=tP$. We define a ``polytopal norm'' on $\bar{\ga}_+$ by putting $\|H\|_P := \inf \{t : H \in P_t \}$. 

\begin{corollary}\label{cor:intersection-range}
There is $\epsilon_0>0$ small enough (in the definition of $S_t$) such that any $H \in \bar{\ga}_+$ for which $e^H S_t \cap S_t \neq \emptyset$ satisfies $\|H\|_P \leq t + 1$. In particular,
\[
\{g^{-1}h: (g,h)\in {\rm supp}(A(\tau))\}\subset K\exp(B_\ga (0,T))K,
\]
where $T=(\tau+1)\|H_0-w_0 H_0\|$.
\end{corollary}

\begin{proof}
We may take $\epsilon_0$ small enough so that $B_{\ga}(0, \epsilon_0) \subset \textnormal{Conv}(W.H_0)$, and thus $B_{\ga}(t H_0, \epsilon_0) \subset \textnormal{Conv}(W.(t + 1)H_0)$. We apply Lemma \ref{lemma:triangle} with $y \in e^H S_t \cap S_t$, $x =K$, and $z = e^H K$, to obtain

\[
H = d_{\bar{\ga}_+}(K, e^H K)  \preceq d_{\bar{\ga}_+}(K, y) + d_{\bar{\ga}_+}(y, e^H K).
\]
Since $y\in S_t$ we have $d_{\bar{\ga}_+}(K, y)\in  B_{\ga}(t H_0, \epsilon_0) $, so that, by the choice of $\epsilon_0$, the first term satisfies $d_{\bar{\ga}_+}(K, y) \preceq (t + 1)H_0$. Furthermore, from the property of $w_0$ recalled above, if $g\in Ke^HK$ then $g^{-1}\in Ke^{-w_0H}K$. The Cartan projection $\kappa$ therefore satisfies $\kappa(g^{-1})=-w_0\kappa(g)$. We deduce that
\[
d_{\bar{\ga}_+}(x,y)=\kappa(g^{-1}h)=\kappa((h^{-1}g)^{-1})=-w_0\kappa(h^{-1}g)=-w_0d_{\bar{\ga}_+}(y,x)
\]
for any $x=gK,y=hK\in G/K$. In particular, the second term satisfies
\[
d_{\bar{\ga}_+}(y, e^H K)=-w_0d_{\bar{\ga}_+}(e^H K,y)=-w_0d_{\bar{\ga}_+}(K,e^{-H}y)\preceq (t+1) (-w_0H_0),
\]
where we have used the $G$-invariance of $d_{\bar{\ga}_+}$ and the fact that $e^{-H}y\in S_t$. It follows that $H\in P_{t+1}$. To obtain the second statement, we apply \eqref{eq:support-kernel} and the fact that $P_t$ is contained in the Euclidean ball in $\ga$ of radius $t\|H_0-w_0 H_0\|$.
\end{proof}

Let $c>0$ be as in Lemma \ref{lemma:general-HS}. If we set $c_1=c\|H_0-w_0 H_0\|$ and $c_2=c_3=2\|H_0-w_0 H_0\|$, we obtain the second term in \eqref{eq:HS-bound}.

\subsection{Main term bound}\label{sec:MTbound}
We denote the main term in Lemma \ref{lemma:general-HS} by
\begin{equation}\label{eq:main-term}
M(\tau)=\int_{\Gamma\backslash G}\int_G| A(\tau)(g,h)|^2dg dh.
\end{equation}
The goal of this subsection is to show, under the assumptions laid out in Theorem \ref{main-geom-thm}, and assuming Theorem \ref{intersection-vol}, that $M(\tau)$ is bounded by the first term in the estimate \eqref{eq:HS-bound}. This will then fully reduce the proof of Theorem \ref{main-geom-thm} to that of Theorem \ref{intersection-vol}.

We shall make use of the following ergodic theorem due to Nevo \cite{Nevo}:
\begin{proposition}[Nevo] \label{prop_nevo}
Let $G$ be a non-compact simple real Lie group with finite center acting by measure-preserving transformations on a probability space $(\Sigma,\mu)$. Let $r_\Sigma$ denote the action on $L^2(\Sigma,\mu)$ and $r^0_\Sigma$ the restriction to the orthocomplement $L^2_0(\Sigma,\mu)$ to the constant functions.  Assume that $r_\Sigma^0$ has a spectral gap. There exist constants $\theta, C > 0$, depending on the integrability exponent for $r_\Sigma^0$, such that for any measurable set $E \subset G$ of finite measure, we have
    \begin{gather*}
        \big\| r_\Sigma^0({\rm vol}(E)^{-1}\mathds{1}_E)\big\|_{\textnormal{op}} \leq C \textnormal{vol}(E)^{-\theta},
    \end{gather*}
    where $\|\cdot \|_{\textnormal{op}}$ denotes the operator norm.
\end{proposition}
\begin{remark} \label{rmk_best_theta}
    In Section 4.1 of \cite{Gorodnik_Nevo} a sketch of the proof of Proposition \ref{prop_nevo} is given. From that discussion one sees that the best possible admissible value of $\theta$ by the proof technique presented there is $\theta = 1/2 - \varepsilon$, corresponding to $r_\Sigma^0$ being tempered; the corresponds to taking $n = 1$ in the first bullet point and $r = \frac{1}{2} - \varepsilon$ in the second bullet point of loc. cit.
\end{remark}

The following preliminary bound is valid under more general assumptions on $G$ and $H_0$ than those laid out in Theorem \ref{main-geom-thm}.

\begin{proposition} \label{prop:main-term-bd}
Let $G$ be a product of non-compact simple real Lie groups and $\Gamma<G$ an irreducible lattice. Fix a simple factor $G_1$ of $G$ and let $H_0\in\overline{\ga_1}_+$ be non-zero. 

There is $\theta>0$, depending on the integrability exponent of $G_1$ acting on $L^2_0(\Gamma\backslash G)$, and $\epsilon_0>0$ (implicit in the definition of $S_t$ and $E_t$), such that for $\tau>1$, and mean-zero $a\in L^\infty(Y)$,
\[
M(\tau)\ll \frac{\|a\|_2^2}{\tau^2} \int_{P_{\tau +1}} e^{2 \rho( H)} \Big| \int_{\max\{\tau, \|H\|_P - 1\}}^{2\tau} e^{- 2 t\rho( H_0 )} \textnormal{vol}(e^HS_t \cap S_t)^{1-\theta}dt \Big|^2 dH,
\]
where the $H$-integral is taken over $P_{\tau+1}\subset G_1$.
\end{proposition}
\begin{proof}
For a measurable set $E\subset G$ of finite measure it will be convenient to let $\varrho_{\Gamma\backslash G}(E)=\varrho_{\Gamma\backslash G}({\rm vol}(E)^{-1}{\mathds{1}}_E)$.  Inserting Lemma \ref{lemma:kernel} into the definition of $M(\tau)$ and changing variables $x\mapsto hx$ and $g\mapsto h^{-1}g$, we have
\begin{align*}
M(\tau)&=\frac{1}{\tau^2}\int_{\Gamma\backslash G}\int_G \bigg| \int_{\tau}^{2\tau} e^{-2t\rho(H_0)}\int_{gE_t\cap hE_t}a(x)dx dt\bigg|^2dg dh\\
&=\frac{1}{\tau^2}\int_{\Gamma\backslash G}\int_G \bigg| \int_{\tau}^{2\tau} e^{-2t\rho(H_0)}\int_{gE_t\cap E_t}a(hx)dx dt\bigg|^2dg dh.
\end{align*}
Using the Cartan decomposition for the integral over $g\in G$ and the fact that $kE_t=E_t$ for all $k\in K$, this equals
\[
\frac{1}{\tau^2}\int_{\Gamma\backslash G}\int_K \int_{\ga_+} \bigg| \int_{\tau}^{2\tau} e^{-2t\rho(H_0)}\int_{ke^H E_t\cap E_t}a(hx)dx dt\bigg|^2 J(H) dH dk dh.
\]
Changing $x\mapsto kx$ and using again that $k^{-1}E_t=E_t$, we get
\[
\frac{1}{\tau^2}\int_{\Gamma\backslash G}\int_K \int_{\ga_+} \bigg| \int_{\tau}^{2\tau} e^{-2t\rho(H_0)}\int_{e^H E_t\cap E_t}a(hkx)dx dt\bigg|^2 J(H) dH dk dh.
\]
Changing variables $h\mapsto hk^{-1}$, we can absorb the integral over $k\in K$ into the integral over $h\in \Gamma\backslash G$. Further recalling that $J(H) \ll e^{ 2 \rho ( H )}$, the above expression is then bounded by
\[
\frac{1}{\tau^2}\int_{\Gamma\backslash G}\int_{\ga_+} e^{2 \rho( H)}\bigg| \int_{\tau}^{2\tau} e^{-2t\rho(H_0)} {\rm vol}(e^H E_t\cap E_t)\big(\varrho_{\Gamma\backslash G} (e^H E_t\cap E_t)a \big)(h) dt\bigg|^2dH dh.
\]
We write $H=H_1+H_2\in \ga_1\oplus{\rm Lie}(\mathcal{A}_2)$ as in \eqref{eq:inter-fact}. Taking $\epsilon_0$ small enough in the definition of $B\subset\mathcal{G}_2$, we may truncate the $H_2$ integral to $\|H_2\|\leq 1$ and apply ${\rm vol}(e^{H_2}B\cap B) \ll 1$. From Corollary \ref{cor:intersection-range} we may truncate the $H_1$ integral to $\|H_1\|_P\leq \tau+1$. By the same token, given $H_1 \in P_{\tau + 1}$, the $t$ for which $e^{H_1}S_t \cap S_t \neq \emptyset$ satisfy $\|H_1\|_P - 1 \leq t $. Inserting this, and applying the Minkowski integral inequality, we obtain an upper bound of the form
\[
\frac{1}{\tau^2} \int_{P_{\tau + 1}\times \mathcal{B}_2(0,1)} e^{2 \rho(H_1)} \Big| \int_{\max\{\tau, \|H_1\|_P - 1\}}^{2\tau} e^{- 2t \rho( H_0 )} \textnormal{vol}(e^{H_1}S_t \cap  S_t)\|\varrho_{\Gamma\backslash G}(e^H E_t\cap E_t)a\|_2dt \Big|^2 dH,
\]
where, using the notation of Section \ref{sec:av-set}, $\mathcal{B}_2(0,r)\subset {\rm Lie}(\mathcal{A}_2)$ is the Euclidean ball of radius $r$, centered at $0$. Let $\varrho^0_{\Gamma\backslash G}$ denote the right-regular representation of $G$ on  $L^2_0(\Gamma\backslash G)$. Since $a\in L^2_0(\Gamma\backslash G)$, we have
\[
\|\varrho_{\Gamma\backslash G}(e^H E_t \cap E_t)a\|_2=\|\varrho_{\Gamma\backslash G}^0(e^H E_t \cap E_t)a\|_2\leq \|a\|_2\|\varrho_{\Gamma\backslash G}^0(e^H E_t \cap E_t)\|_{\rm op},
\]
where $\varrho^0_{\Gamma\backslash G}(E)$ is defined similarly as  $\varrho_{\Gamma\backslash G}(E)$.
From the spectral decomposition  \eqref{eq:spec-decomp}, we have
\[
\|\varrho_{\Gamma\backslash G}^0(e^H E_t \cap E_t)\|_{\rm op} = \sup_{\pi\subset L^2_0(\Gamma\backslash G)} \|\pi(e^H E_t \cap E_t)\|_{\rm op}.
\]
Since $\pi$ is irreducible, it is a tensor product of irreducible representations $\pi_i$ of the simple factors $G_1,\ldots ,G_s$ of $G$. We write $\pi=\pi_1\otimes \sigma$, where $\sigma=\pi_2\otimes\cdots\otimes\pi_s$ is an irreducible representation of $\mathcal{G}_2$. Then, from \eqref{eq:inter-fact}, $\|\pi(e^H E_t \cap E_t)\|_{\rm op}$ is
\[
\|\pi_1(e^{H_1} S_t \cap S_t)\|_{\rm op}\|\sigma(e^{H_2} B \cap B)\|_{\rm op} \ll \|\pi_1(e^{H_1} S_t \cap S_t)\|_{\rm op}\leq \|\varrho_{\Gamma\backslash G}^0\big|_{G_1}(e^{H_1} S_t \cap S_t)\|_{\rm op}.
\]
Here we have used the trivial bound $\|\sigma(e^{H_2} B \cap B)\|_{\rm op} \leq \|\mathds{1}_{e^{H_2} B \cap B}\|_{L^1(G)}$ (independently of $\sigma$). 

Since $\Gamma$ is irreducible in $G$, the restriction of $\rho_{\Gamma\backslash G}^0$ to any simple factor has a spectral gap; see \cite[p.3]{Kelmer_Sarnak} for details. 
Proposition \ref{prop_nevo}, applied to the simple factor $G_1$ acting by right-translation on $\Sigma=\Gamma\backslash G$, implies that there is $\theta>0$, depending on the integrability exponent of $G_1$ acting of $L^2_0(\Gamma\backslash G)$, such that
\[
\|\varrho_{\Gamma\backslash G}^0|_{G_1}(e^{H_1}S_t \cap S_t)\|_{\rm op}\ll  {\rm vol}(e^{H_1} S_t \cap S_t)^{-\theta}.
\]
Inserting this (and renaming $H_1$ to $H$) yields the statement.\end{proof}

We now make full avail of the assumptions on $G$ and $H_0$ in Theorem \ref{main-geom-thm} to conclude the reduction of its proof to Theorem \ref{intersection-vol}.

\begin{proposition}\label{prop:M-tau}
Let $G$ and $H_0$, and the simple factor $G_1$, be as in the statement of Theorem \ref{main-geom-thm}. There is a positive integer $k$, depending only on the reduced root system of $G_1$, and a constant $\theta>0$, depending only on the integrability exponent of $G_1$ acting on $L^2_0(\Gamma\backslash G)$, such that for for all $\tau \gg 1$,
\[
M(\tau)\ll  \frac{(\log \tau)^k}{\tau \theta^2}\|a\|^2_2 .
\]
\end{proposition}
\begin{proof}
If $M$ is the centralizer of $H_0$ in $G_1$, it follows from Theorem \ref{thm:semi-dense} that $\Phi_{{\rm red},M}$ is a semi-dense root system in the reduced roots of $G_1$. Taking $\tau$ sufficiently large, we may insert Theorem \ref{intersection-vol} into Proposition \ref{prop:main-term-bd}, to get
\[
M(\tau)\ll \frac{\|a\|_2^2}{\tau^2}\int_{P_{\tau + 1}}  e^{2\theta\rho(H)}\Big| \int_{\max\{\tau, \|H\|_P - 1\}}^{2\tau}  (\log t)^k  e^{-2\theta t \rho(H_0)}dt \Big|^2 dH.
\]
The inner integral is $O(\frac{(\log \tau)^k}{\theta} e^{-2\theta \|H\|_P \rho(H_0)})$,
so that for $\tau \geq 1$,
\begin{gather} \label{eqn_polytope_integral}
M(\tau)\ll   \frac{(\log \tau)^{2k}}{\tau^2\theta^2}\|a\|^2_2  \int_{P_{2 \tau}} e^{2\theta\rho(H - 2 \|H\|_P H_0)} dH.
\end{gather}
We may then apply Corollary \ref{lemma:Brion-bd} below, and rename $2k$ to $k$, to conclude the proof.
\end{proof}

\subsection{Application of the degenerate Brion's formula}\label{sec:Brion}
Recall the polytope $P$ from \eqref{defn:P}. Let $F_j\subset\overline{\ga}_+$, $1\leq j\leq J$, denote the codimension one faces of $P$ whose relative interior lies in $\ga_+$. Let $\ell_j\in\ga^*$ denote the linear functional characterized by $F_j=P\cap\{H\in\ga\mid \ell_j(H)=1\}$. We define the associated polyhedral cones $C_j=\{\lambda H: H\in F_j, \lambda\in \R_+\}$. Then $P\cap C_j=\{H\in C_j: \ell_j(H)\leq 1\}$, and if $H\in C_j$ then $H\in tP$ precisely when $\ell_j(H)\leq t$. Recalling from Section \ref{sec:ker-supp} the definition $\|H\|_P= \inf \{t : H \in tP\}$, it follows that $\|H\|_P=\ell_j(H)$ for all $H\in C_j$. As $\overline{\ga}_+=\cup_j C_j$, the integral in \eqref{eqn_polytope_integral} can be broken up as
\[
\int_{P_{2 \tau}} e^{2\theta\rho(H - 2 \|H\|_P H_0)} dH=\sum_{j=1}^J I_j(\tau),\qquad I_j(\tau)=\int_{P_{2 \tau}\cap C_j} e^{2\theta\rho(H - 2 \ell_j(H)} H_0) dH.
\]
Now, the integrals $I_j(\tau)$ could be expressed, via Brion's formula \cite{Brion, Barvinok}, as a weighted sum of the value of the integrand at the extremal vertices of $P_{2 \tau}\cap C_j$, \textit{provided that the values of the exponent at adjacent vertices are distinct}. This property does not, however, hold for $I_j(\tau)$, as we presently show.

Observe that the vertices of $P_{2 \tau} \cap C_j$ are simply the vertices of $P \cap C_j$ dilated by a factor of $2 \tau$. That both $0$ and $H_0 -w_0H_0$ are adjacent vertices of $P \cap C_j$ follows directly from the definition of $P$ in \eqref{defn:P} and the definition of $C_j$ above. Now, the exponent in $I_j(\tau)$ clearly vanishes at $0$. Moreover, since $\|v\|_P = 1$ at every vertex $v$ of $P \cap C_j$ other than $0$, the exponent at such a vertex is
\begin{gather*}
\rho(v) - 2 \rho(H_0)=\rho( v) -  \rho( H_0 -w_0H_0),
\end{gather*}
which then clearly also vanishes at $v=H_0-w_0H_0$.

In preparation for a more general form of Brion's formula, which allows for the exponent at adjacent vertices to be the same, we now show that $\rho(v) -  \rho(H_0 -w_0H_0)< 0$ for all other vertices of $P \cap C_j$. 

\begin{lemma}\label{prop_max_rho}
For any $H \in \bar{\ga}_+$, the linear functional $\rho$ attains its maximum on $\textnormal{Conv}(W.H)$ uniquely at $H$. In particular, $\rho(v) - 2\rho(H_0)<0$ for all vertices $v$ of the polytope $P \cap C_j$ other than $0$ and $H_0-w_0H_0$.
\end{lemma}
\begin{proof}
Since $\rho$ is linear, its maximum on $\textnormal{Conv}(W.H)$ is attained at a vertex. Since every vertex of $\textnormal{Conv}(W.H)$ belongs to $W.H$, it is enough to show that the maximal value of $\rho$ on $W.H$ is attained uniquely at $H$.

Denote by $L$ the centralizer of $H$ in $G$ and by $\Phi_L$ the roots of $A$ in $L$. Note that, for any $w\in W$, we have $\rho-w^{-1}\rho = \sum_{\alpha\in \Phi^+,\, w\alpha<0}\alpha$ so that $\rho(H)-\rho(wH)= (\rho-w^{-1}\rho)(H) = \sum_{\alpha\in \Phi^+,\, w\alpha<0} \alpha(H)\ge0$. If equality holds then $\{\alpha\in \Phi^+:w\alpha<0\}\subset\Phi_L$, since $\Phi_L$ consists precisely of those roots that vanish on $H$. But then $w\in W_L$ by \eqref{eq:WI} of Proposition \ref{lemma:WL-stable}. Since $W_L$ is precisely the stabilizer of $H$ in $W$, the case of equality $\rho(H)=\rho(wH)$ implies $wH=H$, so among all points of $W.H$ the value of $\rho$ is uniquely maximal at $H$.
%
\end{proof}

For $\lambda\in\ga^*$, the degenerate version of Brion's formula \cite{Peterson_24} implies that
\begin{equation}\label{eq:degen-Brion}
    \int_{P_{2 \tau} \cap C_j} e^{\lambda(H)} dH = \sum_{\substack{\text{faces $F$ of $P \cap C_j$ on} \\ \text{which  $\lambda(F) = $ constant}}} c_F  \textnormal{vol}(F)  (2 \tau)^{\textnormal{dim}(F)}  e^{2 \tau  \lambda(F)},
\end{equation}
where the constants $c_F$ are independent of $\tau$, and $c_F$ is positive on the maximal dimensional face maximizing $\lambda$. We apply \eqref{eq:degen-Brion} with $\lambda=2\theta(\rho - 2\rho(H_0)\ell_j)\in\ga^*$ and use the preceding discussion to obtain the following result.

\begin{cor}\label{lemma:Brion-bd}
We have $I_j(\tau) \ll \tau$.
\end{cor}

This concludes the proof of Proposition \ref{prop:M-tau} and therefore of Theorem \ref{main-geom-thm}, subject to the verification of Theorem \ref{intersection-vol}.

\section{A spectral approach to bounding intersection volumes} \label{sec_int_vol}

In this section we reduce the proof of Theorem \ref{intersection-vol}, which bounds the intersection volume ${\rm vol}(e^HS_t \cap S_t)$ under suitable assumptions on $G$ and $H_0$, to one combinatorial result and two spectral theoretic results. One of these spectral theoretic results is Theorem \ref{sph-fn-bd} which has already been stated in the introduction; it is valid for general $G$ and $H_0$. The other is stated below as Theorem \ref{thm:elem-spec-int} and crucially requires the assumptions on $G$ and $H_0$ present in Theorem \ref{intersection-vol}. These assumptions are assured to hold for groups having classical and $E_7$ root systems by the combinatorial result, Theorem \ref{thm:semi-dense2}.

After having made this reduction in this section, we then prove Theorem \ref{thm:semi-dense2} in Section \ref{sec:big-subsystems}, Theorem \ref{thm:elem-spec-int} in Section \ref{sec:elem-spec-int}, and Theorem \ref{sph-fn-bd} in Section \ref{sec:sph-fn}.

Throughout this section, we shall usually assume that $G$ is semisimple, as most of the results hold in this generality, while indicating clearly the places where we need to impose the condition of $G$ being simple.  We extend the definition of $S_t$ from \eqref{eq:def-St} to the case of $G$ semisimple in the obvious way.

\subsection{An initial bound due to J.-P. Anker} \label{sub_sec_anker}
We first present a bound on $\textnormal{vol}(e^H S_t \cap S_t)$ communicated to us by J.-P. Anker which gives a weaker bound but with much less work. This will serve as a brief illustration of how harmonic analytic information, such as bounds on the spherical function and the $c$-function, can be used to estimate the geometric intersection volume, and will act as a yardstick for measuring the extent to which a more elementary approach falls short of the bounds in Theorem \ref{intersection-vol}. In this subsection with take $H_0, H \in \overline{\ga}_+$ to be arbitrary. 

We first recall that, for $\lambda \in \mathfrak{a^*}$ we have the bound
\begin{gather*}
    |\varphi_\lambda(e^H)| \leq \varphi_0(e^H). 
\end{gather*}
See, for instance, \cite[Prop. 4.6.1]{Gangolli_Varadarajan_88}. Inserting this, we find
\begin{align*}
    \textnormal{vol}(e^H S_t \cap S_t) &= (\mathds{1}_{S_t} * \mathds{1}_{S_t}^\vee)(e^H) \\
    &= \int_{\ga_+^*} |\widehat{\mathds{1}}_{S_t}(\lambda)|^2 \varphi_\lambda(e^H) |c(\lambda)|^2 d \lambda \\
    & \leq \varphi_0(e^H) \int_{\ga_+^{*}} |\widehat{\mathds{1}}_{S_t}(\lambda)|^2 |c(\lambda)|^2 d \lambda \\
    & = \varphi_0(e^H) \|\mathds{1}_{S_t}\|^2_{L^2} = \varphi_0(e^H) \textnormal{vol}(S_t).
\end{align*}
We have that $\textnormal{vol}(S_t) \asymp e^{2 \rho(t H_0)}$ and by \cite{Anker_87} we have
\begin{gather*}
    \varphi_0(e^H) \asymp e^{-\rho(H)} \prod_{\alpha \in \Phi_{\textnormal{red}}^+} (1 + \alpha(H)).
\end{gather*}
Letting $L$ denote the centralizer of $H$ in $G$, we thus get
\begin{equation}\label{eqn_poly_bound}
    \textnormal{vol}(e^H S_t \cap S_t) \ll e^{\rho(2 t H_0 - H)} \prod_{\alpha \in \Phi_{\textnormal{red}}^+} (1 + \alpha(H)) \ll e^{\rho(2 t H_0 - H)} \| H \|^{|\Phi_{\textnormal{red}}^+ \smallsetminus \Phi_{\textnormal{red}, L}^+|},
\end{equation}
as all roots in $\Phi_{\textnormal{red}, L}^+$ vanish on $H$.  

Inserting \eqref{eqn_poly_bound} (applied to the simple factor $G_1$ of $G$), instead of Theorem \ref{intersection-vol}, into the analysis leading to \eqref{eqn_polytope_integral} would yield
    \begin{gather*}
        M(\tau) \ll \frac{\|a\|^2_2}{\tau^2\theta^2} \int_{P_{2 \tau}} \|H\|^{2 |\Phi_{\textnormal{red}}^{1 +} \smallsetminus \Phi_{\textnormal{red}, L_1}^+| (1-\theta)} e^{\theta \rho(H - 2 \|H\|_P H_0)} dH.
    \end{gather*}
Here, $\Phi_{\textnormal{red}}^1$ and $\Phi_{\textnormal{red}}^{1 +}$ are the reduced roots and positive reduced roots of the simple factor $G_1$, respectively, and $\Phi_{\textnormal{red}, L_1}^+$ denotes the positive reduced roots of $L_1$, the centralizer of $H$ in $G_1$. Now Remark \ref{rmk_best_theta} shows that $2 |\Phi_{\textnormal{red}}^{1 +} \smallsetminus \Phi_{\textnormal{red}, L_1}^+| (1-\theta) \geq 1$ for all $H\neq0$, in which case the above integral is at least $\tau^2$, cancelling out the factor $\frac{1}{\tau^2}$ and thus obtaining no decay. We conclude that a bound such as \eqref{eqn_poly_bound} would not suffice to complete the proof of Theorem \ref{main-theorem}. 

The point of Theorem \ref{intersection-vol} is to strengthen this bound to replace the polynomial factor in $H$ with a log factor, in the case where $H_0$ is taken to be extremal.

\subsection{The endpoint Kunze--Stein phenomenon} \label{subsec_eks}
The Kunze--Stein phenomenon \cite{Kunze_Stein} says that if $G$ is a semisimple real Lie group, then for any $p \in [1, 2)$ there is $c_p>0$ such that, for all $f \in L^p(G)$ and $g \in L^2(G)$, we have
\begin{gather*}
    \|f * g\|_2 \leq c_p \|f\|_p \|g\|_2.
\end{gather*}
In fact, the Kunze--Stein phenomenon is one of the main ingredients powering the Nevo ergodic theorem \cite{Nevo}, as stated in Proposition \ref{prop_nevo}.

In case $G$ has real rank one, Ionescu \cite{Ionescu_00} strengthened this to an endpoint estimate
\begin{equation}\label{endpoint-K-S}
    \|f * g \|_{2, \infty} \leq c \|f\|_{2, 1} \|g\|_{2, 1},
\end{equation}
where $\|\cdot\|_{p, q}$ is the $(p, q)$ Lorentz norm. It was noticed by the second author that if $G$ is a group satisfying the endpoint Kunze--Stein phenomenon for bi-$K$-invariant functions, then
\[
\textnormal{vol}(e^H S_t \cap S_t) \ll e^{\rho(2 t H_0 - H)}
\]
for any choice of $H_0$ and $H$. J.-P. Anker later communicated to us a simplified argument of this implication, which we now present.

For the purposes of this subsection, we shall deviate slightly from our earlier notation. For any $H\in\overline{\ga}_+$ and $r>0$ we define
\begin{gather*}
    S_{H, r} := K \exp(B_{\ga}(H, r))K.
\end{gather*}
The reason for introducing this more precise notation can be seen in the following elementary statement, where we keep track of this radius.
\begin{lemma}\label{eqn_anker1}
Let $H_0, H \in \overline{\ga}_+$ and $r>0$. We have
\[
    \textnormal{vol}(e^H S_{H_0, r} \cap S_{H_0, r}) \leq \textnormal{vol}(g S_{H_0, 2r} \cap S_{H_0, r})
\]
for all $g \in S_{H, r}$.
\end{lemma}

\begin{proof}
It will be enough to show that, for any $Z \in B_{\ga}(0, r)$,
\begin{equation}\label{eqn_anker1bis}
\textnormal{vol}(e^H S_{H_0, r} \cap S_{H_0, r}) \leq \textnormal{vol}(e^{H+Z} S_{H_0, 2r} \cap S_{H_0, r}).
\end{equation}
Indeed, for any $g \in S_{H, r}$ there is $Z \in B_{\ga}(0, r)$ such that $g\in Ke^{H+Z} K$. Since the sets $S_{H_0, r}$ are bi-$K$-invariant, it follows that $\textnormal{vol}(g S_{H_0, 2r} \cap S_{H_0, r}) = \textnormal{vol}(e^{H+Z} S_{H_0, 2r} \cap S_{H_0, r})$.

The inequality \eqref{eqn_anker1bis} in turn follows from
\[
e^H S_{H_0, r} \subset e^{H+Z} S_{H_0, 2r},
\]
valid for any $Z \in B_{\ga}(0, r)$. To see this inclusion, suppose that $x \in e^H S_{H_0, r}$. We deduce from Lemma \ref{lemma:triangle} that
\begin{gather*}
    d_{\overline{\ga}_+}(e^{H + Z}, x) \preceq d_{\overline{\ga}_+}(e^{H+Z}, e^H) + d_{\overline{\ga}_+}(e^H, x).
\end{gather*}
The first term on the right-hand side lies in $B_{\ga}(0, r)$, and the second term lies in $B_{\ga}(H_0, r)$. This implies that the sum lies in $B_{\ga}(H_0, 2r)$, as desired.
\end{proof}

We can re-express Lemma \ref{eqn_anker1} as stating that
\begin{gather*}
    \textnormal{vol}(e^H S_{H_0, r} \cap S_{H_0, r}) \mathds{1}_{S_{H}, r} \leq \mathds{1}_{S_{H_0}, r} * \mathds{1}_{S_{H_0}, 2r}^\vee.
\end{gather*}
We now apply the endpoint Kunze--Stein property \eqref{endpoint-K-S} to obtain
\begin{gather*}
    \textnormal{vol}(e^H S_{H_0, r} \cap S_{H_0, r}) \| \mathds{1}_{S_{H, r}}\|_{2, \infty} \leq \|\mathds{1}_{S_{H_0, r}}\|_{2, 1}  \|\mathds{1}_{S_{H_0}, 2r}\|_{2, 1}.
\end{gather*}
Finally we have the elementary estimate
\begin{gather*}
    \| \mathds{1}_{S_{H_0}, r} \|_{2, q} \asymp e^{\rho(H_0)}
\end{gather*}
for all $1 \leq q \leq \infty$. This implies $\textnormal{vol}(e^H S_{H_0, r} \cap S_{H_0, r}) \ll e^{\rho(2 H_0 - H)}$. 

\begin{remark}
As remarked in \cite{Ionescu_00}, the bound \eqref{endpoint-K-S} fails when $G$ is of rank at least $2$. The fact that Theorem \ref{intersection-vol} nevertheless recovers the same estimate on intersection volumes (up to $\log$ factors) for balls directed by extremal $H_0$, suggests that a version of the endpoint Kunze--Stein phenomenon, in which one restricts to bi-$K$-invariant functions supported near extremal elements, may extend to the higher rank setting. 
\end{remark}

\subsection{Reduction to a spectral estimate}\label{sec:red-to-spec}

We now return to the usual notation for $S_t$, as defined in \eqref{eq:def-St}. 

We introduce a smooth bump function that dominates $S_t$ to facilitate the analysis. Let $\psi \in C_c^\infty(\mathfrak{a})$ be a non-negative $W_M$-invariant function such that $\psi \geq 1$ on $B_{\mathfrak{a}}(0, \epsilon_0)$. Let $\psi_t$ be the translate of $\psi$ by $t H_0$, given by $\psi_t(Z) = \psi(Z - t H_0)$, which satisfies $h_t \geq 1$ on $B_{\mathfrak{a}}(t H_0, \epsilon_0)$. We then define $k_t \in C_c^\infty(G)$ to be the bi-$K$-invariant function on $G$ that satisfies $k_t(e^H) = \sum_{w \in W/W_M} \psi_t(H)$. Then $k_t \geq 1$ on $S_t$.

We have $\textnormal{vol}( S_t \cap e^H S_t ) \le (k_t * k_t^\vee)(e^H)$. It therefore suffices to estimate the convolution $k_t * k_t^\vee$.  We shall do this using the Harish-Chandra transform, by writing
\begin{equation}
\label{kHC}
(k_t * k_t^\vee)(e^H) = \int_{ \ga^*_+} | \widehat{k}_t(\lambda) |^2 \varphi_\lambda(e^H) | c(\lambda)|^{-2} d\lambda.
\end{equation}
We shall bound this expression using Theorem \ref{sph-fn-bd}, together with the following result which is obtained from that theorem through a uniform integration by parts argument. We recall the definition \eqref{Thetadef} of the function $\Theta(H,\lambda)$ appearing in the statement of Theorem \ref{sph-fn-bd}.
\begin{prop}
\label{khatdecay}
Let $G$ be a non-compact semisimple real Lie group with finite center. Assume Theorem \ref{sph-fn-bd}. Let $H_0\in\overline{\ga}_+$ be used to define $k_t$. For any $N > 0$ we have
\[
\widehat{k}_t(\lambda) \ll_N (1 + \| \lambda \|)^{-N} e^{ \rho(tH_0)} \Theta(tH_0,\lambda),
\]
uniformly in $t$.
\end{prop}

Assuming Theorem \ref{sph-fn-bd}, we insert Proposition \ref{khatdecay} into \eqref{kHC} to obtain
\[
(k_t * k_t^\vee)(e^H) \ll_N e^{ \rho( 2tH_0 - H)} \int_{ \ga^{*}_+} (1 + \| \lambda \|)^{-N} \Theta(tH_0, \lambda)^2 \Theta( H, \lambda) |c(\lambda)|^{-2} d\lambda.
\]
The right-hand side is now ``elementary", in the sense that, once one replaces $c(\lambda)$ by standard majorants obtained by Stirling's formula, what remains are explicit combinatorial expressions in the $H$, $H_0$, and $\lambda$. Note that $(k_t * k_t^\vee)(e^H) = 0$ if $t \ll \|H\|$. We must therefore prove the following estimate of the elementary spectral integral.
\begin{theorem}\label{thm:elem-spec-int}
Let $G$ be a non-compact simple real Lie group. Let $H_0 \in\overline{\mathfrak{a}}_+$ and denote by $M$ its centralizer in $G$. If $\Phi_{M,\textnormal{red}}$ is semi-dense in $\Phi_{\textnormal{red}}$, then there exists an integer $k\in\Z_{\geq 0}$ such that
\be
\label{spectral-int-bd}
\int_{ \ga^{*}_+} (1 + \| \lambda \|)^{-N} \Theta(tH_0, \lambda)^2 \Theta( H, \lambda) |c(\lambda)|^{-2} d\lambda \ll (\log t)^k
\ee
for $\| H \| \ll t$.
\end{theorem}

The combination of Theorem \ref{thm:semi-dense2}, Theorem \ref{sph-fn-bd}, Proposition \ref{khatdecay}, and Theorem \ref{thm:elem-spec-int} then implies Theorem \ref{intersection-vol}. We prove Proposition \ref{khatdecay} in the next subsection, leaving the proofs of Theorem \ref{thm:semi-dense2}, Theorem \ref{sph-fn-bd}, and Theorem \ref{thm:elem-spec-int}, which should be thought of as the workhorse of our approach, to the remainder of the paper.

\subsection{Proof of Proposition \ref{khatdecay}}

We let $\Delta$ denote the Laplacian on the symmetric space $X=G/K$. Recall that for $\lambda\in\ga^*$ we have $\Delta\varphi_\lambda =(\|\rho\|^2+\|\lambda\|^2)\varphi_\lambda$, see \cite[p.427, (7)]{Helgason_00}. Let $N$ be a positive integer. Integrating by parts (using Green's second identity) $N$ times we obtain
\[
\widehat{k}_t(\lambda)=(\|\rho\|^2+\|\lambda\|^2)^{-N}\int_X k_t(x)\Delta^N\varphi_{-\lambda}(x)dx=(\|\rho\|^2+\|\lambda\|^2)^{-N}\int_X \Delta^N k_t(x)\varphi_{-\lambda}(x)dx.
\]
Here we have identified the right $K$-invariant functions $k_t(g)$ and $\varphi_{-\lambda}(g)$ with the corresponding functions on $X$. Since the integrand is also left $K$-invariant we can use the integral decomposition \eqref{eq:Cartan-int} to write
\begin{equation}\label{eq:IPP}
\widehat{k}_t(\lambda)=b_G(\|\rho\|^2+\|\lambda\|^2)^{-N}\int_{\ga_+} \Delta^N k_t(e^H)\varphi_{-\lambda}(e^H)J(H)dH.
\end{equation}
Proposition \ref{khatdecay} will follow from the following bound on $\Delta^Nk_t$.

\begin{lemma}\label{lemma:DeltaM}
Let $\psi\in C_c^\infty(\ga)^{W_M}$ and $k_t$ be as defined in Section \ref{sec:red-to-spec}. Let $B\subset\ga$ be a $W_M$-invariant compact subset such that $\textnormal{supp} (\psi) \subset B$. Then, for all $H\in\ga$ and all $N\ge 0$,
\[
\Delta^Nk_t(e^H)\ll_N \sum_{w\in  W/W_M}{\mathds{1}}_B(H-wtH_0),
\]
uniformly in $t>0$.
\end{lemma}

To see how Proposition \ref{khatdecay} follows from this, we insert Lemma \ref{lemma:DeltaM} and Theorem \ref{sph-fn-bd} into \eqref{eq:IPP} to get
\[
\widehat{k}_t(\lambda)\ll_N (1+\|\lambda\|)^{-N} \int_{(tH_0 + B) \cap \mathfrak{a}_+} e^{-\rho(H)} J(H) \Theta(H, \lambda) dH.
\]
Recalling the definition \eqref{eq:def:J}, we have $J(H)\ll e^{2\rho (H)}$ for all $H\in\ga_+$. For $t$ large, the integrand is $O(e^{\rho(t H_0)} \Theta(t H_0, \lambda))$, yielding the bound of Proposition \ref{khatdecay}.

We now turn to proving Lemma \ref{lemma:DeltaM}.  We will prove this by working in radial coordinates.  Let $x_0\in X = G/K$ be the point corresponding to the trivial coset $K$. Let $\Delta_{\rm rad}$ denote the radial part of $\Delta$ relative to the $K$-action on $X$ with $A_+.x_0$ as a transversal manifold, where $A_+ = \exp(\mathfrak{a}_+)$. Then by the bi-$K$-invariance of $k_t$ we have $\Delta k_t(e^H)=\Delta_{\rm rad}k_t(e^H)$. 

We recall the expression for $\Delta_{\rm rad}$ from \cite[Ch. II, Prop. 3.9]{Helgason_00}. Let $\Delta_\mathfrak{a}$ denote the Laplacian on the flat manifold $A.x_0\subset X$. Then
\[
\Delta_{\rm rad}=\Delta_\mathfrak{a}+\sum_{\alpha\in\Phi^+} m_\alpha \coth(\alpha) X_\alpha,
\]
where $m_\alpha=\dim \g_\alpha$, and the element $X_\alpha\in\ga$, defined by $\langle X_\alpha,H\rangle = \alpha(H)$, is viewed as a differential operator on $A.x_0$. Since $A.x_0$ is isometric to $\mathfrak{a}$, we can equivalently work on $\mathfrak{a}$, and instead view $\Delta_{\mathfrak{a}}$ and $X_\alpha$ as differential operators on $\mathfrak{a}$.

For $H\in\ga$ let $T_H f(X)= f(X -H)$ denote the translation operator by $H$. Thus 
\begin{gather*}
    k_t(e^H) = \sum_{w \in W/ W_M} \psi(w H - t H_0) = \sum_{w \in W/ W_M } T_{t H_0} \psi(w H).
\end{gather*}
We define the translated operator $\Delta_{\rm rad}^t = T_{-t H_0} \Delta_{\rm rad} T_{t H_0}$. Explicitly, $\Delta_{\rm rad}^t$ is given by
\[
\Delta_{\rm rad}^t = \Delta_\mathfrak{a} + \sum_{\alpha \in \Phi^+} m_\alpha \coth( \alpha( \cdot - tH_0) ) X_\alpha.
\]
Therefore,
\begin{gather*}
    \Delta^N k_t(e^H) = \sum_{w \in W/ W_M } \big(T_{t H_0} (\Delta_\text{rad}^{t})^N \psi \big)(w H) = \sum_{w \in W / W_M} \big((\Delta_\text{rad}^{t})^N \psi \big)(w H - t H_0).
\end{gather*}
It therefore suffices to prove that $(\Delta_{\rm rad}^t)^N \psi \ll_N {\mathds{1}}_B$ uniformly in $t$. This follows from the following lemma.

\begin{lemma}
    Suppose $\mathcal{B}$ is a bounded subset of $C_c^\infty(\mathfrak{a})^{W_M}$ whose elements are supported in a compact set $\mathcal{K}$. If $t_0 > 0$ is such that the functions $H \mapsto \alpha(H - t H_0)$ are all non-vanishing on $\mathcal{K}$ for $t \geq t_0$ and $\alpha \in \Phi^+ \smallsetminus \Phi^+_M$, then the set $\bigcup_{t \geq t_0} \Delta^t_{\textnormal{rad}} \mathcal{B}$ is also bounded in $C_c^\infty(\mathfrak{a})^{W_M}$.
\end{lemma}
\begin{proof}
The operator $\Delta_{\rm rad}^t$ has singularities along the root hyperplanes corresponding to roots in $\Phi^+_M$.  To deal with these, it will be convenient to also introduce the symmetric space $X_M = M / K_M$, with Laplacian $\Delta_M$.  The radial part of $\Delta_M$ is given by
\[
\Delta_{M, \rm rad} = \Delta_{\ga} + \sum_{\alpha \in \Phi^+_{M}} m_\alpha \coth( \alpha) X_\alpha,
\]
so that
\[
\Delta_{\rm rad}^t = \Delta_{M, \rm rad} + \sum_{\alpha \in \Phi^+ \smallsetminus \Phi_{M}^{+} } m_\alpha \coth( \alpha( \cdot - tH_0) ) X_\alpha.
\]
The advantage of doing this is that we have written $\Delta_{\rm rad}^t$ as the sum of the radial part of the smooth operator $\Delta_M$, and the operators $\coth( \alpha( \cdot - tH_0) ) X_\alpha$ for $\alpha \in \Phi^+ \smallsetminus \Phi^{+}_{M}$, which are non-singular on functions supported on any fixed compact set satisfying the hypotheses of the lemma.

We shall prove that $\Delta_{M, \rm rad}$ is bounded on $C^\infty_c(\mathfrak{a})^{W_M}$ by reducing to the boundedness of $\Delta_M$. To do this, we define the restriction operator $R : C^\infty_c(X_M)^{K_M} \to C^\infty_c(\mathfrak{a})^{W_M}$.  Helgason proves in \cite[Ch. II, \S 5, Th. 5.8]{Helgason_00} that $R$ is an isomorphism of topological vector spaces, and therefore has a continuous inverse which we denote by $E : C^\infty_c(\mathfrak{a})^{W_M} \to C^\infty_c(X_M)^{K_M}$. 

We have $\Delta_{M, \rm rad} = R \circ \Delta_M \circ E$, and as the operators $E$, $R$, and $\Delta_M$ are all bounded on $C^\infty_c(\ga)^{W_M}$ or $C^\infty_c(X_M)^{K_M}$, we see that $\Delta_{M, \rm rad}$ is also bounded on $C^\infty_c(\ga)^{W_M}$, and thus $\bigcup_{t \geq t_0} \Delta_{M, \text{rad}} \mathcal{B}$ is a bounded set.

For any choice of $\varepsilon > 0$, we have that for every $n$ there exists a $C_n$ such that $|\frac{d^n}{d x^n} \coth(x)| < C_n$ for $x \in (\varepsilon, \infty)$. Therefore, by the product rule, we conclude that
\begin{gather*}
    \bigcup_{t \geq t_0} \Big(\sum_{\alpha \in \Phi^+ \smallsetminus \Phi^+_{M}} m_{\alpha} \coth(\alpha (\cdot - t H_0)) X_\alpha \Big) \mathcal{B}
\end{gather*}
is also a bounded set. 
\end{proof}

\section{Some properties of root subsystems}\label{sec:big-subsystems}

In this section, we introduce the notions of \textit{semi-dense} and \textit{extremal} root subsystems. The former are defined by a combinatorial inequality which appears naturally in the proof of Theorem \ref{intersection-vol}, while the latter, when well-defined, are easy to construct in practice. We then prove Theorems \ref{thm:semi-dense} and \ref{thm:semi-dense2}, by examining when such root subsystems exist and how they relate to one another.

\subsection{Semi-dense and extremal root subsystems}

Throughout this section $\Phi$ will denote a reduced root system of rank $n$ in a vector space $V$ with $\Delta$ as a base of simple roots.

We say that a root subsystem of $\Phi$ is \textit{standard} if it is the root subsystem generated by a subset of the simple roots, and that it is \textit{semistandard} if it is the image of a standard root subsystem under an element of the Weyl group. Semistandard root subsystems can be characterized as those arising as $\Phi \cap V'$ where $V'$ is some subspace of $V$ (Chapter VI,\S 1, Proposition 24 of \cite{Bourbaki}). 

\begin{definition}\label{def:big}
Let $\Phi$ be a reduced root system. We say that a semistandard root subsystem $\Phi_0\subset\Phi$ is {\normalfont semi-dense} if, for every semistandard root subsystem $\Psi\subset\Phi$ we have
    \begin{equation}\label{eqn_root_ineq}
        |\Psi \cap \Phi_0| + \textnormal{rank}(\Psi) \geq \frac{1}{2} |\Psi|. 
    \end{equation}
    \end{definition}

\begin{remark}
    Suppose $\Phi = \Phi^{(1)} \sqcup \dots \sqcup \Phi^{(m)}$ with each $\Phi^{(j)}$ irreducible. Suppose $\Phi^{(1)}$ admits a semi-dense root subsystem $\Phi^{(1)}_0$. Then the root subsystem $\Phi^{(1)}_0 \sqcup \Phi^{(2)} \sqcup \dots \sqcup \Phi^{(m)}$ is semi-dense in $\Phi$. 
\end{remark}

As an example of a maximal root subsystem which is \textit{not} semi-dense, let $\Phi$ be of type $B_3$ and take for $\Phi_0$ the $A_1 \times A_1$ root subsystem of $\Phi$ which is obtained by removing the middle node from the Dynkin diagram of $B_3$. We test the inequality \eqref{eqn_root_ineq} with $\Psi=\Phi$. Then $\Phi$ has 18 roots, and $\Phi_0$ has 4 roots. Thus $|\Psi \cap \Phi_0| + \text{rank}(\Psi)=|\Phi_0| + \text{rank}(\Phi)= 4 + 3$ which is smaller than $\frac{1}{2}|\Psi|=\frac{1}{2}|\Phi| = 9$.

If $\Phi$ is of type $A_n$, one can show that any maximal semistandard root subsystem is semi-dense. We will, however, only prove this for a special type of maximal root subsystems, which we call \textit{extremal}, in Proposition \ref{Sibd} below. We introduce the notion of extremal root subsystems for root systems of classical type and $E_7$. This will be our primary source of examples of semi-dense root subsystems:

\begin{definition}\label{def:extremal} 
Let $\Phi$ be a rank $n$ irreducible root system. Let $\Phi_0$ be a rank $n-1$, irreducible, semistandard,  root subsystem of $\Phi$. Then $\Phi_0$ is said to be {\rm extremal} if 
\begin{enumerate}
\item $\Phi$ is of classical type, and $\Phi_0$ in the same family as $\Phi$ (with the convention that $B_1= A_1$, $C_2= B_2$, and $D_3= A_3$), or
\item $\Phi$ is of type $E_7$ and $\Phi_0$ is of type $E_6$.
\end{enumerate}
If $\Phi$ is a reducible rank $n$ root system, $\Phi=\Phi^{(1)}\sqcup\dots\sqcup \Phi^{(m)}$ with all $\Phi^{(j)}$ irreducible, we call a rank $n-1$, semistandard root subsystem $\Phi_0= \Phi_0^{(1)}\sqcup\dots\sqcup \Phi_0^{(m)}$ of $\Phi$ extremal if there exists $1\le m_0\le m$ such that $\Phi_0^{(j)}=\Phi^{(j)}$ for all $j\neq m_0$ and $\Phi_0^{(m_0)}$ is extremal in $\Phi^{(m_0)}$ in the previous sense. 
\end{definition}
We include the degenerate case of $\Phi=A_1$, $\Phi_0=A_0=\emptyset$ in this definition. Note that in this case $\Phi_0$ is semi-dense in $\Phi$: $\Psi=\Phi$ is the only non-empty root subsystem of $\Phi$, and therefore $|\Psi\cap\Phi_0|=0$, $\textnormal{rank}(\Psi)=1$, and $|\Psi|=2$.

The terminology is intentionally reminiscent of the class of extremal fundamental coweights, introduced in Section \ref{sec:str-sing}. Indeed, extremal root subsystems, as defined above, are simply root subsystems obtained by applying an element of $W$ to the root subsystem $\Phi_{M, \textrm{red}}$ of $\Phi_{\textrm{red}}$, with $M$ the centralizer of an extremal element. Though not all semi-dense root subsystems are extremal, the extremal ones provide a natural class of semi-dense root subsystems for all types containing a semi-dense root subsystem.



\subsection{Classical root systems}\label{sec:big-root}

We begin by showing that for classical roots systems, extremal implies semi-dense. This will prove the direct implication of Theorem \ref{thm:semi-dense}, as well as Theorem \ref{thm:semi-dense2}, for classical groups.
\begin{proposition}
\label{Sibd}
    Suppose $\Phi$ is an irreducible root system of classical type. Then any extremal root subsystem is semi-dense. 
\end{proposition}

\begin{proof}
Let $\Phi_0\subset\Phi$ be extremal and let $\Psi\subset\Phi$ be semistandard. In fact, it will be enough to verify Definition \ref{def:big} with $\Psi$ standard. Indeed, writing $\Psi=w\Psi_{\rm st}$, where $\Psi_{\rm st}\subset\Phi$ is a standard root subsystem, we have $|\Psi \cap \Phi_0| = |\Psi_{\rm st} \cap w^{-1} \Phi_0|$. Since $\textnormal{rank}(\Psi)=\textnormal{rank}(\Psi_{\rm st})$ and $|w^{-1}\Phi_0|=|\Phi_0|$, inequality \eqref{eqn_root_ineq} is true for $\Psi$ if and only if it is true for $\Psi_{\rm st}$. 

We now let $\Psi$ be an arbitrary standard root subsystem of $\Phi$. To prove the proposition, it is enough to establish the inequality \eqref{eqn_root_ineq} for each of the irreducible factors of $\Psi$. We may therefore assume that $\Psi$ is irreducible. Moreover, we may assume that $\Psi\not\subset\Phi_0$ since otherwise the inequality \eqref{eqn_root_ineq}  is obviously true: $|\Psi|+\textnormal{rank}(\Psi)\geq |\Psi|/2$. We shall assume throughout the remainder of this proof that $\Psi$ is \textit{standard, irreducible and not contained in $\Phi_0$.}

Our strategy is to reduce  the verification of \eqref{eqn_root_ineq} to the base case $\Psi=\Phi$, which can be checked numerically. In this case we must verify
    \begin{equation}\label{basic-big-ineq}
        |\Phi_0|+\textnormal{rank}(\Phi) \geq \frac{1}{2} |\Phi|,
    \end{equation}
which can be checked by inspection of the following table:

    \begin{table}[H]
    \centering
 \begin{tabular}{|c|c|c|c|c|}
 \hline
  type of $\Phi$ & type of $\Phi_0$ & $|\Phi|$ & $|\Phi_0|$ \\ \hline  
  $A_n$ &$A_{n-1}$ & $n(n+1)$ & $(n-1)n$ \\
  $B_n$& $B_{n-1}$& $2 n^2$ & $2 (n-1)^2$  \\
  $C_n$& $C_{n-1}$& $2 n^2$& $2 (n-1)^2$ \\
  $D_n$& $D_{n-1}$& $2 n(n-1)$& $2(n-1)(n-2)$ \\
  \hline
 \end{tabular}
\end{table}
\noindent It remains then to show that when $\Psi\neq \Phi$ then $\Psi \cap \Phi_0$ is extremal in $\Psi$, in which case we can apply the base case to conclude. 
It will be enough to prove this claim for $\Psi$ maximal. Note that maximal standard subsystems correspond to removing one simple root from $\Delta$.

We divide the proof of the above claim according to the root type.

\medskip

\noindent \textit{Type $A_n$.} The root system of type $A_n$ has a model as vectors in $\mathbb{R}^{n+1}$ of the form $e_i - e_j$ with $i \neq j$. Let $V$ be the subspace spanned by the roots (i.e., the hyperplane orthogonal to $e_1 + \dots + e_{n+1}$). A basis of simple roots is $e_1 - e_2, \dots, e_{n} - e_{n+1}$. Let $\Phi'_0$ be the extremal root subsystem generated by $e_2 - e_3, \dots, e_n - e_{n+1}$. This is exactly the intersection of $\Phi$ with the hyperplane orthogonal to $e_1$. All extremal root subsystems are of the form $\Phi_0=w \Phi'_0$ for some $w\in W$. The extremal root subsystems are therefore precisely the root subsystems of the form
\begin{equation}\label{eq:ext:root:sub:A}
\Phi_0 =\Phi\cap \{\langle e_j, x \rangle = 0\}
\end{equation} 
for some $1\leq j\leq n+1$.

Fix $1\leq j\leq n+1$ and suppose $\Phi_0$ is as in \eqref{eq:ext:root:sub:A}.
    We wish to show that $\Phi_0\cap \Psi$ is extremal in $\Psi$. Since $\Psi$ is a maximal standard root subsystem, there exists $1\leq l<n+1$ such that $\Psi$ is obtained by removing $e_l-e_{l+1}$ from the set of simple roots. 
    Let $V^{n+1}_l$ and $W^{n+1}_l$ denote the linear subspaces of $\mathbb{R}^{n+1}$ generated by $e_1,\ldots, e_l$ and $e_{l+1},\ldots,e_{n+1}$, respectively. Then $\Psi=(\Psi\cap V^{n+1}_l) \sqcup (\Psi\cap W^{n+1}_l)=:\Psi_1\sqcup\Psi_2$ with  $\Psi_1$ and $\Psi_2$ irreducible root systems of type $A_{l-1}$ and $A_{n-l}$, respectively. Moreover, $\Phi_0\cap \Psi$ equals either $\Psi_{1,0}\sqcup\Psi_2$ or $\Psi_1\sqcup \Psi_{2,0}$, depending on whether $j\le l$ or $j\ge l+1$, where $\Psi_{k,0}=\Psi_k\cap \{\langle e_j, x \rangle = 0\}$. By the characterization of extremal root subsystems given just before \eqref{eq:ext:root:sub:A} $\Psi_{k,0}$ is therefore extremal in $\Psi_k$, and thus $\Phi_0\cap \Psi$ is extremal in $\Psi$.
    

\medskip

    \noindent \textit{Types $B_n$ and $C_n$.} The $B_n$ and $C_n$ cases are virtually identical, and thus we just discuss $B_n$. The analysis is in turn very similar to the $A_n$ case. We have as a model of our root system vectors of the form $\pm e_i \pm e_j$ with $i \neq j$ and $\pm e_j$ inside of $\mathbb{R}^n$. We may take as basis $e_1 - e_2, e_2 - e_3, \dots, e_{n-1} - e_n, e_n$. Let $\Phi_0'$ be the extremal root subsystem generated by the $e_2 - e_3, \dots, e_{n-1} - e_n, e_n$. This is exactly the intersection of $\Phi$ with the hyperplane orthogonal to $e_1$. All other extremal root subsystems are of the form $\Phi_0=w \Phi_0'$. There are exactly $n$ of these all of the form $\Phi \cap \{\langle e_j, x \rangle = 0\}$ for some $1 \leq j \leq n$ (because the orbit of $e_1$ under $W$ is exactly vectors of the form $\pm e_j$).

    Suppose $\Phi_0 = \Phi \cap \{ \langle e_j, x \rangle = 0\}$. If $\Psi$ is obtained by removing one of the first $n-1$ simple roots, we again have $\Psi=(\Psi\cap V^{n}_l) \times (\Psi\sqcup W^{n}_l)=:\Psi_1\sqcup\Psi_2$ for some $1\le l<n$, but with $\Psi_2$ now of type $B_{n-l}$ or $C_{n-l}$. The claim then follows as in the $A_n$ case. If $\Psi$ is obtained by removing $e_n$, then $\Psi$ is the usual $A_{n-1}$ root system in $\mathbb{R}^{n}$ and we are also back in the previous case.
    

\medskip

\noindent \textit{Type $D_n$.} The type $D_n$ root system has a model as vectors in $\mathbb{R}^n$ of the form $\pm e_i \pm e_j$ with $i \neq j$. We may take $\Delta = \{e_1 - e_2, e_2 - e_3, \dots, e_{n-1} - e_n, e_{n-1} + e_n \}$. Let $\Phi_0'$ be the standard extremal root subsystem obtained by removing $e_1 - e_2$ from $\Delta$. Then $\Phi_0' = \Phi \cap \{\langle e_1, x \rangle = 0\}$. All other extremal root subsystems are of the form $\Phi_0=w \Phi'_0$; each such one can be expressed as $\Phi\cap \{\langle e_j, x \rangle = 0\}$ for some $j$.

 Suppose $\Phi_0= \Phi\cap \{\langle e_j, x \rangle = 0\}$. If $\Psi$ is obtained by deleting one of the first $n-2$ simple roots, then we proceed as before by looking at the intersection of $\Psi$ with $V_l^n$ and $W_l^n$ (now $\Psi_2$ is of type $D_{n-l}$). If $\Psi$ is obtained by deleting the last simple root, $\Psi$ is just the usual $A_{n-1}$ in $\mathbb{R}^n$ and we are back in the first case. Finally, if $\Psi$ is obtained by removing the second to last root in $\Delta$, then $\Psi$ is also of type $A_{n-1}$ but with basis $e_1-e_2,\ldots, e_{n-2}-e_{n-1}, e_{n-1}+e_n$. But since $\{ \langle e_n, x \rangle = 0 \}= \{ \langle -e_n, x \rangle = 0 \}$, we are again back in the first case.

\end{proof}

\subsection{The type $E_7$ root system}
We now complete the proof of the direct implication of Theorem \ref{thm:semi-dense}, and of Theorem \ref{thm:semi-dense2}, by treating the exceptional group $E_7$.
\begin{prop} \label{lemma_e_7}
    Let $\Phi$ be of type $E_7$, and let $\Phi_0$ be an extremal root subsystem. Then $\Phi_0$ is semi-dense.
\end{prop}
\begin{proof}
    We have as a model for the $E_7$ root system those vectors in $\mathbb{R}^8$ of the form $e_i - e_j$ with $i \neq j$ as well as all permutations of $(\frac{1}{2}, \frac{1}{2}, \frac{1}{2}, \frac{1}{2}, -\frac{1}{2}, -\frac{1}{2}, -\frac{1}{2}, -\frac{1}{2})$. A base of simple roots is given by
    \begin{align*}
        \{& -e_2 + e_3, -e_3 + e_4, -e_4 + e_5, -e_5 + e_6, -e_6 + e_7, -e_7 + e_8, \\
        &\frac{1}{2}(e_1 + e_2 + e_3 + e_4 - e_5 - e_6 - e_7 - e_8) \}.
    \end{align*}
    Notice that all of $\Phi$ is orthogonal to $(1, \dots, 1)$. Let $\Phi_0$ be the standard root subsystem obtained by removing $-e_7 + e_8$ from the base; then $\Phi_0$ is of type $E_6$ and exactly consists of those roots orthogonal to $e_1 + e_8$.

    The Weyl group of $E_7$ has order $2^{10} \cdot 3^4 \cdot 5 \cdot 7$, and that of $E_6$ has order $2^7 \cdot 3^4 \cdot 5$. By the orbit-stabilizer theorem, we get that the $W$-orbit of $e_1 + e_8$ is of size $2^3 \cdot 7$. These are exactly the $2^2 \cdot 7$ vectors of the form $e_i + e_j$ with $i \neq j$, as well as the $2^2 \cdot 7$ vectors of the form $\frac{1}{2}(e_1 + \dots + e_8) - (e_i + e_j)$ with $i \neq j$.

    From here it is now easy to formulate a simple algorithm to verify \eqref{eqn_root_ineq}. Each standard root subsystem can be found by choosing a subset of the base and then calculating which roots lie in their span. Then for each such standard root subsystem $\Psi$ and each vector $v = w.(e_1 + e_8)$ in the $W$-orbit of $e_1 + e_8$ described in the previous paragraph, we calculate the number of roots that are orthogonal to $v$. This determines $|\Psi \cap w \Phi_0|$ from which \eqref{eqn_root_ineq} can be easily checked. We carried out this algorithm in Sage \cite{sage} and found that in all cases \eqref{eqn_root_ineq} was satisfied; see Appendix \ref{sage} for the Sage code that was used.
\end{proof}

\subsection{Exceptional root systems, not of type $E_7$}
Finally, we prove that no root subsystem among the exceptional types $E_6, E_8, F_4, G_2$ is semi-dense, establishing the reverse implication of Theorem \ref{thm:semi-dense}. This accounts for the exclusion of such root systems in Theorem \ref{main-theorem}. 

\begin{prop}\label{lemma:exceptional-small}
    Suppose $\Phi$ is of type $E_6, E_8, F_4$ or $G_2$. Then no root subsystem of $\Phi$ is semi-dense.
\end{prop}
\begin{proof}
    In the case of $G_2$ this is immediate: the only root subsystem is of type $A_1$, and if we take $\Psi= \Phi$ we see that the inequality fails ($G_2$ has 12 roots). 

    The case of $F_4$ is also straightforward: one may check that if we take $\Psi =\Phi$ (which has 48 roots) and $\Phi_0$ to be any standard proper root subsystem (which will be of type $C_3$, $B_3$, or $A_2\times A_1$ if they are maximal), then the desired inequality fails when taking $\Psi=\Phi$. 

    The type $E$ cases are more subtle and the desired inequality ``just barely'' fails to be true (and in fact it holds for $E_7$!). If we take for $\Phi_0$ the $E_7$ standard root subsystem in $E_8$, or the $D_5$ standard root subsystem in $E_6$, then the desired inequality in fact holds for any $\Psi$ standard. Constructing counterexamples to \eqref{eqn_root_ineq} therefore requires more work.

    We start with the case of $\Phi$ being the $E_8$ root system, which has 240 roots. For every maximal standard subroot system of $E_8$ other than $E_7$, one may compute that the desired inequality fails already simply taking $\Psi = \Phi$. Thus we are just left to analyze the case of $\Phi_0$ of type $E_7$.
    
    We can take as our model for $\Phi$ the elements in $\mathbb{R}^8$ of the form $\pm e_i \pm e_j$ with $i \neq j$ as well as $\pm \frac{1}{2} e_1 \pm \frac{1}{2} e_2 \pm \dots \pm \frac{1}{2} e_8$ with an even number of minus signs (sum of coefficients is even). We can take as our basis of simple roots: 
\begin{align*}
    \{& e_1 - e_2, e_2 - e_3, e_3 - e_4, e_4 - e_5, e_5 - e_6, e_6 - e_7, e_6 + e_7, \\
    &-\frac{1}{2}(e_1 + e_2 + e_3 + e_4 + e_5 + e_6 + e_7 + e_8)\}.
\end{align*}
     Let $\Phi_0$ be the $E_7$ standard root subsystem which is obtained from removing $e_1 - e_2$ from the base. This is exactly the intersection of $\Phi$ with the hyperplane orthogonal to $e_1 - e_8$. 

Let $w \in W$ be the element switching $e_3$ and $e_8$. Then $w \Phi_0$ consists of roots orthogonal to $e_1 - e_3$. Consider the standard root subsystem $\Psi$ generated by $e_1 - e_2$ and $e_2 - e_3$. This generates an $A_2$ type root system. However, no root in this root subsystem is orthogonal to $e_1 - e_3$. Therefore $w \Phi_0 \cap \Psi = \emptyset$. Thus $|w \Phi_0 \cap \Psi| + \textnormal{rank}(\Psi) = 2$ but $\frac{1}{2} |\Psi| = 3$. Thus the desired inequality does not hold for $E_8$.

Now suppose $\Phi$ is of type $E_6$. Recall that $E_6$ has 72 roots. One may easily check that if we take $\Phi_0$ to be any maximal subroot system other than type $D_5$ or $A_5$, then the desired inequality fails with $\Psi = \Phi$. Thus we must only analyze further the cases of $A_5$ and $D_5$.

Recall that the $A_2$ root system has a model as vectors in $\mathbb{R}^3$ of the form $(1,-1, 0)$, $(1, 0, -1)$, $(0, 1, -1)$ or their negatives. We may take as our model of $E_6$ those vectors in $\mathbb{R}^9 = \mathbb{R}^3 \times \mathbb{R}^3 \times \mathbb{R}^3$ where we separately take a copy of $A_2$ inside of each of three copies of $\mathbb{R}^3$, together with vectors of the form $(A; B; C)$ where $A, B, C \in \{(\frac{2}{3}, -\frac{1}{3}, -\frac{1}{3}), (-\frac{1}{3}, \frac{2}{3}, -\frac{1}{3}), (-\frac{1}{3}, -\frac{1}{3}, \frac{2}{3})\}$, and vectors of the form $(A'; B'; C')$ where $A', B', C' \in \{(-\frac{2}{3}, \frac{1}{3}, \frac{1}{3}), (\frac{1}{3}, -\frac{2}{3}, \frac{1}{3}), (\frac{1}{3}, \frac{1}{3}, -\frac{2}{3})\}$. We can take as our simple roots
\begin{align*}
   \{ & e_8 - e_9, e_7 - e_8, e_5 - e_6, e_4 - e_5, e_2 - e_3, \\
   &\frac{1}{3}e_1 - \frac{2}{3} e_2 + \frac{1}{3} e_3 - \frac{2}{3} e_4 + \frac{1}{3} e_5 + \frac{1}{3} e_6 - \frac{2}{3} e_7 + \frac{1}{3} e_8 + \frac{1}{3} e_9 \}.
\end{align*}
These roots are all orthogonal to $e_1 + e_2 + e_3, e_4 + e_5 + e_6$, and $e_7 + e_8 + e_9$.

First we let $\Phi_0$ be the standard root system of type $A_5$. This is obtained taking the root subsystem generated by the simple roots with $e_2 - e_3$ removed; equivalently these are the roots orthogonal to $e_1 - e_3$. Let $w \in W$ be the element switching $e_1$ and $e_2$. Then $e_1 - e_3$ is moved to $e_2 - e_3$, and $w \Phi_0$ is those roots orthogonal to $e_2 - e_3$. Let $\Psi$ be the root subsystem generated by the last two simple roots in the list above. Then $\Psi$ is of type $A_2$. However no vector in $\Psi$ is orthogonal to $e_2 - e_3$. Therefore $|w \Phi_0 \cap \Psi| = 0$, and we see that \eqref{eqn_root_ineq} fails.

Lastly we consider $\Phi_0$ of type $D_5$. We shall use a different model for $E_6$ than the one above. Specifically, we shall take $\Phi$ to be the root subsystem of $E_8$ obtained from removing the simple roots $e_1 - e_2$ and $e_2 - e_3$ from the base of $E_8$. Then $\Phi$ is the root subsystem orthogonal to $e_1 - e_2$ and $e_1 - e_8$. Let $\Phi_0$ be the root subsystem of $\Phi$ obtained by further removing the element $e_3 - e_4$ from the base. These are the roots which lie in the intersection of the space containing $\Phi$ with the hyperplane orthogonal to $e_1 + e_2 - 3 e_3 + e_8$. Let $w \in W$ be the element which switches $e_3$ and $e_6$. Then $w \Phi_0$ is orthogonal to $e_1 + e_2 - 3 e_6 + e_8$. Let $\Psi$ be the standard root subsystem generated by $e_6 + e_7$ and $-\frac{1}{2}(e_1 + \dots + e_8)$; then $\Psi$ is of type $A_2$. However, no element in $\Psi$ is orthogonal to $e_1 + e_2 - 3 e_6 + e_8$, so $w \Phi_0 \cap \Psi = \emptyset$. Thus, like before, \eqref{eqn_root_ineq} fails to hold.
\end{proof}

\section{Bounding the elementary spectral integral}\label{sec:elem-spec-int}

In this section we prove Theorem \ref{thm:elem-spec-int}. Let $r$ be the rank of $G$. We define coordinates on $\mathfrak{a}^*_+$ by the map $x \in \R_{> 0}^r \mapsto \lambda(x) = x_1 \varpi_1 + \ldots + x_r \varpi_r \in \mathfrak{a}^*_+$.  Applying this change of coordinates, the integral in \eqref{spectral-int-bd} becomes
\be
\label{spectral-int-2}
\int_{ \R_{> 0}^r} (1 + \| x \|)^{-N} \Theta(tH_0, \lambda(x))^2 \Theta( H, \lambda(x)) |c(\lambda(x))|^{-2} dx.
\ee
We shall bound this integral by decomposing $\R_{> 0}^r$ into a barycentric subdivision and a further refinement at the scale $t^{-1}$. Similar decompositions of $\ga_+$, rather than $\ga_+^*$, were considered in \cite{Marshall,Zhang}. Since only reduced root systems occur in this section, we shall simplify notation by dropping the subscript $\Phi_{\textnormal{red}}$ everywhere, so that $\Phi$, $\Phi_M$, $\Phi_M^+$, etc. denote sets of reduced roots.

\subsection{Barycentric subdivision and corresponding estimates}
The first subdivision we consider is determined by an ordering of the coordinates $x_i$. For each permutation $\sigma \in S_r$, let
\[
T_\sigma=\{x\in \R_{>0}^r: x_{\sigma(1)} > \cdots > x_{\sigma(r)}\}.
\]
Away from the hyperplanes $x_i=x_j$, the regions $T_\sigma$ form a disjoint decomposition of  $\R_{> 0}^r$. We refer to this as the \textit{barycentric subdivision} of $\R_{>0}^r$, since it induces the usual barycentric subdivision of the affine simplex $\{x\in \R_{>0}^r: x_1+\cdots +x_r=1\}$.

\begin{figure}[ht]
\centering

\begin{tikzpicture}[scale=4.2]
    \coordinate (A) at (0,0);
    \coordinate (B) at (1,0);
    \coordinate (C) at (0.5,0.8660254);
    \coordinate (M12) at ($(A)!0.5!(B)$);
    \coordinate (M13) at ($(A)!0.5!(C)$);
    \coordinate (M23) at ($(B)!0.5!(C)$);

    \draw[thick] (A)--(B)--(C)--cycle;
    \draw[thick] (C)--(M12);
    \draw[thick] (B)--(M13);
    \draw[thick] (A)--(M23);
\end{tikzpicture}

\caption{\label{bary-fig} The barycentric subdivision when $r=3$.}
\end{figure}
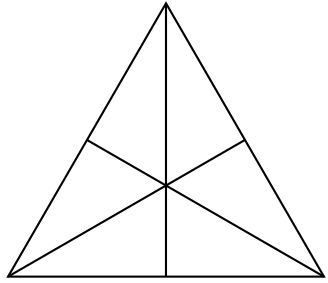

For $1 \le i \le r$, define $\Phi_{\sigma, i}$ to be the set of roots orthogonal to $\varpi_{\sigma(1)}, \ldots, \varpi_{\sigma(i)}$. In particular, $\Phi_{\sigma,r}=\emptyset$. We furthermore set $\Phi_{\sigma,0}=\Phi$. We observe that $\Phi_{\sigma, i}$ is the standard sub-root system consisting of all roots expressible as linear combinations of simple roots in $\Delta\smallsetminus\{\alpha_{\sigma(1)}, \ldots, \alpha_{\sigma(i)}\}$.

\begin{lemma}
\label{rootsize}
Let $1\le i\le r$. If $\alpha \in \Phi_{\sigma, i-1} \smallsetminus \Phi_{\sigma, i}$, then $| \langle \lambda(x), \alpha \rangle | \asymp x_{\sigma(i)}$ on $T_\sigma$.
\end{lemma}

\begin{proof}
We may assume that $\alpha \in \Phi^+$.  Let $e_1, \ldots, e_r$ be the standard basis for $\R^r$. If $x \in T_\sigma$, then $x = x_{\sigma(1)} e_{\sigma(1)} + \ldots + x_{\sigma(r)} e_{\sigma(r)}$ with $x_{\sigma(1)} > \ldots > x_{\sigma(r)}$. We have $\langle \alpha, \varpi_{\sigma(j)} \rangle = 0$ for $1 \le j \le i-1$, and $\langle \alpha, \varpi_{\sigma(j)} \rangle > 0$ for $i \le j \le r$, so that
\begin{align*}
\langle \lambda(x), \alpha \rangle & = \langle x_{\sigma(1)} \varpi_{\sigma(1)} + \cdots + x_{\sigma(r)} \varpi_{\sigma(r)}, \alpha \rangle \\
& = x_{\sigma(i)} \langle \varpi_{\sigma(i)}, \alpha \rangle + \cdots + x_{\sigma(r)} \langle \varpi_{\sigma(r)}, \alpha \rangle \asymp x_{\sigma(i)},
\end{align*}
as required.
\end{proof}

We now use Lemma \ref{rootsize} to estimate the various terms appearing in \eqref{spectral-int-2} on $T_\sigma$.

\begin{lemma}\label{lemma:barycentric-bounds}
Let $\sigma\in S_r$. For $x\in T_\sigma$ and $\|H\|\ll t$, we have
\be
\label{thetabd}
\Theta(H, \lambda(x)) \ll \prod_{i = 1}^r \min( t, x_{\sigma(i)}^{-1} +1)^{|\Phi_{\sigma, i-1}^+| - |\Phi_{\sigma, i}^+| }.
\ee
Furthermore we have
\be
\label{thetaH0bd}
\Theta(tH_0, \lambda(x)) \ll \sum_{w \in W} \prod_{i = 1}^r \min( t, x_{\sigma(i)}^{-1} +1)^{n_M(\sigma, i, w) },
\ee
where, for $1\le i\le r$, we have put $n_M(\sigma, i, w) = |(\Phi^+ \smallsetminus \Phi^+_M) \cap w(\Phi_{\sigma, i-1} \smallsetminus \Phi_{\sigma, i})|$.
Finally, we have
\be
\label{cnear}
|c(\lambda(x))|^{-2}\ll (1+x_{\sigma(1)})^{d} \prod_{i = 1}^r x_{\sigma(i)}^{ 2|\Phi_{\sigma, i-1}^+| - 2|\Phi_{\sigma, i}^+| },
\ee
where $d := \sum_{\alpha \in \Phi^+} m_\alpha + m_{2 \alpha}$ is the sum of the multiplicities of all relative roots.

\end{lemma}

\begin{proof}
We begin with $\Theta(H, \lambda(x))$.  We first apply the trivial estimate $\alpha(H) \ll t$, which gives
\[
\Theta(H, \lambda(x)) \ll \sum_{w \in W} \prod_{\alpha \in \Phi^+} \min( t, | \langle w\lambda(x), \alpha \rangle |^{-1} +1).
\]
We next break the product over $\Phi^+$ into the subsets $\Phi^+ \cap w(\Phi_{\sigma, i-1} \smallsetminus \Phi_{\sigma, i})$ and apply Lemma \ref{rootsize} to each subset, which gives
\begin{align*}
\prod_{\alpha \in \Phi^+} \min( t, | \langle w\lambda(x), \alpha \rangle |^{-1} +1) & \asymp \prod_{i = 1}^r \prod_{\alpha \in \Phi^+ \cap w(\Phi_{\sigma, i-1} \smallsetminus \Phi_{\sigma, i}) } \min( t, x_{\sigma(i)}^{-1} +1) \\
& = \prod_{i = 1}^r \min( t, x_{\sigma(i)}^{-1} +1)^{ |\Phi_{\sigma, i-1}^+| - |\Phi_{\sigma, i}^+| },
\end{align*}
yielding \eqref{thetabd}.

Next, we handle the $c$-function bound. When $|\langle \lambda(x), \alpha \rangle|$ is small, by \eqref{eqn_c_lambda_small} we have
\begin{gather*}
    |c_\alpha(\lambda(x))|^{-2} \ll |\langle \lambda(x), \alpha \rangle|^2.
\end{gather*}
When $|\langle \lambda(x), \alpha \rangle|$ is large, by \eqref{eqn_c_lambda_big} we instead have
\begin{gather*}
    |c_\alpha(\lambda(x))|^{-2} \ll |\langle \lambda(x), \alpha \rangle|^{m_\alpha + m_{2 \alpha}} \ll (1 + x_{\sigma(1)})^{m_\alpha + m_{2 \alpha}}. 
\end{gather*}
Combining these bounds, we obtain for all $x\in T_{\sigma}$
\begin{gather*}
    |c(\lambda(x))|^{-2} \ll (1 + x_{\sigma(1)})^{d} \prod_{\alpha \in \Phi^+} |\langle \lambda(x), \alpha \rangle|^2.
\end{gather*}
Applying the same argument used to prove \eqref{thetabd} we find \eqref{cnear}.

Finally, we apply a similar argument to $\Theta(tH_0, \lambda)$, with the added observation that $\alpha(tH_0) = 0$ for $\alpha \in \Phi^+_M$.  These roots therefore make no contribution to the formula for $\Theta(tH_0, \lambda)$, while for $\alpha \in \Phi^+ \smallsetminus \Phi^+_M$ we have $\alpha(tH_0) \ll t$ as before.  It follows that
\[
\Theta(tH_0, \lambda(x)) \ll \sum_{w \in W} \prod_{\alpha \in \Phi^+ \smallsetminus \Phi^+_M} \min( t, | \langle w\lambda(x), \alpha \rangle |^{-1} +1).
\]
As before, we partition $\Phi^+ \smallsetminus \Phi^+_M$ into the subsets $(\Phi^+ \smallsetminus \Phi^+_M) \cap w(\Phi_{\sigma, i-1} \smallsetminus \Phi_{\sigma, i})$. This gives
\begin{align*}
\prod_{\alpha \in \Phi^+ \smallsetminus \Phi^+_M} \min( t, | \langle w\lambda(x), \alpha \rangle |^{-1} +1) & \asymp \prod_{i = 1}^r \prod_{ \substack{ \alpha \in (\Phi^+ \smallsetminus \Phi^+_M) \\ \cap w(\Phi_{\sigma, i-1} \smallsetminus \Phi_{\sigma, i}) } } \min( t, x_{\sigma(i)}^{-1} +1) \\
& = \prod_{i = 1}^r \min( t, x_{\sigma(i)}^{-1} +1)^{ n_M(\sigma, i, w) },
\end{align*}
yielding \eqref{thetaH0bd}.
\end{proof}

\subsection{Singular refinement and corresponding estimates}
We now refine the estimates in Lemma \ref{lemma:barycentric-bounds} according to how many of the ordered coordinates $x_{\sigma(i)}$ exceed $t^{-1}$. We partition each barycentric chamber $T_\sigma$ into the following disjoint subregions:
\[
T_{\sigma, l}=\begin{cases}
\{x\in T_\sigma :  t^{-1} > x_{\sigma(1)}\}, & l=0;\\
\{x\in T_\sigma : x_{\sigma(l)} > t^{-1} > x_{\sigma(l+1)}\}, & 1\le l\le r-1;\\
\{x\in T_\sigma : x_{\sigma(r)} > t^{-1} \}, & l=r.
\end{cases}
\]
Note that $l$ records the number of coordinates $x_{\sigma(i)}$ that are greater than $t^{-1}$.  Similar decompositions of $\ga_+$, rather than $\ga_+^*$, are used in \cite{Marshall,Zhang}; \cite{Zhang} refers to these as the barycentric--semiclassical subdivision.

\begin{figure}[ht]
\centering
\begin{tikzpicture}[scale=4.2]
    \coordinate (A) at (0,0);
    \coordinate (B) at (1,0);
    \coordinate (C) at (0.5,0.8660254);
    \coordinate (M12) at ($(A)!0.5!(B)$);
    \coordinate (M13) at ($(A)!0.5!(C)$);
    \coordinate (M23) at ($(B)!0.5!(C)$);

    \def\eps{0.08}

    \coordinate (L1a) at ({1-\eps},0);
    \coordinate (L1b) at ({0.5*(1-\eps)},{0.8660254*(1-\eps)});
    \coordinate (L2a) at ({\eps},0);
    \coordinate (L2b) at ({\eps+0.5*(1-\eps)},{0.8660254*(1-\eps)});
    \coordinate (L3a) at ({0.5*\eps},{0.8660254*\eps});
    \coordinate (L3b) at ({1-0.5*\eps},{0.8660254*\eps});

    \draw[thick] (A)--(B)--(C)--cycle;
    \draw[thick] (C)--(M12);
    \draw[thick] (B)--(M13);
    \draw[thick] (A)--(M23);

    \draw[dashed] (L1a)--(L1b);
    \draw[dashed] (L2a)--(L2b);
    \draw[dashed] (L3a)--(L3b);
\end{tikzpicture}

\caption{The singular refinement of the barycentric subdivision when $r=3$ and $t>3$. The subregion $T_{\sigma,0}$ does not appear since it does not meet the affine simplex $\{x\in \R_{>0}^3: x_1+x_2+x_3=1\}$.}
\end{figure}
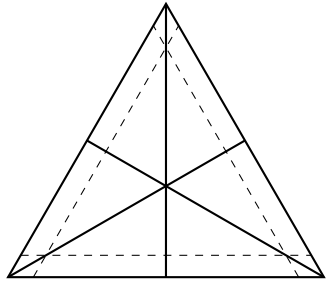

It will be convenient to introduce the following notation.

\begin{definition}\label{def:s-and-S}
For $\sigma\in S_r$, $w\in W$, and $1\le i\le r$, let
\[
s_M(\sigma, i, w)= |\Phi_{\sigma, i-1}^+| - |\Phi_{\sigma, i}^+| - 2n_M(\sigma, i, w),
\]
where $n_M(\sigma,i,w)$ was defined in Lemma \ref{lemma:barycentric-bounds}. For $1\le i\le r+1$ put
\[
S_M(\sigma, i, w) = |\Phi_{\sigma, i}^+| - 2|(\Phi^+ \smallsetminus \Phi^+_M) \cap w \Phi_{\sigma, i}|.
\]
\end{definition}

In particular, since $\Phi_{\sigma,r}=\emptyset$, we have $S_M(\sigma, r+1, w)=0$, and since $\Phi_{\sigma,0}=\Phi$ we have $S_M(\sigma, 1, w)=|\Phi^+| - 2|(\Phi^+ \smallsetminus \Phi^+_M)|=|\Phi^+| - 2(|\Phi^+| -|\Phi^+_M|)=-|\Phi^+| + 2|\Phi^+_M|$.

\begin{lemma}\label{lemma:refined-bound}
Let $\sigma\in S_r$, $w\in W$, and  $1\le l\le r+1$. There exists an $N'$ such that for $\lambda(x)\in T_{\sigma,l}$ and $\|H\|\ll t$ we have
\[
 \Theta(tH_0, \lambda(x))^2 \Theta( H, \lambda(x)) |c(\lambda(x))|^{-2} \ll  (1 + x_{\sigma(1)})^{N'} \Big(\sum_{w \in W} t^{-S_M(\sigma, l,w) } \prod_{i = 1}^{l} x_{\sigma(i)}^{s_M(\sigma,i,w) }\Big),
\]
where here, and throughout, a product from $i=1$ to $l$ is assigned the value 1 when $l=0$. 
\end{lemma}
\begin{proof}
On $T_{\sigma, l}$ we have $x_{\sigma(i)} > t^{-1}$ for $i \le l$ and  $x_{\sigma(i)} < t^{-1}$ for $i \geq l+1$, and so
\begin{align*}
\min( t, x_{\sigma(i)}^{-1} +1) & \asymp x_{\sigma(i)}^{-1} + 1 = x_{\sigma(i)}^{-1}( 1 + x_{\sigma(i)}), \quad i \le l, \\
\min( t, x_{\sigma(i)}^{-1} +1) & = t \le t ( 1 + x_{\sigma(i)}), \quad i \ge l+1.
\end{align*}
As a result, we deduce from the bound \eqref{thetabd} for $\Theta(H, \lambda)$ that
\[
\Theta(H, \lambda(x)) \ll (1 + x_{\sigma(1)})^{|\Phi^+|} t^{|\Phi_{\sigma, l}^+| } \prod_{i = 1}^{l} x_{\sigma(i)}^{ -|\Phi_{\sigma, i-1}^+| + |\Phi_{\sigma, i}^+|}.
\]
Likewise, the bound \eqref{thetaH0bd} for $\Theta(tH_0, \lambda)$ becomes
\[
\Theta(tH_0, \lambda(x)) \ll (1 + x_{\sigma(1)})^{|\Phi^+ \smallsetminus \Phi^+_M|} \sum_{w \in W}t^{|(\Phi^+ \smallsetminus \Phi^+_M) \cap w \Phi_{\sigma, l}|}  \prod_{i = 1}^{l} x_{\sigma(i)}^{- n_M(\sigma, i, w) }.
\]
We simplify the bound \eqref{cnear} for $|c(\lambda(x))|^{-2}$ to
\[
|c(\lambda(x))|^{-2} \ll (1 + x_{\sigma(1)})^{d} t^{ - 2|\Phi_{\sigma, l}^+|} \prod_{i = 1}^{l} x_{\sigma(i)}^{ 2|\Phi_{\sigma, i-1}^+| - 2|\Phi_{\sigma, i}^+| }.
\]
Combining these estimates yields the lemma.\end{proof}

For $1\le l\le r$, we put
\begin{equation}\label{eq:defn-I-int}
I_{M,\sigma, l,w}(t)=  \int_{ x_{\sigma(1)} > \cdots > x_{\sigma(l)} > t^{-1} } (1 + x_{\sigma(1)})^{-N} \prod_{i = 1}^{l} x_{\sigma(i)}^{s_M(\sigma, i,w) } dx_{\sigma(1)} \ldots dx_{\sigma(l)}
\end{equation}
and set $I_{M,\sigma, 0,w}(t)=1$.

\begin{cor}\label{cor:red-to-I-integral}
In the above notation, and under the conditions of Lemma \ref{lemma:refined-bound}, we have, for $N''$ much larger than $N'$,
\[
\int_{ T_{\sigma, l} } (1 + x_{\sigma(1)})^{-N''} \Theta(tH_0, \lambda(x))^2 \Theta( H, \lambda(x)) |c(\lambda(x))|^{-2} dx\ll \sum_{w\in W} t^{ -S_M(\sigma,l,w) +l - r}I_{M,\sigma, l,w}(t).
\]
\end{cor}
\begin{proof}
From Lemma \ref{lemma:refined-bound} it follows that the left-hand side is bounded by
\[
\sum_{w \in W} t^{ -S_M(\sigma, l,w) }\int_{ T_{\sigma, l} } (1 + x_{\sigma(1)})^{-N} \prod_{i = 1}^{l} x_{\sigma(i)}^{s_M(\sigma,i,w) } dx_{\sigma(1)} \ldots dx_{\sigma(r)}.
\]
When $l=r$ this gives the stated bound. For $0\le l\le r-1$, we perform the integrals over the variables which don't appear in the product, to find 
\[
\int_0^{t^{-1}}\int_0^{x_{\sigma(l+1)}}  \cdots \int_0^{x_{\sigma(r)}}  dx_{\sigma(l+1)} \ldots dx_{\sigma(r)}\le \int_0^{t^{-1}}\cdots  \int_0^{t^{-1}} dx_{\sigma(l+1)} \ldots dx_{\sigma(r)}=t^{l-r},
\]
as desired.
\end{proof}

\subsection{Intersection cardinalities of root systems}

Before proceeding to estimate the integrals appearing in Corollary \ref{cor:red-to-I-integral}, we establish a few important properties of the quantities $s_M(\sigma,i,w)$ and $S_M(\sigma,i,w)$ appearing in Definition \ref{def:s-and-S}.
\begin{lemma}\label{lemma:S-properties}
Let $\sigma\in S_r$, $w\in W$, and $\Phi_M$ a standard sub-root system. Then
\begin{enumerate}
\item\label{S-recursive} for $1\le i\le r$ we have $S_M(\sigma,i,w) + s_M(\sigma,i,w) = S_M(\sigma,i-1,w)$;
\item\label{S-relation} for $0\le i\le r$ we have $S_M(\sigma,i,w)=|\Phi_M\cap w\Phi_{\sigma,i}|-\frac12 |\Phi_{\sigma,i}|$.
\end{enumerate}
\end{lemma}

\begin{proof}
For point \eqref{S-recursive} we have
\begin{align*}
S_M(\sigma,i-1,w)-S_M(\sigma,i,w)&=|\Phi_{\sigma, i-1}^+| - 2|(\Phi^+ \smallsetminus \Phi^+_M) \cap w \Phi_{\sigma, i-1}|\\
&\qquad\qquad\;-(|\Phi_{\sigma, i}^+| - 2|(\Phi^+ \smallsetminus \Phi^+_M) \cap w \Phi_{\sigma, i}|)\\
&=|\Phi_{\sigma, i-1}^+| -|\Phi_{\sigma, i}^+|+ 2|(\Phi^+ \smallsetminus \Phi^+_M) \cap w (\Phi_{\sigma, i-1}\smallsetminus  \Phi_{\sigma, i})|\\
&=s_M(\sigma,i,w).
\end{align*}
For point \eqref{S-relation} we calculate
\begin{align*}
S_M(\sigma, i, w) &= |\Phi_{\sigma, i}^+| - 2|(\Phi^+ \smallsetminus \Phi^+_M) \cap w\Phi_{\sigma, i}^+|\\
&= \frac12 |\Phi_{\sigma, i}| - |(\Phi \smallsetminus \Phi_M) \cap w \Phi_{\sigma, i}|\\
&= \frac12 |\Phi_{\sigma, i}| - (|w\Phi_{\sigma, i}|-  |\Phi_M\cap w\Phi_{\sigma, i} |),
\end{align*}
which simplifies to the stated identity.\end{proof}

\subsection{Integrating over the refined subdivision}

By contrast with the preceding steps, in which $H_0$, and hence $\Phi_M$, was allowed to be arbitrary, the remainder of the argument will rely critically on the semi-dense hypothesis on $\Phi_M$ present in the statement of Theorem \ref{thm:elem-spec-int}.

\begin{cor}\label{cor:S-non-negative}
Let $\Phi_M$ be a semi-dense root subsystem. Fix $\sigma\in S_r$, $w\in W$, and $0\le i\le r$. Then $S_M(\sigma,i,w)+r\ge i$.
\end{cor}

\begin{proof}
We apply Lemma \ref{lemma:S-properties}\eqref{S-relation}, and the defining inequality \eqref{eqn_root_ineq} for semi-dense root subsystems, with $\Phi_0 = \Phi_M$ and $\Psi=w\Phi_{\sigma, i}$.
\end{proof}

Recall the integral $I_{M,\sigma, l,w}(t)$ defined in \eqref{eq:defn-I-int}. When combined with Corollary \ref{cor:red-to-I-integral}, the following lemma completes the proof of Theorem \ref{thm:elem-spec-int}.

\begin{lemma} \label{lemma_power_k}
Let $\Phi_M$ be a semi-dense root subsystem. Fix $\sigma\in S_r$, $w\in W$. For $1\le l\le r$ we let
\[
e_M(\sigma,l,w)=\#\{1\le i\le l: S_M(\sigma,i,w)+r= i\}.
\]
and put $e_M(\sigma,0,w)=0$. Then
\[
t^{ -S_M(\sigma,l,w) +l - r}I_{M,\sigma, l,w}(t)\ll (\log t)^{e_M(\sigma,l,w)},
\]
the implied constant depending only on the integer $N$ appearing in the definition of $I_{M,\sigma, l,w}(t)$. \end{lemma}

\begin{proof}
For notational simplicity, we drop the dependence on $w$, $\sigma$, and $M$ in the notation $I_{M,\sigma, l,w}(t)$, $S_M(\sigma,l,w)$ and $s_M(\sigma,i,w)$, writing simply $I_l(t)$, $S(l)$ and $s(i)$ for these quantities. In particular, $S(r)=0$.

For $l=0$, recall from the discussion preceding Corollary \ref{cor:red-to-I-integral} that, by convention, $I_0(t)=1$. We must therefore show that $S(0)+r\ge 0$, which follows from Corollary \ref{cor:S-non-negative}.

For $1\le l\le r$ we let
\[
J_l(t)=\int_{ x_{\sigma(1)} > \cdots > x_{\sigma(l)} > t^{-1} } (1 + x_{\sigma(1)})^{-N} x_{\sigma(l)}^{ S(l) + r-l} \prod_{i = 1}^{l} x_{\sigma(i)}^{s(i) } dx_{\sigma(1)} \ldots dx_{\sigma(l-1)}.
\]
As was the case for $I_{M,\sigma, l,w}(t)$, we do not indicate the dependence on $N$ in the notation, and we shall allow $N$ to be taken sufficiently large. From the inequality $(t^{-1})^{ S(l) + r-l } < x_{\sigma(l-1)}^{ S(l) + r-l }$, it follows that $t^{ -S(l) +l - r }I_l(t)<J_l(t)$. We shall bound $J_l(t)$ by induction.

The base case is when $l=1$, which we handle first. Since $N$ can be taken sufficiently large, and $S(1)+s(1)=S(0)$ from Lemma \ref{lemma:S-properties}\eqref{S-recursive}, we have
\[
J_2(t)=\int_{t^{-1}}^\infty (1 + x_{\sigma(1)})^{-N} x_{\sigma(1)}^{ S(1) +s(1)+ r - 1}  dx_{\sigma(1)}\ll \int_{t^{-1} }^1 x_{\sigma(1)}^{ S(1) +s(1)+ r - 1}  dx_{\sigma(1)}=\int_{t^{-1} }^1 x_{\sigma(1)}^{ S(0)+ r - 1}  dx_{\sigma(1)}.
\]
From Corollary \ref{cor:S-non-negative} the exponent satisfies $S(0) + r -1 \ge -1$. If this exponent is $\ge 0$ then we may extend the lower limit of integration to $0$ and obtain a bound of $1$. When the exponent is $-1$, the integral is bounded by $\log t$. This establishes the lemma when $l=1$. 

We now assume that $2\le l\le r$ and aim to show that
\begin{equation}\label{eq:J-induction}
J_l(t)\ll (\log t)^{\delta_l} J_{l-1}(t),\quad\textrm{where}\quad
\delta_l=
\begin{cases} 
0,& S(l-1)+r> l-1;\\
1,& S(l-1)+r= l-1.
\end{cases}
\end{equation}
For this we consider the inner integral over $x_{\sigma(l)}$ in $J_l(t)$. Using the recursive identity $S(l)+s(l)=S(l-1)$ from Lemma \ref{lemma:S-properties}\eqref{S-recursive} we have
\[
\int_{ t^{-1}}^{ x_{\sigma(l-1)} } x_{\sigma(l)}^{S(l) +r-l+s(l)} dx_{\sigma(l)}=\int_{ t^{-1}}^{ x_{\sigma(l-1)} } x_{\sigma(l)}^{S(l-1) + r-l} dx_{\sigma(l-1)}.
\]
From Corollary \ref{cor:S-non-negative} the exponent satisfies $S(l-1) + r-l\ge -1$, with equality occurring precisely when $\delta_l=1$. If $\delta_l=0$, then this exponent is $\ge 0$ and we may extend the lower limit of integration to $0$ and obtain a bound of $x_{\sigma(l-1)}^{S(l-1) + r-(l-1)}$, yielding \eqref{eq:J-induction}. If $\delta_l = 1$, then the inner integration over $x_{\sigma(l-1)}$ now gives
\[
\int_{ t^{-1}}^{ x_{\sigma(l-1)} } x_{\sigma(l)}^{-1} dx_{\sigma(l)} = \log x_{\sigma(l-1)} - \log(t^{-1}) \ll  (1 + x_{\sigma(1)})+\log t\ll (\log t)(1 + x_{\sigma(1)}).
\]
The $\log t$ term may be brought to the outside of the integral, and the $1 + x_{\sigma(1)}$ may be absorbed, yielding
\[
J_l(t)\ll (\log t)\int_{ x_{\sigma(1)} > \cdots > x_{\sigma(l-1)} > t^{-1} } (1 + x_{\sigma(1)})^{-N}  \prod_{i = 1}^{l-1} x_{\sigma(i)}^{s(i) } dx_{\sigma(1)} \ldots dx_{\sigma(l-1)}.
\]
Since, in the definition of $J_{l-1}$, the power of $x_{\sigma(l-1)}$ in front of the product $\prod_{i=1}^{l-1}x_{\sigma(i)}^{s(i) }$ is $S(l-1)+r-(l-1)=0$ when $\delta_l=1$, the integral above is $(\log t)J_{l-1}(t)$, establishing \eqref{eq:J-induction} in this case as well.
\end{proof}


\section{Bounds for spherical functions}
\label{sec:sph-fn}

This section contains the proof of Theorems \ref{sph-fn-bd} and \ref{sph-fn-cx}.  Section \ref{sec:sph-fn-conseq} provides some illustration of these theorems, and explains why Theorem \ref{sph-fn-cx} should be sharp, and Section \ref{sec:sph-fn-prev} discusses the relation between these theorems and previous work.  Section \ref{sec:sph-fn-outline} gives an outline of the proof of Theorem \ref{sph-fn-bd}, and the proof is carried out in Sections \ref{sec:sph-fn-prelim} to \ref{sec:Lmainbd}.  We prove Theorem \ref{sph-fn-cx} in Section \ref{sec:cx-pf}.

\subsection{Illustration of Theorems \ref{sph-fn-bd} and \ref{sph-fn-cx}}
\label{sec:sph-fn-conseq}

In this section, we make some remarks to help the reader understand the significance of Theorems \ref{sph-fn-bd} and \ref{sph-fn-cx}.  This includes explicating these theorems in the case of ${\rm SL}_2(\C)$, deriving some simpler consequences of Theorem \ref{sph-fn-bd} for a general group, and explaining why Theorem \ref{sph-fn-cx} should be sharp based on stationary phase considerations.  We also note that the bounds of Theorems \ref{sph-fn-bd} and \ref{sph-fn-cx} are equivalent in the case when $G$ is a complex group and $\lambda \in \ga^*$ is bounded, by Lemma \ref{sph-bd-equiv} below.

\subsubsection{The case of ${\rm SL}_2(\C)$}

When $G = {\rm SL}_2(\C)$, we have the following explicit formula for the spherical function.  If $\alpha$ is the unique element of $\Phi^+$, we may identify $\ga^*_\C$ with $\C$ and $\ga$ with $\R$ using the basis elements $\alpha$ and $\alpha^\vee/2$ respectively.  We then have
\be
\label{SL2C}
\varphi_\lambda(e^t) = \frac{\sin(\lambda t)}{\lambda \sinh t},
\ee
and we may compare Theorems \ref{sph-fn-bd} and \ref{sph-fn-cx} with the bounds that may be deduced from this formula.  For simplicity, we shall assume that $\lambda$ is real, and we assume without loss of generality that $\lambda$ and $t$ are non-negative.  In this case, Theorem \ref{sph-fn-bd} states that
\be
\label{SL2C1}
\varphi_\lambda(e^t) \ll (1 + \lambda)^a \min( t+1, \lambda^{-1} + 1) e^{-t},
\ee
and Theorem \ref{sph-fn-cx} states
\be
\label{SL2C2}
\varphi_\lambda(e^t) \ll e^{-t} \frac{t+1}{t\lambda +1}.
\ee
We note that these bounds are of the same strength if $\lambda$ is bounded, as established in Lemma \ref{sph-bd-equiv}, and it may be checked that \eqref{SL2C2} is stronger than \eqref{SL2C1} as expected.

We shall show that \eqref{SL2C2} is sharp with complete uniformity in $t$ and $\lambda$.  To do this, we consider the cases $t\lambda < 1$ and $t \lambda \ge 1$ separately.  When $t\lambda < 1$ we have $\sin(t \lambda) \asymp t\lambda$, so that \eqref{SL2C} implies that $\varphi_\lambda(e^t) \asymp t / \sinh t$.  On the other hand, \eqref{SL2C2} gives
\[
\varphi_\lambda(e^t) \ll (t+1) e^{-t} \asymp t / \sinh t,
\]
which shows that \eqref{SL2C2} is optimal in this range.

When $t\lambda \ge 1$ the best bound one can give for $\sin(\lambda t)$ is $|\sin(\lambda t)| \le 1$, and so by \eqref{SL2C} the best bound one can give for $\varphi_\lambda(e^t)$ is $\varphi_\lambda(e^t) \ll 1/(\lambda \sinh t)$.  In comparison, \eqref{SL2C2} becomes
\[
\varphi_\lambda(e^t) \ll e^{-t} \frac{t+1}{t\lambda},
\]
and this is equivalent to $\varphi_\lambda(e^t) \ll 1/(\lambda \sinh t)$ as one sees by applying the asymptotic $1/\sinh t \asymp e^{-t}(t+1)/t$.

The bound \eqref{SL2C2} reflects the fact that the decay of $\varphi_0(e^t) = t / \sinh(t) \asymp t e^{-t}$ is slower than that of $\varphi_\lambda$ for $\lambda \neq 0$, and moreover that this transition starts to be observed at the point $e^t$ when $\lambda \asymp 1/t$.  Theorem \ref{sph-fn-cx} can be viewed as an extension of this bound to a general group, where the function $(t+1)/(t\lambda +1)$ is replaced by a product of similar expressions over $\Phi_\textnormal{red}^+$.

\subsubsection{Consequences on a general group}

We next derive some simpler corollaries of Theorem \ref{sph-fn-bd}, both because of their intrinsic interest, and to facilitate the comparison between Theorem \ref{sph-fn-bd} and previous bounds for the spherical function in the next section.  First, we make the observation that when $\lambda$ is at distance $\gg 1$ from the singular set, Theorem \ref{sph-fn-bd} simply gives
\be
\label{sphericalcor1}
\varphi_\lambda(e^H) \ll (1 + \| \lambda \|)^a \underset{w \in W}{\max} \, e^{-(\rho + w \Im \lambda)(H) }.
\ee
Recall the definition of $f_\alpha(H,\lambda)$ in \eqref{defn-falpha}. For general $\lambda$, we may apply the bound $f_\alpha(H,\lambda) \le |\alpha(H)| + 1$ to deduce
\be
\label{sphericalcor2}
\varphi_\lambda(e^H) \ll (1 + \| \lambda \|)^a \underset{w \in W}{\max} \, e^{-(\rho + w \Im \lambda)(H) } \prod_{\alpha \in \Phi_\textnormal{red}^+} (|\alpha(H)| + 1).
\ee
Theorem \ref{sph-fn-bd} may be viewed as giving an interpolation between these two bounds, based on how singular $\lambda$ is.  

Finally, we have the following.

\begin{cor}

Let $\eta$ denote the maximum size of the sets $\Phi^+_{\textnormal{red},M}$, where $M$ runs over the maximal standard Levi subgroups of $G$.  We then have
\[
\varphi_\lambda(e^H) \ll (1 + \| \lambda \|)^a (1 + \| H \|)^{ \eta} \min( \| H \| + 1, \| \lambda \|^{-1} + 1)^{|\Phi_\textnormal{red}^+| - \eta} \underset{w \in W}{\max} \, e^{-(\rho + w \Im \lambda)(H) }
\]
\end{cor}

\begin{proof}
We first apply the bound $\alpha(H) \ll \| H \|$ to obtain
\[
f_\alpha(H,w\lambda) \ll \min( \| H \| + 1, |\langle w\lambda, \alpha \rangle|^{-1} + 1),
\]
which gives
\[
\Theta(H,\lambda) \ll \prod_{\alpha \in \Phi_\textnormal{red}^+} \min( \| H \| + 1, |\langle \lambda, \alpha \rangle|^{-1} + 1)
\]
and therefore
\[
\varphi_\lambda(e^H) \ll (1 + \| \lambda \|)^a \underset{w \in W}{\max} \, e^{-(\rho + w \Im \lambda)(H) } \prod_{\alpha \in \Phi_\textnormal{red}^+} \min( \| H \| + 1, |\langle \lambda, \alpha \rangle|^{-1} + 1).
\]
It will therefore be enough to show that
\[
\prod_{\alpha \in \Phi_\textnormal{red}^+} \min( \| H \| + 1, |\langle \lambda, \alpha \rangle|^{-1} + 1) \ll (1 + \| H \|)^{ \eta} \min( \| H \| + 1, \| \lambda \|^{-1} + 1)^{|\Phi_\textnormal{red}^+| - \eta}
\]
for any $\lambda \in \ga^*_\C$.  We are free to multiply $\lambda$ by $i$, and so we may assume that $\| \Re \lambda \| \gg \| \lambda \|$.  We may also assume that $\Re \lambda \in \overline{\ga}^*_+$.  It follows that there is some maximal standard Levi subgroup $M$ such that $|\langle \lambda, \alpha \rangle| \gg \| \lambda \|$ for all $\alpha \in \Phi_\textnormal{red}^+ \smallsetminus \Phi^+_{\textnormal{red},M}$, and therefore that
\[
\min( \| H \| + 1, |\langle \lambda, \alpha \rangle|^{-1} + 1) \ll \min( \| H \| + 1, \| \lambda \|^{-1} + 1)
\]
for $\alpha \in \Phi_\textnormal{red}^+ \smallsetminus \Phi^+_{\textnormal{red},M}$.  This implies
\begin{align*}
\prod_{\alpha \in \Phi_\textnormal{red}^+} \min( \| H \| + 1, |\langle \lambda, \alpha \rangle|^{-1} + 1) & \ll \prod_{\alpha \in \Phi^+_{\textnormal{red},M}} (\| H \| + 1) \prod_{\alpha \in \Phi_\textnormal{red}^+ \smallsetminus \Phi^+_{\textnormal{red},M}} \min( \| H \| + 1, \| \lambda \|^{-1} + 1) \\
& = (1 + \| H \|)^{ |\Phi^+_{\textnormal{red},M}|} \min( \| H \| + 1, \| \lambda \|^{-1} + 1)^{|\Phi_\textnormal{red}^+ \smallsetminus \Phi^+_{\textnormal{red}, M}|} \\
& \le (1 + \| H \|)^\eta \min( \| H \| + 1, \| \lambda \|^{-1} + 1)^{|\Phi_\textnormal{red}^+| - \eta}
\end{align*}
as required.
\end{proof}

\subsubsection{Theorem \ref{sph-fn-cx} and stationary phase}

We now explain why the bound of Theorem \ref{sph-fn-cx} is to be expected, and should be sharp, from the viewpoint of stationary phase applied to the integral \eqref{euclideandef}. The critical set of the phase function appearing in this integral is understood, see \cite[\S 1]{Duistermaat_Kolk_Varadarajan_79}.  In particular, when $Z = H \in \ga$, the integral in \eqref{euclideandef} reduces to one over $K/Z_K(\ga)$, and \cite{Duistermaat_Kolk_Varadarajan_79} shows that when $H$ and $\lambda$ are both regular the critical points occur at $W$ and are non-degenerate.  Moreover, for each $w \in W$, there is a fixed basis for the tangent space $T_w(K/Z_K(\ga))$ in which the Hessian is diagonal with eigenvalues $\alpha(H) \langle \lambda, w\alpha \rangle$, each occurring with multiplicity 2 \cite[Prop. 1.4]{Duistermaat_Kolk_Varadarajan_79}. (The fact that the multiplicities are 2 is due to $G$ being complex.)  By stationary phase, the strongest bound one can give for the contribution of the integral near $w$ to \eqref{euclideandef} is therefore $(1 + |\alpha(H) \langle \lambda, w\alpha \rangle|)^{-1}$, and summing these over $W$ gives the bound of Theorem \ref{sph-fn-cx}.

\subsection{Relation with previous work}
\label{sec:sph-fn-prev}

We now discuss the relation between Theorem \ref{sph-fn-bd} and existing bounds for the spherical function.  Compared with previous results, Theorem \ref{sph-fn-bd} has strong dependence on $H$, but weak dependence on $\lambda$, at least when $\lambda$ is large.  This is convenient for us, on both counts.  It is crucial for the applications in this paper that our bounds have good dependence on $H$.  On the other hand, the fact that the bound of Theorem \ref{sph-fn-bd} is polynomially growing in $\lambda$ does not matter for us, as in our applications there are always other factors present that decay rapidly in $\lambda$.
\medskip

On a general group, the problem of giving bounds for the spherical function that are uniform in $H$ and $\lambda$ is a difficult one, and far from being completely solved.  If we restrict our attention to bounds that hold for all $H$ and $\lambda$, with $\lambda$ varying over $\ga^*$ or $\ga^*_\C$, the only previous results we are aware of are obtained by reducing directly to the case of $\lambda = 0$, and as a result do not capture most of the interaction between $H$ and $\lambda$.

To describe these bounds, it follows from Kostant's convexity theorem that
\[
|\varphi_\lambda(e^H)| \le \varphi_0(e^H) \, \underset{w \in W}{\max} \, e^{-w \Im \lambda(H) }
\]
for any $\lambda \in \ga^*_\C$, see for instance \cite[Prop 4.6.1]{Gangolli_Varadarajan_88}.  Combined with the Harish-Chandra bound for $\varphi_0$ \cite[Theorem 4.6.4]{Gangolli_Varadarajan_88}, which states that
\be
\label{HCbound}
\varphi_0(e^H) \ll (1 + \| H \|)^{ |\Phi_\textnormal{red}^+| } e^{-\rho(H)},
\ee
we obtain
\be
\label{HCbound2}
\varphi_\lambda(e^H) \ll (1 + \| H \|)^{ |\Phi_\textnormal{red}^+| } \underset{w \in W}{\max} \, e^{-(\rho + w \Im \lambda)(H) }.
\ee
The bound \eqref{HCbound} was strengthened by Anker \cite{Anker_87} to
\be
\label{Ankerbd}
\varphi_0(e^H) \asymp \prod_{\alpha \in \Phi_\textnormal{red}^+} (1 + |\alpha(H)|) e^{-\rho(H)},
\ee
which gives the corresponding improvement in \eqref{HCbound2} to
\be
\label{Ankerbd2}
\varphi_\lambda(e^H) \ll \prod_{\alpha \in \Phi_\textnormal{red}^+} (1 + |\alpha(H)|) \underset{w \in W}{\max} \, e^{-(\rho + w \Im \lambda)(H) }.
\ee
We note that the upper bound for $\varphi_0$ provided by Theorem \ref{sph-fn-bd} is the same as the upper bound in \eqref{Ankerbd}, so that the theorem is sharp in this case.  Moreover, \eqref{Ankerbd2} is equivalent to the corollary \eqref{sphericalcor2} of Theorem \ref{sph-fn-bd}, if we restrict $\lambda$ to a bounded set in
\begin{equation}\label{temp-tube}
\ga^*(\kappa)=\{ \lambda \in \ga^*_\C : \| \Im \lambda \| < \kappa \},
\end{equation} 
 the tube of radius $\kappa>0$ about $\ga^*\subset\ga_\C^*$.
\medskip

The paper \cite{Narayanan_Pasquale_Pusti_14} proves an asymptotic for $\varphi_\lambda$ similar to \eqref{Ankerbd} for any $\lambda$. However, the use of the Harish-Chandra expansion in their proof means that their result only holds for $H$ in a set of the form $H_0 + \ga_+$ for a fixed $H_0 \in \ga_+$.  Moreover, their error terms are not uniform in $\lambda$. We have not checked how this asymptotic compares with Theorem \ref{sph-fn-bd}. We also mention Cor. 4.5.5 of \cite{Gangolli_Varadarajan_88}, which implies a uniform estimate for $\varphi_\lambda(e^H)$ when $H \in H_0 + \ga_+$ and $\lambda \in \lambda_0 + \ga^{*}_+$ for fixed $\lambda_0 \in \ga^{*}_+$.
\medskip

Finally, we assume that $G$ is complex, and discuss the strength of Theorem \ref{sph-fn-cx}, as well as its relation to existing bounds for the spherical function in this case.  

Bounds for $\varphi_\lambda$ and $\varphi_\lambda^E$ in the complex case were studied by Barlet--Clerc \cite{Barlet_Clerc_86}, Clerc \cite{Clerc_87}, and Cowling--Nevo \cite{Cowling_nevo_01}, and may be deduced from the results of Duistermaat--Kolk--Varadarajan \cite{Duistermaat_Kolk_Varadarajan_79}.  Theorem \ref{sph-fn-cx} strengthens these results, subject to the condition that $\lambda$ is real.  However, we note that the papers \cite{Barlet_Clerc_86,Clerc_87,Cowling_nevo_01} also obtain results for complex $\lambda$, and \cite{Cowling_nevo_01} also bounds certain derivatives of $\varphi_\lambda(e^H)$ in the $H$-variable.

\subsection{Outline of proof of Theorem \ref{sph-fn-bd}}
\label{sec:sph-fn-outline}

We now give an outline of the proof of Theorem \ref{sph-fn-bd}.  We first discuss the case when $G = \textnormal{SL}_2(\mathbb{R})$, before describing the modifications required in the general case.  As in the case of ${\rm SL}_2(\C)$ discussed above, we shall identify $\ga$ and $\ga^*_\C$ with $\R$ and $\C$ using the basis elements $\alpha^\vee/2$ and $\alpha$ respectively, where $\alpha \in \Phi^+$ is the unique positive root.  Thus we wish to show that there exist $a, \kappa > 0$ such that
\begin{gather}
\label{SL2bd}
    |\varphi_\lambda(e^t)| \ll (1 + |\lambda|)^a e^{-t/2 + |\Im \lambda| t} \min (1 + t, 1 + |\lambda|^{-1})
\end{gather}
for $t \geq 0$ and $|\Im \lambda| < \kappa$.

First, we note that $|\varphi_\lambda(e^t)| \le e^{|\Im \lambda| t} \varphi_0(e^t)$, which follows by taking absolute values in the Harish-Chandra integral formula and using the fact that $H(k e^t) \le t$ for all $k \in {\rm SO}(2)$.  This implies \eqref{SL2bd} in the range where $t < 1$ (or $t < R$ for any fixed $R$).

Now suppose $t \geq 1$. Let $\theta(t, \lambda)$ be defined as
\begin{gather}
    \theta(t, \lambda) := c(\lambda) e^{i \lambda t} + c(-\lambda) e^{-i \lambda t}.
\end{gather}
The asymptotic formula of Gangolli--Varadarajan, recalled in Proposition \ref{prop_hc_expansion}, states that there is $\kappa > 0$ such that
\be
\label{SL2GV}
    |\varphi_{\lambda}(e^t) - e^{-t/2} \theta(t, \lambda)| \ll (1 + t)^a (1 + |\lambda|)^a e^{-(1/2 + \varepsilon) t}
\ee
for $|\Im \lambda| < \kappa$, where $a, \varepsilon > 0$ are constants (see \cite[\S 5.1]{Gangolli_Varadarajan_88} for an illustrative discussion of the rank one case).  We may take these to be the constants $a$ and $\kappa$ in our bound \eqref{SL2bd}.  As the right hand side of \eqref{SL2GV} is less than the right hand side of \eqref{SL2bd}, it suffices to show that
\begin{gather}
\label{thetabd1}
    |\theta(t, \lambda)| \ll (1 + |\lambda|)^a e^{|\Im \lambda| t} \min (1 + t, 1 + |\lambda|^{-1}),
\end{gather}
for $t \geq 1$ and $|\Im \lambda| < \kappa$. In fact, we will show the stronger bound
\be
\label{thetabd2}
|\theta(t, \lambda)| \ll e^{|\Im \lambda| t} \min (1 + t, 1 + |\lambda|^{-1}).
\ee

In the case of ${\rm SL}_2(\R)$, the explicit nature of $\theta(t, \lambda)$ means that this can be exhibited quite directly, and we do this first before describing the argument that must be used in the case of a general group.  We first suppose that $|\lambda| \geq 1$.  If we assume that $\kappa < 1$, we have the bound $|c(\lambda)| \ll_\kappa 1$ for $|\Im \lambda| < \kappa$. We thus straight away get
\begin{gather*}
    |\theta(t, \lambda)| \ll e^{|\Im \lambda| t} \ll e^{|\Im \lambda| t} \min(1 + |\lambda|^{-1}, 1 + t),
\end{gather*}
as required.

We next suppose that $|\lambda| < 1$.  The $c$-function is a meromorphic function with a simple pole at $\lambda = 0$. Despite this singularity, $\theta(t, \lambda)$ can be extended holomorphically to $\lambda = 0$.  If we write $c(\lambda) = a_{-1}\lambda^{-1} + g(\lambda)$ where $g(\lambda)$ is an entire function, we have
\begin{gather*}
    c(\lambda) e^{i \lambda t} + c(-\lambda) e^{-i \lambda t} = a_{-1} \frac{e^{i \lambda t} - e^{-i \lambda t}}{\lambda} + g(\lambda)e^{i \lambda t} + g(-\lambda) e^{-i \lambda t}.
\end{gather*}
Because $|\lambda| < 1$, the last two terms above can be bounded by $C e^{|\Im \lambda| t}$ for some $C$. The first term can be bounded in two different ways, as
\[
\Big| \frac{e^{i \lambda t} - e^{-i \lambda t}}{\lambda} \Big| \leq 2 |\lambda|^{-1} e^{|\Im \lambda| t} \quad \text{and} \quad \Big| \frac{e^{i \lambda t} - e^{-i \lambda t}}{\lambda} \Big| \leq 2 t e^{|\Im \lambda| t},
\]
where the first bound is elementary and the second follows from applying the mean value theorem to $f(x) = e^{i \lambda x} - e^{-i \lambda x}$ on the interval $[0,t]$.  We thus get the bound
\begin{gather*}
    |\theta(t, \lambda)| \ll e^{|\Im \lambda| t} \min(1 + |\lambda|^{-1}, 1 + t),
\end{gather*}
as required.  This completes the proof of Theorem \ref{sph-fn-bd} in this case.
\medskip

For later purposes, it will be convenient to use a different method, based on the maximum modulus principle, to establish \eqref{thetabd2}. To describe it, we first observe that if $\kappa < 1$, we have
\[
c(\lambda) \ll_\kappa 1 + | \lambda |^{-1}
\]
when $| \Im \lambda | < \kappa$, as follows from \eqref{eqn_c_lambda_big}. This implies that
\be
\label{thetabd3}
\theta(t, \lambda) \ll (1 + | \lambda |^{-1}) e^{ | \Im \lambda | t},
\ee
which establishes \eqref{thetabd2} when $| \lambda | \ge t^{-1}$.  We therefore assume that $| \lambda | < t^{-1}$. For these $\lambda$, we shall prove \eqref{thetabd2} by applying the maximum modulus principle, on the disk of radius $2 t^{-1}$ centered at $\lambda$. This gives
\be
\label{thetamax}
\theta(t, \lambda) \leq \max_{\alpha} \theta(t, \lambda + 2t^{-1} e^{i\alpha}).
\ee
As discussed above, we are free to assume that $t \ge R$ for any fixed $R$, and we choose $R$ large enough that $3 R^{-1} < \kappa$. This implies that the contour $\lambda + 2t^{-1} e^{i\alpha}$ is contained in the strip $| \Im z | < \kappa$, and lies outside the disk $|z| < t^{-1}$. We may therefore combine \eqref{thetabd3} and \eqref{thetamax} to obtain
\begin{align*}
\theta(t, \lambda) & \ll \max_{\alpha} (1 + | \lambda + 2t^{-1} e^{i\alpha} |^{-1}) e^{ | \Im (\lambda + 2t^{-1} e^{i\alpha}) | t} \\
& \le 2e^{ | \Im \lambda | t} \max_{\alpha} (1 + | \lambda + 2t^{-1} e^{i\alpha} |^{-1}),
\end{align*}
and moreover, because $| \lambda + 2t^{-1} e^{i\alpha} | \ge t^{-1}$, this gives
\[
\theta(t, \lambda) \ll (1 + t)e^{ | \Im \lambda | t}
\]
as required.
\medskip

We finish this section by describing the modifications that must be made to this approach in the case of a general group, and we now let $G$ be general again.  We shall apply the asymptotic of Gangolli--Varadarajan along maximal Levi subgroups of $G$.  To do this, for each maximal Levi $L$ we define a cone $\mathcal{C}_L$ in $\ga$ that is adapted to the centralizer of $L$ in $\ga$, and such that $\overline{\ga}_+ \subset B_R(0) \cup \bigcup_L \mathcal{C}_L$.  As before, the case of $H \in B_R(0)$ may be handled easily, so we fix $L$ and bound $\varphi_\lambda(e^H)$ for $H \in \mathcal{C}_L$.  If we define
\[
\theta_L(H, \lambda) = \sum_{w \in W_L \backslash W} c^L(w \lambda) \varphi^L_{w \lambda}(e^H),
\]
then Proposition \ref{prop_hc_expansion} reduces the problem to proving that
\be
\label{thetaoutline}
\theta_L(H, \lambda) \ll (1 + \| \lambda \|)^a \Theta(H, \lambda) \underset{w \in W}{\max} \, e^{-(\rho_L + w \Im \lambda)(H)}.
\ee
For this purpose, recall from \eqref{temp-tube} the notation $\mathfrak{a}^*(\kappa)$, where $\kappa>0$. For a constant $C>0$, we define
\begin{equation}\label{defn:a-reg-kappa}
\mathfrak{a}_{\rm reg}^* (\kappa,C) =\{\lambda\in\mathfrak{a}^*(\kappa): |\langle\alpha ,\lambda \rangle|\ge C\;\forall\;\alpha \in \Phi_{\textnormal{red}} \}
\end{equation}
and $\mathfrak{a}_{\rm sing}^* (\kappa,C) = \mathfrak{a}^* (\kappa) \smallsetminus \mathfrak{a}_{\rm reg}^* (\kappa,C)$. We may show \eqref{thetaoutline} when $\lambda \notin \mathfrak{a}_{\rm reg}^* (\kappa, \| H \|^{-1})$ by applying Theorem \ref{sph-fn-bd} inductively to $L$, together with standard bounds for $c^L$.  We then deduce \eqref{thetaoutline} for $\lambda \in \mathfrak{a}_{\rm reg}^* (\kappa, \| H \|^{-1})$ from this using the maximum principle.  We do this by choosing a contour around $\lambda$ of the form $\lambda + \nu e^{i\alpha}$, for some $\nu \in \ga^*_\C$ depending on $\lambda$, such that
\[
\underset{\alpha}{\max} \, \Theta(H, \lambda + \nu e^{i\alpha}) \, \underset{w \in W}{\max} \, e^{-(\rho_L + w \Im (\lambda + \nu e^{i\alpha}))(H)} \ll \Theta(H, \lambda) \, \underset{w \in W}{\max} \, e^{-(\rho_L + w \Im \lambda)(H)}.
\]
We require $\nu$ to satisfy $\| \nu \| \ll \| H \|^{-1}$, so that the exponential terms are roughly constant, and also that $\lambda + \nu e^{i\alpha}$ is at distance $\gg \| H \|^{-1}$ from all root hyperplanes.  In Lemma \ref{lemma:diskchoice}, we show that we may do this by choosing $\nu = C \| H \|^{-1} \rho$, where $C \asymp 1$ depends on $\lambda$.  We take $\rho$ as the direction of $\nu$ because it is not orthogonal to any root, and we need to choose $C$ depending on $\lambda$ because we now have multiple root hyperplanes to avoid, and a choice of $C$ that avoids one hyperplane may end up intersecting another.

\subsection{Preliminaries}
\label{sec:sph-fn-prelim}

We recall the notation $\mathfrak{a}^* (\kappa)$ and $\mathfrak{a}_{\rm reg}^* (\kappa,C)$ from \eqref{temp-tube} and \eqref{defn:a-reg-kappa}.

\begin{lemma}
\label{lemma:diskchoice}
    Let $0 < \kappa' < \kappa$ be given. Then there exists $\sigma > 1$ depending only on $G$, and $D > 0$ depending on $\kappa$, $\kappa'$, and $G$ such that the following holds. For each $s < D$ and $\lambda \in \mathfrak{a}^*(\kappa')$, there exists $C = C(s,\lambda) \in [1, \sigma]$ such that we have $\lambda + z \rho \in \mathfrak{a}^*_{\textnormal{reg}}(\kappa, s)$ for $|z| = C s$.  
\end{lemma}
\begin{proof}
    Let $\tau$ be a real number satisfying $\tau | \langle \rho, \alpha \rangle | > 2$ and $\tau > |\langle \rho, \alpha \rangle| + 1$ for all $\alpha \in \Phi_{\textnormal{red}}$ (so in particular $\tau > 1$). We claim that we can take $\sigma = \tau^{2|\Phi_{\textnormal{red}}^+| + 2}$. 
    
    To prove this, let $0 < \kappa' < \kappa$, and let $\lambda \in \mathfrak{a}^*(\kappa')$ be given. It is clear that if $C \in [1, \sigma]$ and $|z| = C s$ for $s$ sufficiently small depending on $\kappa$ and $\kappa'$, then $\lambda + z \rho \in \mathfrak{a}^*(\kappa)$. We must therefore show that there is a choice of $C \in [1, \sigma]$ such that $|\langle \lambda + z \rho, \alpha \rangle| \geq s$ for all $|z| = C s$ and $\alpha \in \Phi_{\textnormal{red}}$. 

    Consider the disjoint intervals $(\tau^{2k} s, \tau^{2k+2} s)$ for $0 \leq k \leq |\Phi_{\textnormal{red}}^+|$. Because there are at most $|\Phi_{\textnormal{red}}^+|$ numbers of the form $|\langle \lambda, \alpha \rangle|$, one of these intervals does not contain $|\langle \lambda, \alpha \rangle|$ for any $\alpha \in \Phi_{\textnormal{red}}$. We let $C = \tau^{2k + 1}$ where $k$ is index of one of the empty intervals. For each $\alpha$ our choice of $k$ implies that either $|\langle \lambda, \alpha \rangle| \leq \tau^{2k} s$ or $|\langle \lambda, \alpha \rangle| \geq \tau^{2k+2} s$. In the first case we have
\begin{align*}
| \langle \lambda + z \rho, \alpha \rangle| & \ge |z| | \langle \rho, \alpha \rangle| - | \langle \lambda, \alpha \rangle| \\
& \ge \tau^{2k+1} s | \langle \rho, \alpha \rangle| - \tau^{2k} s\\
& = \tau^{2k}(\tau | \langle \rho, \alpha \rangle| - 1) s \\
& > s,
\end{align*}
while in the second case a similar computation gives
\[
| \langle \lambda + z \rho, \alpha \rangle| \ge \tau^{2k+1} s (\tau - | \langle \rho, \alpha \rangle|) > s,
\]
as required.
\end{proof}

\subsection{Proof of Theorem \ref{sph-fn-bd}: Reduction step}

We will prove Theorem \ref{sph-fn-bd} by induction on the semisimple real rank of $G$.  The base case, when the rank is zero, is trivial, since the spherical functions in this case are simply complex exponentials. We may therefore assume that Theorem \ref{sph-fn-bd} holds for all proper Levi subgroups of $G$.

Let $B_R(0)$ denote the ball of radius $R$ around the origin in $\ga$.  As in the rank one case discussed above, Theorem \ref{sph-fn-bd} holds when $H \in B_R(0)$, for any fixed $R$; this follows from the bound $| \varphi_\lambda(e^H)| \le \underset{w \in W}{\max} \, e^{-w \Im \lambda(H) } \varphi_0(e^H)$, which in turn follows from e.g. Prop. 4.6.1 of \cite{Gangolli_Varadarajan_88}.  For the other $H$, we shall cover $\overline{\ga}_+ \smallsetminus B_R(0)$ with cones adapted to the centers of the maximal Levi subgroups of $G$.

Let $\Sigma = \Delta \smallsetminus \{\beta \}$ for some $\beta\in \Delta$.  Let $L_\Sigma$ be the Levi associated to $\Sigma$, which is the centralizer of the subspace $\ga_\Sigma = \{ H \in \ga : \alpha(H) = 0, \alpha \in \Sigma \} = \mathbb{R} \beta^\vee$.  Let $\Phi_\Sigma$ be the root system of $L_\Sigma$, which is given by $\Phi_\Sigma = \Phi \cap \sum_{\alpha \in \Sigma} \R \alpha$.  We let $\mathcal{C}_\Sigma \subset \ga$ be the cone consisting of those $H$ satisfying
\be
\label{CLdef}
\alpha(H) > c \| H \| \text{ for } \alpha \in \Phi^+ \smallsetminus \Phi_\Sigma^+, \quad \| H \| > R.
\ee
Roughly speaking, this is a cone around $\ga_\Sigma \cap \overline{\ga}_+$.  If $c$ is chosen small enough, the $\mathcal{C}_\Sigma$ for various $\Sigma$, together with $B_R(0)$, cover $\overline{\ga}_+$. Thus, from now on assume that we have fixed such a sufficiently small $c$.

We fix a $\Sigma$, and denote $L_\Sigma$ simply by $L$. We prefer to write $\mathcal{C}_L$ for $\mathcal{C}_\Sigma$ and $\Phi_L$ for $\Phi_\Sigma$.  We shall prove Theorem \ref{sph-fn-bd} for $H \in \overline{\ga}_+ \cap \mathcal{C}_{L}$, using our hypothesis that it holds for $L$, together with the following two ingredients.  The first is the asymptotic expansion of $\varphi_\lambda$ along $L$ given in Theorem 5.9.3 of \cite{Gangolli_Varadarajan_88}.

\begin{prop}\label{lem:complexGV}
Let $L$ be a standard Levi subgroup of $G$ and put
\[
\theta_L(H, \lambda)= \sum_{w \in W_L \backslash W} c^L(w \lambda) \varphi^L_{w \lambda}(e^H).
\]
Then there is $a,\kappa >0$, depending only on the constant $c$ appearing in \eqref{CLdef}, such that for all $\lambda \in \ga^*(\kappa)$ and $H \in \mathcal{C}_L$, we have 
\be
\label{GVasymp}
\varphi_\lambda(e^H)= e^{-\rho^L(H) } \theta_L(H, \lambda)+O((1 + \| \lambda \|)^a e^{-\rho(H)  - \kappa \| H \|} ).
\ee
\end{prop}

\begin{proof}
We note that $\theta_L(H, \lambda)$ is initially defined only for $\lambda$ regular, but it is shown in Prop. 5.8.2 of \cite{Gangolli_Varadarajan_88} that it extends holomorphically to $\ga^*(\kappa)$ if $\kappa$ is sufficiently small.  (This result is stated for a vector-valued function denoted $\Theta(\lambda,m)$ in \cite{Gangolli_Varadarajan_88}, but $\theta_L(H, \lambda)$ is equal to the first coordinate of $\Theta(\lambda,e^H)$.)  This may also be shown in a more elementary way by proving that the poles of $c^L$ along the root hyperplanes cancel in the sum.

For any $\zeta > 0$, Gangolli--Varadarajan \cite{Gangolli_Varadarajan_88} (5.9.4) define a set $A^+(H_0:\zeta) \subset A$, which in our case satisfies
\[
A^+(H_0:\zeta) = \exp( \{ H \in \overline{\ga}_+ : \alpha(H) > \zeta \| H \| \text{ for } \alpha \in \Phi^+ \smallsetminus \Phi^+_L \} ).
\]
If we choose $\zeta = c$, then we have $\exp( \mathcal{C}_L \cap \overline{\ga}_+) \subset A^+(H_0:\zeta)$.  In Theorem 5.9.3(b) of \cite{Gangolli_Varadarajan_88}, the authors prove an asymptotic for $\varphi_\lambda(e^H)$ for all $e^H \in A^+(H_0:\zeta)$ that is equivalent to \eqref{GVasymp}, which implies the proposition.  Note that we take the differential operator $b \in U(\g)$ appearing in that theorem to be trivial, in which case the operator $\gamma_0(b)$ defined in (5.9.2) of \cite{Gangolli_Varadarajan_88} is also trivial.
\end{proof}

The second ingredient is the following estimate for the main term $\theta_L$ in the asymptotic expansion \eqref{GVasymp}.

\begin{prop}\label{prop:Lmainbd}
For a standard Levi subgroup $L$ write
\begin{equation}\label{def:ThetaL}
\Theta_L(H, \lambda) = \sum_{w \in W_L} \prod_{ \alpha \in \Phi_{\textnormal{red},L}^+} f_\alpha(H,w\lambda)
\end{equation}
and write $\Theta_G=\Theta$ to accord with \eqref{Thetadef}. Assume Theorem \ref{sph-fn-bd} holds for $L$ with corresponding constants $a, \kappa$, so that
\[
\varphi^L_\lambda(e^H) \ll (1 + \| \lambda \|)^{a} \Theta_L(H, \lambda) \underset{w \in W_L}{\max} \, e^{-(\rho_L + w \Im \lambda)(H)}
\]
for $\lambda \in \ga^*(\kappa)$ and $H \in \overline{\ga}_{L,+}$, where $\overline{\ga}_{L,+}$ is defined in \eqref{aLdef}.  Then there are $a', \kappa' > 0$ such that for all $H \in \overline{\ga}_+\cap \mathcal{C}_L$ and all $\lambda \in \ga^*(\kappa')$, we have
\[
\theta_L(H, \lambda) \ll (1 + \| \lambda \|)^{a'} \Theta(H, \lambda) \underset{w \in W}{\max} \, e^{-(\rho_L + w \Im \lambda)(H)}.
\]
\end{prop}

We shall prove Proposition \ref{prop:Lmainbd} in the next paragraph.  Let us show how to deduce Theorem \ref{sph-fn-bd} for $H \in \overline{\ga}_+ \cap \mathcal{C}_{L}$ from these results.  First, we apply the asymptotic expansion of $\varphi_\lambda$ along $L$ given in Proposition \ref{lem:complexGV}. The error term there is dominated by the majorant of Theorem \ref{sph-fn-bd}, and Proposition \ref{prop:Lmainbd} shows that the same is true of the main term, which completes the proof.

\subsection{Proof of Proposition \ref{prop:Lmainbd}}
\label{sec:Lmainbd}

Recall the definition of $\mathfrak{a}_{\rm reg}^*(\kappa)$ and $\mathfrak{a}_{\rm reg}^* (\kappa,C)$ from \eqref{temp-tube} and \eqref{defn:a-reg-kappa}. We shall show that the case of $\lambda \in \mathfrak{a}_{\rm reg}^* (\kappa,c^{-1} \|H\|^{-1})$ can be treated in a straightforward way, using the hypothesis of the proposition. We shall then deduce the bound for all $\lambda \in \mathfrak{a}^*(\kappa')$ for an appropriate choice of $\kappa'$ from the case of $\lambda \in \mathfrak{a}_{\rm reg}^*(\kappa,c^{-1}\|H\|^{-1})$ using the maximum modulus principle, aided by Lemma \ref{lemma:diskchoice}.

For any $\lambda\in\mathfrak{a}^*(\kappa)$, we have the bound
\begin{equation}\label{CM-bound}
c^L(\lambda)\ll \prod_{\alpha\in\Phi^+_{\textnormal{red}}\smallsetminus\Phi_{\textnormal{red},L}^+} (1+|\langle \lambda,\alpha\rangle|^{-1}),
\end{equation}
since for $s$ small we have $c_{\alpha}(s) \ll 1 + |s|^{-1}$ because of the simple pole at $s = 0$, and for $s$ large with imaginary part bounded by $\kappa$, we have $c_{\alpha}(s) \ll_\kappa |s|^{-m_{\alpha}/2 - m_{2 \alpha}/2} \ll 1 + |s|^{-1}$. Using \eqref{CM-bound} and the assumption of the proposition, we get
\begin{align*}
\theta_L(H, \lambda)&\ll (1 +\|\lambda\|)^a \Psi_L(H,\lambda) \underset{w \in W}{\max} \, e^{-(\rho_L + w \Im \lambda)(H)}
\end{align*}
for any $\lambda\in\mathfrak{a}(\kappa)^*$, where
\begin{align*}
\Psi_L(H,\lambda)=\sum_{w\in W_L\backslash W}\Theta_L(H, w\lambda)\prod_{\alpha\in\Phi_{\textnormal{red}}^+\smallsetminus\Phi_{\textnormal{red},L}^+} (1+|\langle w\lambda,\alpha\rangle|^{-1}).
\end{align*}
We claim that, when $\lambda\in\mathfrak{a}_{\rm reg}^* (\kappa, c^{-1} \|H\|^{-1})$, we have $\Psi_L(H,\lambda)\ll \Theta(H,\lambda)$. To see this, first note that, by the $W$-invariance of $\langle\cdot,\cdot\rangle$ and the $W_L$-invariance of $\Phi_{\textnormal{red}}^+\smallsetminus\Phi_{\textnormal{red},L}^+$ (see Proposition \ref{lemma:WL-stable}),  we have 
\[
\prod_{\alpha \in \Phi_{\textnormal{red}}^+\smallsetminus\Phi_{\textnormal{red},L}^+} (1+|\langle w\lambda,\alpha\rangle|^{-1})=\prod_{\alpha \in \Phi_{\textnormal{red}}^+\smallsetminus\Phi_{\textnormal{red},L}^+} (1+|\langle sw\lambda, s\alpha\rangle|^{-1})=\prod_{\alpha \in \Phi_{\textnormal{red}}^+\smallsetminus\Phi_{\textnormal{red},L}^+} (1+|\langle sw\lambda, \alpha\rangle|^{-1}).
\]
for any $s\in W_L$. Thus, recalling the definition of $\Theta_L(H, w\lambda)$ in \eqref{def:ThetaL}, we find
\begin{align*}
\Psi_L(H,\lambda)&=\sum_{w\in W_L\backslash W}\sum_{s\in W_L}\prod_{\alpha\in\Phi_{\textnormal{red},L}^+}f_\alpha(H,sw\lambda)\prod_{\alpha \in \Phi_{\textnormal{red}}^+\smallsetminus\Phi_{\textnormal{red},L}^+} (1+|\langle w\lambda,\alpha\rangle|^{-1})\\
&=\sum_{w\in W_L\backslash W}\sum_{s\in W_L}\prod_{\alpha \in \Phi_{\textnormal{red},L}^+}f_\alpha(H,sw\lambda)\prod_{\alpha \in \Phi_{\textnormal{red}}^+\smallsetminus\Phi_{\textnormal{red},L}^+}(1+|\langle sw\lambda,\alpha\rangle|^{-1})\
\end{align*}
Note that for $H\in\mathcal{C}_L$ and $\lambda\in\mathfrak{a}_{\rm reg}^*(\kappa, c^{-1} \|H\|^{-1})$, we have $|\langle \alpha,\lambda\rangle|^{-1} \leq c\|H\| < \alpha(H)$ for all $\alpha \in \Phi_{\textnormal{red}}^+\smallsetminus\Phi_{\textnormal{red},L}^+$, so that
$(1+|\langle sw\lambda, \alpha\rangle|^{-1})\ll f_\alpha(H,sw\lambda)$. We put together the above estimates to obtain
\[
\Psi_L(H,\lambda)\ll\sum_{w\in W_L\backslash W}\sum_{s\in W_L}\prod_{\alpha\in\Phi_{\textnormal{red}}^+}f_\alpha(H,sw\lambda)=\sum_{s\in W}\prod_{\alpha\in\Phi_{\textnormal{red}}^+}f_\alpha(H,s\lambda)
=\Theta(H,\lambda).
\]
The above bounds then establish the case of $\lambda\in\mathfrak{a}_{\rm reg}^*(\kappa; c^{-1} \|H\|^{-1})$.

We now let $\kappa'\in (0,\kappa)$ and take $\lambda \in \mathfrak{a}^* (\kappa')$.  We apply Lemma \ref{lemma:diskchoice} to $\lambda$ and $s = c^{-1} \| H \|^{-1}$, and let $C = C(H,\lambda) \in [1, \sigma]$ be the constant produced. Note that to apply Lemma \ref{lemma:diskchoice}, we may have had to make $R$ bigger which we are free to do. Applying the maximum modulus principle, together with the already established bound in the case of $\lambda \in \mathfrak{a}_{\rm reg}^* (\kappa,c^{-1}\|H\|^{-1})$, gives
\begin{align*}
\theta_L(H, \lambda) & \le \underset{ |z| = C c^{-1} \| H \|^{-1}}{\max} \theta_L(H, \lambda + z \rho) \\
& \ll \underset{ |z| = C c^{-1} \| H \|^{-1}}{\max} (1 + \| \lambda + z \rho \|)^a \Theta(H, \lambda + z \rho) \underset{w \in W}{\max} \, e^{-(\rho_L + w \Im (\lambda + z\rho))(H)} \\
& \ll (1 + \| \lambda \|)^a \underset{w \in W}{\max} \, e^{-(\rho_L + w \Im (\lambda + z\rho))(H)} \underset{ |z| = C c^{-1} \| H \|^{-1}}{\max} \Theta(H, \lambda + z \rho).
\end{align*}
Note that in the above chain of inequalities it was crucial that, even though $C(\lambda)$ varies with $\lambda$, $C(\lambda)$ always lies in a bounded range. 

It therefore remains to show that
\[
\underset{ |z| = C c^{-1} \| H \|^{-1}}{\max} \Theta(H, \lambda + z \rho) \ll \Theta(H, \lambda),
\]
which in turn follows from
\be
\label{minima}
\min( |\alpha(H)| +1, | \langle \lambda, \beta \rangle |^{-1} +1) \asymp \min( |\alpha(H)| +1, | \langle \lambda + z\rho, \beta \rangle |^{-1} +1)
\ee
for $|z| = C c^{-1} \| H \|^{-1}$ and $\alpha, \beta \in \Phi_{\textnormal{red}}$. In the case when $| \langle \lambda, \beta \rangle | < c^{-1} \| H \|^{-1}$, we have $| \langle \lambda + z\rho, \beta \rangle | \ll \| H \|^{-1}$, so that
\[
| \langle \lambda, \beta \rangle |^{-1}, | \langle \lambda + z\rho, \beta \rangle |^{-1} \gg \| H \| \gg |\alpha(H)|.
\]
This implies that both minima are $\asymp |\alpha(H)| + 1$, as required.  When $| \langle \lambda, \beta \rangle | \ge c^{-1} \| H \|^{-1}$, we have $| \langle \lambda, \beta \rangle | \asymp | \langle \lambda + z\rho, \beta \rangle |$, which again implies \eqref{minima}.

This completes the proof of Proposition \ref{prop:Lmainbd} and hence of Theorem \ref{sph-fn-bd}.

\subsection{Proof of Theorem \ref{sph-fn-cx}}
\label{sec:cx-pf}

We now deduce Theorem \ref{sph-fn-cx} from Theorem \ref{sph-fn-bd}.  We will do this by passing between the bounds \eqref{cxeuclideanbd} and \eqref{cxbound} for $\varphi_\lambda^E$ and $\varphi_\lambda$, and so for the convenience of the reader we recall that these bounds state that
\be
\label{cxeuclideanbd2}
\varphi_\lambda^E(H) \ll \sum_{w \in W} \prod_{\alpha \in \Phi^+} (1 + |\alpha(H) \langle w\lambda, \alpha \rangle|)^{-1},
\ee
and
\be
\label{cxbound2}
\varphi_\lambda(e^H) \ll e^{-\rho(H)} \prod_{\alpha \in \Phi^+} (|\alpha(H)|+1) \sum_{w \in W} \prod_{\alpha \in \Phi^+} (1 + |\alpha(H) \langle w\lambda, \alpha \rangle|)^{-1},
\ee
for $H \in \overline{\ga}^+$ and $\lambda \in \ga^*$.  To establish these, Lemma \ref{sph-bd-equiv} below shows that Theorem \ref{sph-fn-bd} implies \eqref{cxbound2} when $\| \lambda \| < 1$, which in turn implies \eqref{cxeuclideanbd2} when $\| \lambda \| < 1$.  We then use the fact that $\varphi_{r\lambda}^E(H/r) = \varphi_\lambda^E(H)$ for any $r > 0$, which follows from the bilinearity of the inner product in the phase function in \eqref{euclideandef}, to deduce \eqref{cxeuclideanbd2} for all $\lambda$. This implies \eqref{cxbound2} for all $\lambda$, and completes the proof.

\begin{lemma}
\label{sph-bd-equiv}
Let $R > 0$ be given.  In the range when $\lambda \in \ga^*$ satisfies $\| \lambda \| < R$, the bounds for $\varphi_\lambda$ given by Theorems \ref{sph-fn-bd} and \ref{sph-fn-cx} are equivalent up to a constant factor depending on $R$.
\end{lemma}

\begin{proof}
If we define
\[
g_\alpha(H,w\lambda) = \frac{|\alpha(H)|+1}{|\alpha(H) \langle w\lambda, \alpha \rangle| + 1}
\]
for $H \in \ga$, $\lambda \in \ga^*_\C$, and $w \in W$, then the bound of Theorem \ref{sph-fn-cx} reads
\[
\varphi_\lambda(e^H) \ll e^{-\rho(H)} \sum_{w \in W} \prod_{\alpha \in \Phi^+} g_\alpha(H,w\lambda).
\]
It therefore suffices to prove that $f_\alpha(H,w\lambda) \asymp_R g_\alpha(H,w\lambda)$ for $\| \lambda \| < R$, as this will imply that
\[
\Theta(H,\lambda) = \sum_{w \in W} \prod_{\alpha \in \Phi^+} f_\alpha(H,w\lambda) \asymp_R \sum_{w \in W} \prod_{\alpha \in \Phi^+} g_\alpha(H,w\lambda)
\]
as required.  To establish $f_\alpha \asymp g_\alpha$, we consider the cases $|\alpha(H) \langle w\lambda, \alpha \rangle| < 1$ and $|\alpha(H) \langle w\lambda, \alpha \rangle| \ge 1$ separately.  When $|\alpha(H) \langle w\lambda, \alpha \rangle| < 1$, we have
\[
f_\alpha(H,w\lambda) = |\alpha(H)| + 1 \asymp g_\alpha(H,w\lambda),
\]
as required.  When $|\alpha(H) \langle w\lambda, \alpha \rangle| \ge 1$, we have
\[
f_\alpha(H,w\lambda) = |\langle w\lambda, \alpha \rangle|^{-1} + 1
\]
and
\[
g_\alpha(H,w\lambda) \asymp \frac{|\alpha(H)|+1}{|\alpha(H) \langle w\lambda, \alpha \rangle|} = |\langle w\lambda, \alpha \rangle|^{-1} + |\alpha(H) \langle w\lambda, \alpha \rangle|^{-1}.
\]
The result follows in this case because $|\langle w\lambda, \alpha \rangle|^{-1} \gg_R 1$ and $|\alpha(H) \langle w\lambda, \alpha \rangle|^{-1} \le 1$.
\end{proof}

\appendix
\section{Sage code for $E_7$ computation}\label{sage}
\begin{lstlisting}[language=Python, caption={Semi-dense root subsystem analysis for $E_7$}]
from copy import deepcopy

# Generate the roots of the form e_i - e_j.
roots = []
for i in range(0, 8):
    for j in range (i+1, 8):
        dummy = [0, 0, 0, 0, 0, 0, 0, 0]
        dummy[i] = 1
        dummy[j] = -1
        dummy2 = [0, 0, 0, 0, 0, 0, 0, 0]
        dummy2[i] = -1
        dummy2[j] = 1
        roots.append (tuple (dummy))
        roots.append (tuple (dummy2))
        

# Generate the roots which are a permutation of 1/2*(1, 1, 1, 1, -1, -1, -1, -1). Because we are ultimately only concerned with linear (in)dependence, we can ignore the factor of 1/2.
for i in range (0, 8):
    for j in range (i+1, 8):
        for k in range (j+1, 8):
            for l in range (k+1, 8):
                dummy = [1, 1, 1, 1, 1, 1, 1, 1]
                dummy[i] = -1
                dummy[j] = -1
                dummy[k] = -1
                dummy[l] = -1
                roots.append (tuple (dummy))
                

# "base" is the standard base (except that the last root has been scaled by 2 which has no effect on the subsequent analysis).
base = [(0, -1, 1, 0, 0, 0, 0, 0), (0, 0, -1, 1, 0, 0, 0, 0), (0, 0, 0, -1, 1, 0, 0, 0), (0, 0, 0, 0, -1, 1, 0, 0), (0, 0, 0, 0, 0, -1, 1, 0), (0, 0, 0, 0, 0, 0, -1, 1), (1, 1, 1, 1, -1, -1, -1, -1)]


# The argument "indices" is a list of indices which is then used to pick out the standard root subsystem spanned by the corresponding elements in the base. The roots in the span are exactly those which are linearly dependent with the specified subset of the base.
def subroot (indices):
    of_interest = []
    for entry in indices:
        of_interest.append (base[entry])
    
    to_return = []
    for entry in roots:
        if linearly_dependent (of_interest, entry) == True:
            to_return.append (entry)
    
    return to_return
    

# This detects whether or not a list of vectors is linearly dependent.
def linearly_dependent (first, second):
    new_first = deepcopy(first)
    new_first.append (second)
    my_mat = matrix (new_first)
    prod = my_mat*my_mat.transpose()
    if prod.determinant() == 0:
        return True
    return False
    

# This generates the list of all elements in the W-orbit of e_1 + e_8 which are all of the form e_i + e_j, or some permutation of 1/2*(1, 1, 1, 1, 1, 1, -1, -1). Again, we can ignore the factor of 1/2.
normals = []
for i in range (0, 8):
    for j in range (i+1, 8):
        dummy = [0, 0, 0, 0, 0, 0, 0, 0]
        dummy[i] = 1
        dummy[j] = 1
        dummy2 = [1, 1, 1, 1, 1, 1, 1, 1]
        dummy2[i] = -1
        dummy2[j] = -1
        normals.append (dummy)
        normals.append (dummy2)
        

# This simply takes the dot product of two vectors.
def dot (v1, v2):
    total = 0
    for i in range (0, len(v1)):
        total += v1[i]*v2[i]
    return total


# The outer for loop iterates through all of the standard root subsystem. The inner for loop ("for normal in normals") computes the intersection of the specified standard root subsystem with each root system of the form w.Phi_0 with w in the Weyl group and Phi_0 the standard E6 root subsystem.
for i in range (0, 2**7):
	# We can enumerate all subsets of the base of simple roots by numbers between 0 and 2^7 expressed in binary. We call the associated list of indices "my_indices".
    pre_indices = Integer(i).digits(2)
    my_indices = []
    for j in range (0, len(pre_indices)):
        if pre_indices[j] == 1:
            my_indices.append (j)
    print (my_indices)
    
    # "Psi" is the root subsystem associated to indices
    Psi = subroot (my_indices)
    for normal in normals:
    	# Elements in "intersected" correspond to elements in Psi \cap w.Phi_0, where w.Phi_0 is the root system orthogonal to "normal". 
        intersected = []
        for root in Psi:
            if dot (normal, root) == 0:
                intersected.append (root)
        # We check whether or not the semi-dense root subsystem inequality is satisfied.
        ineq = 2*len(intersected) + 2*len (my_indices) - len (Psi)
        if ineq < 0:
            print (entry)
            print (intersected)
            print (my_indices)
            print (Psi)

# The program did not print anything implying that the inequality was never violated.
\end{lstlisting}

\normalem
\printbibliography

\end{document}